\documentclass[11pt,twoside]{preprint}
\usepackage[a4paper,innermargin=1in,outermargin=1in,
bottom=1.5in,marginparwidth=1in,marginparsep
=3mm]{geometry}
\usepackage{amsmath,amsthm,amssymb,enumerate}
\usepackage{hyperref}
\usepackage{mathrsfs}
\usepackage{breakurl}
\usepackage[nameinlink,capitalize]{cleveref}

\usepackage[osf]{newtxtext}
\usepackage[varqu,varl]{inconsolata}
\usepackage[cal=boondoxo]{mathalfa}
\usepackage[bigdelims,vvarbb,cmintegrals]{newtxmath}

\usepackage[dvipsnames]{xcolor}
\usepackage{mhequ}
\usepackage{comment}
\usepackage{microtype}
\usepackage{pifont}

\colorlet{darkblue}{blue!90!black}
\colorlet{darkred}{red!90!black}
\colorlet{darkgreen}{green!60!black}

\usepackage{accents}
\newcommand*{\bdot}[1]{
  \accentset{\mbox{\large\bfseries .}}{#1}}

\renewcommand{\theequation}{\thesection.\arabic{equation}}

\newtheorem{theorem}{Theorem}[section]
\newtheorem{lemma}[theorem]{Lemma}
\newtheorem{proposition}[theorem]{Proposition}
\newtheorem{corollary}[theorem]{Corollary}

\theoremstyle{definition}
\newtheorem{assumption}[theorem]{Assumption}
\newtheorem{definition}[theorem]{Definition}

\theoremstyle{remark}
\newtheorem{remark}[theorem]{Remark}
\Crefname{assumption}{Assumption}{Assumptions}

\newcommand{\vn}[1]{{\vert\kern-0.23ex\vert\kern-0.23ex\vert #1 
    \vert\kern-0.23ex\vert\kern-0.23ex\vert}}
\usepackage{mathtools}

\allowdisplaybreaks

\def\eps{\varepsilon}

\newcommand{\R}{{\mathbb{R}}}

\def\T{\mathbb{T}}

\newcommand{\E}{\mathbb{E}}

\def\path{\mathrm{path}}

\newcommand{\tand}{\quad\text{and}\quad}

\author{Konstantinos Dareiotis \thanks{School of Maths, University of Leeds, Leeds, U.K., \url{K.Dareiotis@leeds.ac.uk}}\,, 
Teodor Holland \thanks{School of Maths, University of Leeds, Leeds, U.K., \url{T.Holland1@leeds.ac.uk}}\,,    Khoa L\^e \thanks{School of Maths, University of Leeds, Leeds, U.K., \url{K.Le@leeds.ac.uk}}}
\title{Regularisation by multiplicative noise for reaction--diffusion equations}
\begin{document}

\maketitle

\begin{abstract}
We consider the stochastic reaction--diffusion equation in $1+1$ dimensions driven by multiplicative space--time white noise,
with a  distributional drift belonging to a Besov--H\"older space with any regularity index larger than $-1$.
We assume that the diffusion coefficient is a regular function which is bounded away from zero.
By using a combination of stochastic sewing techniques and Malliavin calculus, we show that the equation admits a unique solution.
\end{abstract}
\tableofcontents

 \section*{Acknowledgment}
KD and KL have been supported by  the Engineering \& Physical Sciences Research Council (EPSRC) [grant number EP/Y016955/1]. TH has been  supported by the Engineering \& Physical Sciences Research Council (EPSRC) [grant number EP/W523860/1].

\section{Introduction}
This paper is concerned with regularisation by noise phenomena for the  stochastic reaction--diffusion equation on the $1$-dimensional torus $\mathbb{T} := \mathbb{R}/\mathbb{Z}$ driven by multiplicative space--time white noise
\begin{equs}\label{SHE}
(\partial_t - \Delta)u = b(u) + \sigma(u)\xi, \qquad u|_{t=0}=u_0. 
\end{equs}
We show that \eqref{SHE} admits a unique solution  provided that $\sigma$ is regular, bounded away from zero  and that $b$ is a distribution with a regularity index more than $-1$ in the Besov--H\"older scale.

For results on  regularisation by noise phenomena for stochastic differential equations (SDEs), we refer the reader to \cite{zvonkin1974transform, veretennikov1980strong, davie2007uniqueness, catelliergubinelli, krylovrockner}. Concerning stochastic partial differential equations (SPDEs), 
the first results on  regularisation by noise can be traced back to the works of Gy\"ongy and Pardoux \cite{gyongy1993quasi},  \cite{Gyongy1993regularisation}. Therein, the authors consider SPDEs of the form 
\begin{equs}
(\partial_t - \Delta)u = b(u) + \xi, \qquad u|_{t=0}=u_0,
\label{additiveSHE}
\end{equs}
which corresponds to \eqref{SHE} with \(\sigma=1\).
It is well known that the deterministic counterpart of \eqref{additiveSHE} admits a unique  solution provided that $b$ is a  Lipschitz continuous function. Without Lipschitz regularity, solutions may not exists or may not be identified uniquely. 
The situation changes in the presence of noise. It is shown in \cite{gyongy1993quasi} and \cite{Gyongy1993regularisation} that \eqref{additiveSHE} admits a unique  strong solution provided that $b$ is merely the sum of a bounded  measurable function and an $L_p$-integrable function with some $p \geq 2$. Similar results were obtained for SPDEs in an abstract Hilbert-space framework with bounded and measurable drift in \cite{dapratoflandolipriolarockner}.  In \cite{butkovsky2019regularisation}, Butkovsky and Mytnik show when \(b\) is bounded and measurable, path-by-path uniqueness also holds for (\ref{additiveSHE}). 
For such drift, discrete approximation schemes for the solution of \eqref{additiveSHE} have been established with an optimal rate in \cite{butkovsky2023optimal}, quantifying earlier results from \cite{gyongylattice1,gyongylattice2}. 

Notice that in all the previous results, $b$ is quite  irregular, nevertheless it is a function. 
The first well-posedness result which accommodates distributional drift $b$ is due to  Athreya, Butkovsky, Mytnik, and the third co-author in \cite{athreya2024well}. 
In such case, the composition $b(u)$ is not well-defined a priori and solutions to \eqref{additiveSHE} are defined in a regularised sense which is similar to that of Bass and Chen in \cite{bass2001stochastic} (see \cref{defofsolution}). 
They show in \cite{athreya2024well} that \eqref{additiveSHE} admits a unique probabilistically strong solution provided  that $b$ belongs to the Besov space $\mathscr{B}_{q,\infty}^\alpha$ with $\alpha - 1/q \geq -1$, $\alpha>-1$,  and $q \in [1,\infty]$. 
Such Besov space includes bounded measurable functions, \(L_1\)-integrable functions, as well as Radon measures. To obtain such results, \cite{athreya2024well} establishes Lipschitz regularity for some related singular integrals using the stochastic sewing lemma introduced  by the third co-author in \cite{le2020sewing}.
The regularity threshold  $-1$ is in agreement with the finite dimensional analogue \cite{catelliergubinelli} where it is shown that any SDE driven by additive fractional Brownian motion with Hurst parameter $H \in (0,1)$ has a unique solution provided that the drift belongs to the Besov--H\"older space $\mathscr{B}_{\infty,\infty}^\alpha$ with $\alpha> 1 - \frac{1}{2H}$. 
The two results are related by setting 
$H = 1/4$, which is the temporal regularity of the random field solution of \eqref{additiveSHE} with $b =0$.
Quantitative convergence of discrete approximation schemes under the assumptions of \cite{athreya2024well} is also considered by Gouden\'ege, Haress and Richard in \cite{goudenège2024numericalapproximationstochasticheat}, extending \cite{butkovsky2023optimal}. 

All of the aforementioned  results concern the additive noise case. For the multiplicative case, much less is known. 
In \cite{ballygyongypardoux, gyongy1995nondegeratequasilinear,alabertgyongysingular}, the authors show that \eqref{SHE} has a unique solution when \(\sigma\) is regular and bounded away from \(0\), the drift is  measurable and bounded/locally bounded/integrable respectively. The proofs from these references rely on the Girsanov theorem, $L_p$-estimates for the density of the driftless equation (as obtained in \cite{pardoux1993absolute}), and a comparison principle.
Our result herein is an analogue of \cite{athreya2024well} for the multiplicative noise case. Namely, we show existence and uniqueness for \eqref{SHE} when  \(\sigma\) is regular and bounded away from \(0\), the  drift belongs to the Besov--H\"older space $\mathscr{B}_{q,\infty}^\alpha$ with $\alpha>-1$ and $q = \infty$. For simplicity, we do not consider the case when $q<\infty$, which allows us to obtain qualitative stability results and highlight the essential elements of our approach. 
Similar to \cite{athreya2024well}, our method also relies on the  stochastic sewing lemma from \cite{le2020sewing} which does not rely on Girsanov  theorem nor comparison principles. Therefore, the techniques within could also be applied to equations driven by  L\'evy noise and to systems of equations. Compare with \cite{athreya2024well}, while the probabilistic properties of the noise term in the additive case are explicitly understood, this is no longer the case for our multiplicative equation \eqref{SHE}. Therefore, employing the sewing methods in the present paper is more involved than \cite{athreya2024well}. In the sewing arguments in previous works, one approximates a solution using the integral form of the corresponding equation.  This works quite well in the additive noise case, \cite{catelliergubinelli,athreya2024well}. It also works quite well in some multiplicative noise cases if the noise is not too irregular,  for example  equations driven by fractional Brownian motion with Hurst parameter $H>1/2$ \cite{dareiotis2024path}. However, for $H < 1/2$, this approach leads to suboptimal results. The same is true for the  setting of the present paper. With such an approach, one would only be able to obtain well-posedness when $b$ has positive H\"older regularity.
Instead, in order to cover the whole regime $b \in \mathscr{B}^\alpha_{\infty, \infty}$ with $\alpha>-1$,   we  come up with sewing arguments that employ the flow of the driftless equation (see \cref{methodofproofsection} for a more detailed overview of our method). Consequently,  we need (and obtain) some regularisation estimates related to the density of the solution to the driftless equation and its derivatives. These estimates are  achieved via Malliavin calculus  which demands a relatively high regularity from \(\sigma\). 
 This  approach  is not equation-specific  but rather works as a general principle. In fact,  we are using it in a parallel work (\cite{RoughDistributional}) to improve the results from \cite{dareiotis2024path} to optimal, in the case of rough equations driven by fractional Brownian motion with Hurst parameter $H \in (1/3, 1/2)$ . 

The paper is organised as follows: In  \cref{notationsection},  we introduce some necessary notation. In  \cref{mainresultsection}, we define the key concepts, state the main result, and give an overview of the proof. Facts on stochastic sewing and Malliavin calculus are summarised in \cref{sec.prelim}.  In \cref{malliavinsection}, we study the regularity of the flow of the driftless equation, provide bounds on the moments of its Malliavin derivatives, and  establish relevant nondegeneracy estimates. In  \cref{driftilessapproxsection}, we quantify how well the flow of the driftless equation approximates the solution of \eqref{SHE}. In  \cref{regularisationsection}, we use the tools developed up until then to derive regularisation results for the solution of (\ref{SHE}) via the stochastic sewing lemma. These estimates are then used in  \cref{Spsection} to prove a stability result,
and in  \cref{driftsection} to estimate the drift term of \eqref{SHE}. The main result is proven in  \cref{mainproofsection}.
Finally, the appendix contains some useful auxiliary elementary estimates which are used throughout the article.

\section{Notation}\label{notationsection}

Let $H := L_2([0,1]\times \mathbb{T})$. Let $\xi :=\{\xi(h): h \in H\}$ be an isonormal Gaussian process\footnote{i.e. $\xi$ is a centered Gaussian family of random variables with $\mathbb{E}(\xi(h)\xi(\bar{h})) = \langle h, \bar{h} \rangle_H$ for all $h,\bar{h} \in H$.} on a complete probability space $(\Omega,\mathscr{F},\mathbb{P})$, and suppose that $\mathscr{F}$ is generated by $\xi$. Let $(\mathscr{F}_t)_{t \in [0,1]}$ be the filtration generated by $\xi$ and augmented by the $\sigma$-algebra $\mathcal{N}$ generated by all $\mathbb{P}$-null sets, that is
\begin{equs}
    \mathscr{F}_t := \sigma\big(\big\{\xi(\mathbf{1}_{[0,r)\times B}) : r \in [0,t], B \in \mathscr{B}(\mathbb{T})\big\}\big)\vee \mathcal{N}
\end{equs}
where for two $\sigma$-algebras $\mathscr{X},\mathscr{Y}$ we denote $\mathscr{X} \vee \mathscr{Y} := \sigma(\mathscr{X} \cup \mathscr{Y})$.
The predictable $\sigma$-algebra on $\Omega \times [0,1]$ is denoted by $\mathscr{P}$. The conditional expectation given $\mathscr{F}_t$ is denoted by $\mathbb{E}^t := \mathbb{E}(\cdot | \mathscr{F}_t)$. We use $L_p$ as a shorthand for $L_p(\Omega)$. For a sub-$\sigma$-algebra $\mathscr{G}\subset \mathscr{F}$, the conditional $L_p$-norm is denoted by
$$\|\cdot\|_{L_p|\mathscr{G}} := (\mathbb{E}(|\cdot|^p\vert \mathscr{G}))^{1/p},$$
and for $p \in [1,\infty)$, $q \in [1,\infty]$ we denote
\begin{equs}\|\cdot\|_{L_{p,q}^\mathscr{G}} := \| \|\cdot\|_{L_p|\mathscr{G}}\|_{L_q}.
\label{LpqGnpormdef}\end{equs}
Let $A\subseteq \mathbb{T}^d$ and $(B,|\cdot|)$ be a normed space. We denote by $\mathbb{B}(A,B)$  the collection of measurable functions $f: A \to B$ such that
$$\|f\|_{\mathbb{B}(A,B)} := \sup_{x \in A}|f(x)| <\infty.$$
 We denote space of continuous functions $f: A \to B$ by $C(A,B)$, and it is also canonically equipped with the $\mathbb{B}$-norm.
For $\alpha \in \mathbb{N}$ we denote by $C^\alpha(A,B)$ the space of continuous functions $f:A \to B$ such that for all multi-indices $l \in (\mathbb{Z}_{\geq 0})^d$ with $|l| \leq \alpha$ the derivative $\partial^l f$ is continuous, and 
\begin{equs}\|f\|_{C^\alpha(A,B)} := \sum_{|l| \leq \alpha}\|\partial^l f\|_{\mathbb{B}} <\infty.
\end{equs}
By convention the above sum includes the term $\| \partial^{(0,\dots,0)}f\|_{\mathbb{B}}$, where we define $\partial^{(0,\dots,0)}f:=f$.
For $\alpha \in (0,1)$ and $f:A \to B$, the $\alpha$-H\"older seminorm of $f$ is given by
$$[f]_{C^\alpha(A,B)} : = \sup_{\substack{x,y \in A\\ x\neq y}} \frac{|f(x) - f(y)|}{|x-y|^\alpha}.$$
For $\alpha \in (1, \infty) \setminus \mathbb{Z}$ we then denote by $C^\alpha(A,B)$ the space of all functions such that for all multi-indices  $l \in (\mathbb{Z}_{\geq 0})^d$ with $|l|< \alpha$, the derivative $\partial^l f$ exists, and 
\begin{equ}
\| f\|_{C^\alpha(A,B)} := \|f\|_{C^{\lfloor \alpha \rfloor}(A,B)} + \sum_{\alpha-1\leq|l|<\alpha} [\partial^l f]_{C^{\alpha - |l|}(A,B)} <\infty.
\end{equ}
The collection of smooth (i.e. infinitely differentiable)  and bounded functions with bounded derivatives will be denoted by
$$C^\infty(A,B) := \bigcap_{n =0}^{\infty} C^n(A,B).$$ 
For $\alpha<0$ we say that a Schwartz-distribution $f$ is of class $C^\alpha(\mathbb{R}, \mathbb{R})$, if
\begin{equs}
\|f\|_{C^\alpha(\mathbb{R},\mathbb{R})} : = \sup_{\varepsilon \in (0,1]} \varepsilon^{-\alpha/2}\| P_\varepsilon^{\mathbb{R}} f \|_{\mathbb{B}(\mathbb{R}, \mathbb{R})} <\infty,
\end{equs}
where for $(t,x) \in [0,1]\times \mathbb{R}$, $P_t^{\mathbb{R}} f(x) :=  \int_{\mathbb{R}}p_t^{\mathbb{R}}(x-y) f(y) dy$ and $p^\mathbb{R}$ is the \emph{heat kernel} on $\mathbb{R}$ defined by
\begin{equs}\label{heatkernel}
p_t^{\mathbb{R}}(x) := \frac{1}{(4\pi t)^{d/2}}\exp\Big(-\frac{|x|^2}{4t}\Big).
\end{equs}
Note that for any $\alpha \in \mathbb{R}\setminus \mathbb{Z}$, the space $C^\alpha(\mathbb{R},\mathbb{R})$ coincides with the Besov space $\mathscr{B}_{\infty,\infty}^\alpha$.
We also define the \emph{periodic heat kernel} on $\mathbb{T}$ for $t \in [0,1]$ and $x,y \in \mathbb{T}$ by
$$p_t(x,y) := \sum_{k \in \mathbb{Z}}p_t^{\mathbb{R}}(x-y+k)$$
and for $g:\mathbb{T} \to \mathbb{R}$ we denote $P_t g(x) := \int_{\mathbb{T}} p_t(x,y)g(y)dy$.
When no ambiguity can arise, we will simply abbreviate $C^\alpha(A)$ or $C^\alpha$ for $C^\alpha(A,B)$.
For $\alpha>-1$ we denote the completion of $C^\infty$ in the norm $\|\cdot\|_{C^\alpha}$ by $C^{\alpha+}$.
\begin{remark}\label{holderinclusions}
For all $\varepsilon>0$ we have the inclusions
$C^{\alpha+ \varepsilon} \subset C^{\alpha+} \subset C^{\alpha}$.
\end{remark}

\section{Formulation and main result}\label{mainresultsection}   
\subsection{Formulation and main result}
We introduce our main assumptions and the concept of solutions to \eqref{SHE}. 
\begin{assumption}\label{assumptions}  The function  $b$ is of class $C^{\alpha+}$ for some $\alpha \in (-1,0)$ and the function $\sigma$ is of class $C^4$.    Moreover,  there exists a positive constant $\mu$ such that 
$$\sigma^2 (x) \geq \mu^2 \quad \text{for all }x \in \R.$$ Finally, 
 the initial condition $u_0 : \mathbb{T} \to \mathbb{R}$ is a bounded and continuous deterministic function.
\end{assumption}

\begin{definition}[Regularised solution]\label{defofsolution}    
Let $u: \Omega \times [0,1]\times \mathbb{T} \to \mathbb{R}$ be a $\mathscr{P}\otimes \mathscr{B}(\mathbb{T})$-measurable random field, such that $u(t,x)$ is continuous in $(t,x) \in [0,1]\times \mathbb{T}$. We say that $u$ is a \emph{regularised solution}  of (\ref{SHE}) if there exists a $\mathscr{P}\otimes \mathscr{B}(\mathbb{T})$-measurable random field $D^{u}:\Omega \times [0,1]\times \mathbb{T} \to \mathbb{R}$ such that 
\begin{enumerate}
\item For any sequence $(b^n)_{n\in \mathbb{N}} \subset C^\infty$ such that $b^n \to b$ in $C^{\alpha}$, we have that
\begin{equs}\label{driftdef}
\sup_{(t,x) \in[0,1]\times \mathbb{T}}\Big|D_t^u(x) -\int_0^t \int_{\mathbb{T}}p_{t-r}(x,y)b^n(u(r,y))dydr\Big| \longrightarrow 0\end{equs}  
in probability.
\item
For each $(t,x) \in [0,1]\times \mathbb{T}$,
\begin{equs}u(t,x) = P_tu_0(x) + D_t^u(x) + \int_0^t \int_{\mathbb{T}} p_{t-r}(x,y) \sigma(u(r,y)) \xi(dy,dr)\qquad \text{a.s.} \label{integraleqn}
\end{equs}
\end{enumerate}
\end{definition}

\begin{remark}
	For a given regularised solution $u$, the random field $D^u$ is uniquely characterised by relation \eqref{driftdef}. Furthermore, in the more regular setting when $\alpha \geq 0$, \cref{defofsolution} reduces to  the standard notion of a mild solution. In such case, one has $D^u_t(x)=\int_0^t \int_{\mathbb{T}}p_{t-r}(x,y)b(u(r,y))dydr$.
\end{remark}

For $(S,T) \in [0,1]^2$ such that $S \leq T$, let us define the simplices \begin{equs}\,[S,T]_{\leq}^2 := \{  (s,t) \in [S,T]^2: s \leq t\}\, \text{ and } \,[S,T]_{<}^2 := \{(s,t) \in [S,T]^2 : s<t\}.
\end{equs}

To describe the regularity of the solutions, we introduce the following spaces of random fields.
\begin{definition}[The spaces $\mathscr{V}_p^\beta$, $\mathscr{U}_p^\beta$ and $\mathscr{U}^\beta$]\label{defofclassV} Let $ \beta \in [0, 1]$ and $p \in[1,\infty)$. We denote by $\mathscr{V}_p^{\beta}[0,1]$ the collection of all $\mathscr{P}\otimes \mathscr{B}(\mathbb{T})$-measurable functions $f:\Omega\times[0,1]\times \mathbb{T}\to \mathbb{R}$  such that $f \in \mathbb{B}([0,1]\times \mathbb{T}, L_p)$ and 
$$[f]_{\mathscr{V}_p^{\gamma}[0,1]} := \sup_{x \in \mathbb{T}}\sup_{(s,t) \in [0,1]_{<}^2}\frac{\|f_t(x) - P_{t-s}f_s(x)\|_{L_{p,\infty}^{\mathscr{F}_s}}}{|t-s|^{\gamma}} < \infty. $$
For $(S,T) \in [0,1]^2_{\leq}$, the space $\mathscr{V}_p^{\beta}[S,T]$ and the corresponding seminorm are defined analogously. 
We denote by  $\mathscr{U}_p^\beta$ the collection of all 
 regularised solutions $u$ of  \eqref{SHE} such that $D^u \in \mathscr{V}_p^{\beta}[0,1]$. 
We moreover define
$$\mathscr{U}^\beta := \bigcap_{p=1}^\infty \mathscr{U}_p^\beta.$$
\end{definition}

We are now in position to state our main theorem. 

\begin{theorem}[Well-posedness]\label{mainresult}
Let Assumption \ref{assumptions} hold. There exists a regularised solution $u$ to (\ref{SHE}) in the class $\mathscr{U}^{1 + \alpha/4}$. Moreover if $v$ is another solution of (\ref{SHE}) in the class $\mathscr{U}_2^\beta$ for some $\beta \geq \frac{1}{2}- \frac{\alpha}{4}$, then $u(t,x) = v(t,x)$ almost surely for all $(t,x) \in [0,1]\times \mathbb{T}$.
\end{theorem}

\subsection{Overview of methods of proofs}\label{methodofproofsection} 
\label{sub:overview_of_methods_of_proofs}
	The bulk of the proofs relies on moment estimates for singular  integrals which are typically of the form
	\begin{align*}
		I:=\int_0^1\int_{\mathbb{T}}h(y)f(u(r,y))dydr
	\end{align*}
	where $h$ is an integrable function, $u$ is a solution to \eqref{SHE} and $f$ is a distribution with negative H\"older regularity. 
	An effective tool to  estimate moments of $I$, which emerges from \cite{le2020sewing}, is the stochastic sewing lemma. 
	Heuristically, the lemma decomposes $I$ corresponding to partitions of the time interval $[0,1]$ with vanishing mesh size. More precisely, let $\pi$ be a partition of $[0,1]$, then one writes
	\begin{align*}
		I=\sum_{[s,t]\in \pi}\int_s^t\int_{\mathbb{T}}h(y)f(u(r,y))dydr.
	\end{align*}
	On each subinterval $[s,t]$, we approximate the random variable $\int_s^t\int_{\mathbb{T}}(\ldots)dydr$ by its conditional expectation given $\mathscr{F}_s$, i.e. $\E^s\int_s^t\int_{\mathbb{T}}(\ldots)dydr$. Because  the conditional law of $u(r,y)$ given $\mathscr{F}_s$ is unknown a priori, we further  approximate $u(r,y)$ by a random variable, denoted by $\psi^s(r,y)$. There are two desirable properties for these approximations. First,  one must recover $I$ when the mesh size of $\pi$ vanishes, namely
	\begin{align*}
		I=\lim_{|\pi|\downarrow0} \sum_{[s,t]\in \pi}\E^s\int_s^t\int_{\mathbb{T}}h(y)f(\psi^s(r,y))dydr.
	\end{align*}
	Second, the conditional expectation $\E^s f(\psi^s(r,y))$ is well-defined and can be estimated so that for some $p\ge2$ and $\varepsilon>0$, one has
	\begin{align}
		\label{me.c1}
		\|\E^s\int_s^t\int_{\mathbb{T}}h(y)f(\psi^s(r,y))dydr\|_{L_p(\Omega)}\lesssim(t-s)^{\frac12+\varepsilon}
		\\\shortintertext{and}
		\|\E^s\int_a^t\int_{\mathbb{T}}h(y)[f(\psi^s(r,y))-f(\psi^a(r,y)]\|_{L_p(\Omega)}\lesssim(t-s)^{1+\varepsilon}
		\label{me.c2}
	\end{align}
	for every $s\le a\le t$.
	Under these two properties, the stochastic sewing lemma can be applied, and it provides estimates for the $p$-th moment of $I$.


	Let us explain how \(\psi^s\) is chosen.
	Relation \eqref{integraleqn} provides a natural decomposition of a solution as the sum of a nondegenerate noise and the drift, namely
	\begin{align*}
    u(t,x)=P_tu_0(x)+D^u_t(x)+V_t(x), \text{ where }V_t(x)=\int_0^t \int_{\mathbb{T}} p_{t-r}(x,y) \sigma(u(r,y)) \xi(dy,dr).
	\end{align*}
	It follows that for each $s\le t$, 
	\begin{align*}
		u(t,x)=P_{t-s}u(s,\cdot)(x)+ [D^u_t(x)-P_{t-s}D^u_s(x)]+[V_t(x)-P_{t-s}V_s(x)].
	\end{align*}
	One could then choose to approximate $u(t,x)$ by the random variable
	\begin{align*}
		\psi^s(t,x)
		&:=P_{t-s}u(s,\cdot)(x)+[V_t(x)-P_{t-s}V_s(x)].
	\end{align*}
	The error of this approximation can be quantified by the following estimate
	\begin{equs}  \label{eq:order_approx}
		\|u(t,x)- \psi^s(t,x)\|_{L_p(\Omega)} \lesssim |t-s|^\gamma
	\end{equs}
	for every \(s\le t\) and 
	for some \(\gamma>0\). The larger the value of \(\gamma\) is, the better the approximation is. 
	We note that
	\begin{align*}
		[V_t(x)-P_{t-s}V_s(x)]
		=\int_s^t \int_{\mathbb{T}} p_{t-r}(x,y) \sigma(u(r,y)) \xi(dy,dr).
	\end{align*}
	In the additive case (i.e. when  $\sigma$ is a constant), $V_t(x)-P_{t-s}V_s(x)$ has a normal distribution and hence, the conditional expectation $\E^s f(\psi^s(t,x))$ can be  evaluated precisely. 
	The stochastic sewing method described above can be applied  (i.e. achieving \eqref{me.c1} and \eqref{me.c2}) under  some suitable regularity assumptions on $f$ and that $\gamma>1/2-\alpha/4 \approx 3/4$ for $\alpha \approx -1$ (recall that $-1< \alpha<0$ is the regularity of the drift). This is the approach from \cite{athreya2024well}.

 Going toward the multiplicative noise case, one might hope that a similar argument would work. Notice that in this case, the distribution of 
 $V_t(x)-P_{t-s}V_s(x)$ conditionally on $\mathscr{F}_s$ is not known a priori. A naive way to circumvent this issue is to consider 
\begin{align}\label{id.psinaive}
		\psi^s(t,x)
		&:=P_{t-s}u(s,\cdot)(x)+\int_s^t \int_{\mathbb{T}} p_{t-r}(x,y) \sigma(u(s,y)) \xi(dy,dr),
	\end{align}
 which is obtained by freezing the solution in the integrand at time $s$. In this way, conditionally on $\mathscr{F}_s$, $\psi^s(t,x)$ once again has a normal distribution, which allows for concrete analysis. However, one can not go far with this choice as it is immediate that
 \begin{equs}
	u(t,x)- \psi^s(t,x) =\int_s^t \int_{\mathbb{T}} p_{t-r}(x,y) \big( \sigma(u(r,y))- \sigma(u(s,y)) \big)\xi(dy,dr) ,
 \end{equs}
 whose moments are (expectedly)  of order \( |t-s|^{1/2}\) (consisting of two contributions of the same order \(1/4\) from the stochastic integral and from the temporal regularity of the solution). The exponent \(1/2\) falls short of the required threshold \(3/4\) which is necessary in the additive case.  
 This makes the naive approximation \eqref{id.psinaive} unsuitable for the sewing method under \cref{assumptions}.

 Moving forward, to resolve these issues, we introduce the following approximation 
 \begin{equs}\label{id.psigood}
     \psi^s(t,x):= \phi^{u(s,\cdot),s}(t,x), 
 \end{equs}
	where $\phi^{z,s}$ denotes the solution to the \emph{driftless equation}
	\begin{align*}
		(\partial_t - \Delta)\phi^{z,s} = \sigma(\phi^{z,s})\xi, \quad \phi^{z,s}(s,\cdot)=z(\cdot).
	\end{align*}
Observe that when \(\sigma\) is a constant, \eqref{id.psinaive} and \eqref{id.psigood} coincide, but otherwise they are generally different.  
Indeed, we show in  \cref{driftilessapproxsection} that  the approximation \eqref{id.psigood} satisfies the estimate \eqref{eq:order_approx}  with $\gamma= 1+\alpha/4$ which is larger than \(1/2-\alpha/4\) as is  required for the application of the sewing method. 
The distribution of $\psi^s(t,x)$ conditioned on $\mathscr{F}_s$  might not be as explicit as in the additive noise case but nevertheless, one can extract the information which is sufficient to verify \eqref{me.c1} and \eqref{me.c2}. This essentially boils down to obtaining estimates related to the density of the solution of the driftless equation and its derivatives, which are achieved by tools from Malliavin calculus (see  \cref{nondegeneracysection}). 

When comparing our method to the existing ones from the literature, we can draw some similarities as well as genuine  differences. 
The works \cite{ballygyongypardoux, gyongy1995nondegeratequasilinear,alabertgyongysingular} also utilise estimates on the density of the solution to the driftless equation, however, in a completely different way. In fact, these works use Girsanov theorem to extract relevant and useful a priori estimates for the solution to \eqref{SHE} from the solution of the driftless equation. Under our main assumption, the Girsanov theorem is not applicable which makes this argument obsolete. Additionally, our uniqueness argument relies on qualitative stability estimates, as opposed to comparison principles in the aforementioned works. 
Compare to \cite{athreya2024well}, we also use stochastic sewing method. However, while \cite{athreya2024well} relies on the approximation \eqref{id.psinaive}, we introduce and utilise the better approximation \eqref{id.psigood}. To the best of our knowledge, this is the first time it has been used in the study of regularisation by noise phenomena by sewing methods. Furthermore, because the conditional law of \(\psi^s\) is not  explicit, additional works have been carried out in order to apply the sewing method successfully.

The driftless equation also appears in \cite{catellier2022regularization} in the study of regularisation by multiplicative fractional noise for SDEs.  In this work, the authors employ a transformation, which is based on the inverse of the flow generated by the driftless equation, to transform the original equation into an additive one. Comparing the results of \cite{catellier2022regularization} and  \cite{dareiotis2024path} reveals that such transformation is quite demanding and does not lead to results which are in alignment with \cite{catelliergubinelli}. 
The connection between \eqref{SHE} and the driftless equation is well-known, perhaps since the Girsanov theorem. Another instance of such relation appears in  \cite{imkeller2001conjugacy} in a different context. Our work therefore exhibits a new connection between the two equations.

\section{Preliminaries}\label{sec.prelim}
We state and recall some facts about stochastic sewing  and Malliavin calculus. When a result is known, we refer to the cited references for proofs. Other statements are relatively straightforward and short proofs are given. 
\subsection{Stochastic sewing}\label{sewingsection}
Let $(S,T) \in [0,1]_{<}^2$. For any functions  $\mathscr{A}:[S,T] \to \mathbb{R}$, $A:[S,T]^2_\le \to \mathbb{R}$, for any $(s,t) \in [0,1]_{\leq}^2$ and $a \in [s,t]$, we define $\mathscr{A}_{s,t} := \mathscr{A}_t - \mathscr{A}_s$, and $\delta A_{s,a,t} := A_{s,t} - A_{s,a} - A_{a,t}$. 

The first stochastic sewing lemma is introduced by the third co-author in \cite{le2020sewing}. To best suit our purpose herein, we state a conditional version of the lemma which applies in settings with $L_{q,p}^{\mathscr{F}_s}$-norms (defined in (\ref{LpqGnpormdef})).
This version is originated from the works \cite{friz2021rough,athreya2024well,le2023stochastic}, where the reader can find its proof.

\begin{lemma}[Conditional stochastic sewing lemma]\label{conditionalSSL}
  \label{lem:conditionalSSL}
Let $p,q$ satisfy $2 \leq q \leq p \leq \infty$ with $q <\infty$. Let $(S,T) \in [0,1]_{\leq}^2$ and let $A:[S,T]_{\leq}^2 \to L_p(\Omega)$ be a function such that for any $(s,t) \in [S,T]_{\leq}^2$ the random vector $A_{s,t}$ is $\mathscr{F}_t$-measurable. 
Suppose that for some $\varepsilon_1, \varepsilon_2>0$ and $C_1,C_2$ the bounds
\begin{equs}
\|A_{s,t}\|_{L_{q,p}^{\mathscr{F}_s}} \leq C_1|t-s|^{1/2+\varepsilon_1}, \quad
\|\mathbb{E}^s \delta A_{s,a,t}\|_{L_p} \leq C_2|t-s|^{1+\varepsilon_2}   \label{condiSSL2}
\end{equs}
hold for all $S \leq s \leq a \leq t \leq T$. Then, there exists a unique map $\mathscr{A}:[S,T] \to L_p(\Omega)$ such that $\mathscr{A}_S =0$, $\mathscr{A}_t$ is $\mathscr{F}_t$-measurable for all $t \in [S,T]$, and the following bounds hold for some constants $K_1,K_2>0$:
\begin{equs}
\|\mathscr{A}_{s,t} - A_{s,t}\|_{L_{q,p}^{\mathscr{F}_s}} &\leq K_1 |t-s|^{1/2 + \varepsilon_1} \label{condiSSL3}
\\
\|\mathbb{E}^s(\mathscr{A}_{s,t} - A_{s,t})\|_{L_p} &\leq K_2 |t-s|^{1+\varepsilon_2}. \label{condiSSL4}
\end{equs}
Furthermore, there exists a constant $K$ depending only on $\varepsilon_1,\varepsilon_2, d,p$ such that $\mathscr{A}$ satisfies the bound
$$\|\mathscr{A}_{s,t}\|_{L_{q,p}^{\mathscr{G}_s}} \leq K C_1 |t-s|^{1/2 + \varepsilon_1} + K C_2 |t-s|^{1+\varepsilon_2}$$
for all $(s,t)\in [S,T]_{\leq}^2$.
\end{lemma}
We will call $A$   \emph{a germ of} the process $\mathscr{A}$. 
In practice, we mostly take \(q=p\) (in which case, $L_{q,p}^{\mathscr{F}_s}$-norm and $L_{p}$-norm coincide) and \(q=\infty\).  
\subsection{Malliavin calculus}
Let $\mathscr{W}$ denote the the space of \emph{smooth and cylindrical random variables}, i.e. random variables of the form
$$F  =f(\xi(h_1),\dots, \xi(h_n))$$
for some $n \in \mathbb{N}$, $h_1,\dots, h_n \in H$, and for some smooth $f$ such that $f$ and its partial derivatives of all orders have polynomial growth.
The \emph{Malliavin derivative} of such a random variable is given by
$$\mathscr{D}_{\theta,\zeta}F:= \sum_{i=1}^n \partial_{i}f(\xi(h_1),\dots, \xi(h_n)) h_i(\theta,\zeta)$$
for all $(\theta,\zeta) \in [0,1]\times \mathbb{T}$ where $\partial_i$ denotes partial derivative with respect to the $i$-th argument. For all $k \in \mathbb{Z}_{\geq 0}$, $p \geq 1$ the iterated Malliavin derivative $\mathscr{D}^k$ is closable as an operator from $L_p(\Omega)$ into $L_p(\Omega ; H^{\otimes k})$. By convention, the $0$-th Malliavin derivative is the identity map, and $H^{\otimes 0} := \mathbb{R}$.
For $k \in \mathbb{Z}_{\geq 0}$ and $p \geq 1$, we denote by $\mathscr{W}^k_p$ the completion of $\mathscr{W}$ with respect to the norm
$$F \mapsto \|F\|_{\mathscr{W}^k_p} :=\Big( \mathbb{E}|F|^p + \sum_{i=1}^k \mathbb{E}\| \mathscr{D}^iF\|_{H^{\otimes i}}^p \Big)^{1/p}.$$
We moreover use the notation
$$\mathscr{W}^k := \bigcap_{p \geq 1} \mathscr{W}_p^k.$$
On the class $\mathscr{W}_p^k$ one can also define the $\bdot{\mathscr{W}}_p^k$-seminorm by
$$F \mapsto \|F\|_{\bdot{\mathscr{W}}_p^k} := \|\|\mathscr{D}^kF\|_{H^\otimes k}\|_{L_p}.$$ By convention, we have
$$\|\cdot\|_{\mathscr{W}_p^0} = \|\cdot\|_{\bdot{\mathscr{W}}_p^0}  = \|\cdot\|_{L_p}.$$
Note that $\|\cdot\|_{\mathscr{W}_p^k}$ and $\sum_{i=0}^K\|\cdot\|_{\bdot{\mathscr{W}}_p^i}$ are equivalent norms. The above definitions can be extended for the Hilbert-space valued case as follows. Let $V$ be a separable Hilbert-space, and consider the family $\mathscr{W}(V)$ of random variables of the form
$$F = \sum_{i=1}^m F_i v_i$$
for some $F_1,\dots, F_m \in \mathscr{W}$, and $ v_1,\dots, v_m \in V$. For $k\geq 1$, we define
$$\mathscr{D}^k F := \sum_{j=1}^m \mathscr{D}^kF_j \otimes v_j.$$
Then $\mathscr{D}^k$ is a closable operator from $L_p(\Omega;V)$ into $L_p(\Omega; H^{\otimes k}\otimes V)$ for any $p \geq 1$. We define the space $\mathscr{W}^k_p(V)$ as the completion $\mathscr{W}(V)$ with respect to the norm
$$F \mapsto \|F\|_{\mathscr{W}^k_p(V)} :=\Big( \mathbb{E}\|F\|_V^p + \sum_{i=1}^k \mathbb{E}\| \mathscr{D}^iF\|_{H^{\otimes i}\otimes V}^p \Big)^{1/p}.$$
For a random variable $u \in L_2(\Omega;H)$ it is said that $u \in \text{dom}(\delta)$, if there exists a constant $c>0$ such that
$$\mathbb{E}\langle \mathscr{D}F, u \rangle_H \leq c\|F\|_{L_2}$$
for all $F \in \mathscr{W}^1_2$. If this holds, then $\delta(u)$ denotes the unique element of $L_2(\Omega)$ that satisfies
\begin{equs}\mathbb{E}(F\delta(u)) = \mathbb{E}\langle \mathscr{D}F , u \rangle_H
\end{equs}
for any $F \in \mathscr{W}^1_2$. The random variable $\delta(u)$ is called the \emph{Skorokhod integral} (or the \emph{divergence}) of $u$. If in addition $u$ is adapted, then the Skorokhod integral coincides with the usual stochastic integral, that is for all $t \in [0,1]$ we have
$$\int_0^t \int_{\mathbb{T}} u(r,y) \xi(dy,dr) = \delta(u \mathbf{1}_{[0,t]}).$$
The following result follows from \cite[Proposition 2.1.4]{nualart2006malliavin}
\begin{proposition}
[Malliavin integration by parts]\label{repeatedMalliavinbyparts}
Let $n \in \mathbb{N}$, $u, G_0 \in \mathscr{W}^n$ and let $f:\mathbb{R} \to \mathbb{R}$ be $n$ times differentiable. Suppose moreover that for all $p \in [1,\infty)$, we have
$\mathbb{E}\|\mathscr{D}u(t,x)\|_{H}^{-p} <\infty.$
Define iterated Skorokhod integrals recursively for $k \in \{0,\dots, n-1\}$ by
\begin{equs}
G_{k+1} := \delta\Big(\frac{\mathscr{D}u}{\|\mathscr{D}u\|_H^2}G_k\Big).
\end{equs}    
The following holds:
\begin{equs}
\mathbb{E}\big(\nabla^n f(u) G_0\big) = \mathbb{E}\big( f(u) G_n\big).
\end{equs}
\end{proposition}
We also recall the combinatorial notation from \cite{chen2021regularity}.  Let $n \in \mathbb{N}$.
\begin{itemize}
\item For $1 \leq k \leq n$, we denote by $\Lambda(n,k)$ the set of partitions of the integer $n$ of length $k$, that is, if $\lambda \in \Lambda(n,k)$, then $\lambda \in \mathbb{N}^k$, and by writing $\lambda = (\lambda_1 ,\dots , \lambda_k)$, it satisfies
$$\lambda_1 \geq \dots \geq \lambda_k \geq 1 \qquad\text{ and }\qquad \sum_{i=1}^k \lambda_i = n. $$
\item For $\lambda \in \Lambda(n,k)$, we let $\mathcal{P}(n, \lambda)$ be all partitions of $n$ ordered objects $\{\theta_1 ,\dots , \theta_n\}$, with $\theta_1 \geq \dots \geq \theta_n$ into $k$ groups $\{\theta_1^1,\dots, \theta_{\lambda_1}^1\},\dots, \{\theta_1^k,\dots, \theta_{\lambda_k}^k\}$, such that within each group the elements are ordered, i.e.
$\theta_1^j \geq \dots \geq \theta_{\lambda_j}^j$ for $1 \leq j \leq k$.
Note that
$| \mathcal{P}(n,\lambda)| = { n \choose \lambda_1,\dots, \lambda_k} = \frac{n!}{\lambda_1!\dots \lambda_k!}.$
\item For a generic element
\begin{align*}
\gamma &:= ((\theta_1,\zeta_1),\dots ,(\theta_n,\zeta_n)) \in ([0,1]\times \mathbb{T})^n,
\end{align*}
we will denote by $\hat{\gamma}_k$ the element of $([0,1]\times \mathbb{T})^{n-1}$ that is obtained by omitting the $k$-th entry of $\gamma$, i.e.
\begin{equs}
\hat{\gamma}_k &:= ((\theta_1,\zeta_1),\dots, (\theta_{k-1},\zeta_{k-1}),(\theta_{k+1},\zeta_{k+1}),\dots, (\theta_n,\zeta_n)). \label{gammahatdef}
\end{equs}
\end{itemize}

We state some generic estimates on the Malliavin derivatives of functions of random variables which are needed in later sections. The proofs of these results rely purely on elementary  principles, such as the chain rule.
\begin{proposition}[{\cite[Lemma 5.3]{chen2021regularity}}]\label{chenhunualart}
Suppose that $f \in C^n$ and $\phi \in \mathscr{W}^n$. Then, for almost all $\gamma = ((\theta_1,\zeta_1), \dots, (\theta_n, \zeta_n)) \in ([0,1]\times \mathbb{T})^n$, we have
\begin{equation}\label{CHN1}
\mathscr{D}_\gamma^n f(\phi) = \sum_{k=1}^{n} f^{(k)}(\phi) \sum_{\lambda \in \Lambda(n,k)} \sum_{\mathcal{P}(n,\lambda)} \prod_{j=1}^k \mathscr{D}^{\lambda_j}_{(\theta_1^j, \zeta_1^j),\dots, (\theta_{\lambda_j}^j, \zeta_{\lambda_j}^j) }\phi.
\end{equation}
\end{proposition}

\begin{lemma}\label{Dsigmaubuckling}
Fix some constants $\varepsilon>0$ and $n \in \mathbb{N}$. For $i \in \{1,\dots, 4\}$ consider random variables $\phi^i \in \mathscr{W}^n$. Suppose that for all $p \in [1,\infty)$ and $k \in \{1,\dots, n-1\}$ there exists a constant $N_0 = N_0(k,p)$ such that
\begin{equs}\max_{i \in \{1,\dots,4\}}\|\phi^i\|_{\bdot{\mathscr{W}}_p^k} \leq N_0 \varepsilon^k. \label{uilessthanepsilontothek}
\end{equs}
Suppose that $f: \mathbb{R} \to \mathbb{R}$ is smooth. For all $p \in [1,\infty)$ the following statements hold.
\begin{enumerate}[(a)]
    \item \label{fofphi1bound}
There exists a constant $N = N(n,p,\|f\|_{C^n})>0$ such that
\begin{equs}\|f(\phi^1)\|_{\bdot{\mathscr{W}}_p^n} \leq N \varepsilon^{n} + N \|\phi^1\|_{\bdot{\mathscr{W}}_p^n}.\end{equs}
\item \label{fphi1minusfphi2bound}
There exists a constant $N = N(n,p, \|f\|_{C^{n+1}})$ such that
\begin{equs}\|f(\phi^1) - f(\phi^2)\|_{\bdot{\mathscr{W}}_p^n} \leq N\sum_{i=0}^{n-1}\varepsilon^{n-i}\|\phi^1-\phi^2\| _{\bdot{\mathscr{W}}_{2p}^i}+ N\|\phi^1-\phi^2\|_{\bdot{\mathscr{W}}_{p}^n}. \label{sigmaudiffboundgeneral}
\end{equs}
\item\label{generalities4point}
Suppose moreover that (\ref{uilessthanepsilontothek}) also holds for $k=n$.
There exists a constant $N = N(n,p,\|f\|_{C^{n+2}})$ such that
\begin{equs}
&\|f(\phi^1) - f(\phi^2) - f(\phi^3) + f(\phi^4)\|_{\bdot{\mathscr{W}}_p^n}\\
&\qquad\leq N \sum_{i+j+k =n} \|\phi^1 - \phi^2\|_{\bdot{\mathscr{W}}_{4p}^i}\Big(\|\phi^1 - \phi^3\|_{\bdot{\mathscr{W}}_{4p}^j} + \|\phi^2 - \phi^4\|_{\bdot{\mathscr{W}}_{4p}^j}\Big)\varepsilon^{k}\\
&\qquad\qquad + N \sum_{i=0}^{n-1}\|\phi^1 - \phi^2 - \phi^3 + \phi^4\|_{\bdot{\mathscr{W}}_{2p}^i}\varepsilon^{n-i}+ N\|\phi^1  - \phi^2 - \phi^3 + \phi^4\|_{\bdot{\mathscr{W}}_p^n}.
\end{equs}
\end{enumerate}
\end{lemma}

\begin{proof}
By (\ref{CHN1}) we can see that
\begin{equs}
\| f(\phi^1) \|_{\bdot{\mathscr{W}}_p^n} &\lesssim \|f\|_{C^1}\|\phi^1\|_{\bdot{\mathscr{W}}_p^n} +\|f\|_{C^n} \sum_{k=2}^n \sum_{\lambda \in \Lambda(n,k)} \sum_{\mathscr{P}(n,\lambda)}\prod_{j=1}^k \| \phi^1\|_{\bdot{\mathscr{W}}_{2^k p}^{\lambda_j}}.\end{equs}
The second term in this expression can be estimated using (\ref{uilessthanepsilontothek}) by
\begin{equs}
  \sum_{k=2}^n \sum_{\lambda \in \Lambda(n,k)} \sum_{\mathscr{P}(n,\lambda)}\prod_{j=1}^k \varepsilon^{ \lambda_j}\lesssim  \varepsilon^{n}
\end{equs}
where we used the definition of $\Lambda(n,k)$. This proves point (\ref{fofphi1bound}).

We proceed by proving point (\ref{fphi1minusfphi2bound}). Note that by the Minkowski inequality and the Leibniz rule, by (\ref{uilessthanepsilontothek}), and by point (\ref{fofphi1bound}) we get
\begin{equs}
    \|f(\phi^1) - f(\phi^2)\|_{\bdot{\mathscr{W}}_p^n} & = \Big\|\int_0^1 f'(\theta \phi^1 + (1-\theta)\phi^2)(\phi^1-\phi^2)  d \theta \Big\|_{\bdot{\mathscr{W}}_p^n}\\
    & \lesssim \int_0^1 \Big(\sum_{i=0}^{n-1} \| f'(\theta \phi^1 + (1-\theta)\phi^2)\|_{\bdot{\mathscr{W}}_{2p}^{n-i}}\|\phi^1-\phi^2\|_{\bdot{\mathscr{W}}_{2p}^i}\\
    &\qquad\qquad\qquad+ \|\| f'\big(\theta \phi^1 + (1-\theta)\phi^2\big)\mathscr{D}^n(\phi^1 - \phi^2)\|_{H^{\otimes n}}\|_{L_p}\Big) d \theta\\
    & \lesssim \sum_{i=0}^{n-1} \|f'\|_{C^{n-i}}\varepsilon^{n-i} \|\phi^1- \phi^2\|_{\bdot{\mathscr{W}}_{2p}^i} + \|f'\|_{\mathbb{B}}\|\phi^1-\phi^2\|_{\bdot{\mathscr{W}}_{p}^n}.
\end{equs}
From here point (\ref{fphi1minusfphi2bound}) follows.

Finally, we prove point (\ref{generalities4point}). By \cref{4pointcompar} we have that
\begin{equs}
&\|f(u^1) - f(u^2) - f(u^3) + f(u^4)\|_{\bdot{\mathscr{W}}_p^n}\\
&\leq
\Big\| \int_0^1 \int_0^1(\phi^1 - \phi^2)\Big(\theta(\phi^1 - \phi^3) + (1-\theta)(\phi^2 - \phi^4)\Big) \nabla^2f(\Theta_1(\theta,\eta)) d\eta d\theta\Big\|_{\bdot{\mathscr{W}}_p^n}\\
&\qquad +\Big\|(\phi^1 - \phi^2 - \phi^3 + \phi^4)\int_0^1 \nabla f(\Theta_2(\theta)) d\theta \Big\|_{\bdot{\mathscr{W}}_p^n} \\
 &=: A +B,
\end{equs}
where for each $\theta,\eta \in [0,1]$, the expressions $\Theta_1(\theta,\eta), \Theta_2(\theta)$ are convex combinations of $\phi^1,\dots,\phi^4$.
By the Minkowski inequality and by the H\"older inequality, we get
we get
\begin{equs}
&A \lesssim \int_0^1 \int_0^1 \sum_{i+j+k =n} \|\phi^1 - \phi^2\|_{\bdot{\mathscr{W}}_{4p}^i}\|\theta(\phi^1 - \phi^3) + (1-\theta)(\phi^2 - \phi^4)\|_{\bdot{\mathscr{W}}_{4p}^j}\\
&\qquad\qquad\qquad\qquad\qquad\qquad \times\|\nabla^2 f(\Theta_1(\theta,\eta))\|_{\bdot{\mathscr{W}}_{4p}^k} d\eta d\theta. 
\end{equs}
By using point (\ref{fofphi1bound}) for $k \geq 1$, and using the regularity of $f$ for $k=0$, we can see that
$
\|\nabla^2 f(\Theta(\theta,\eta))\|_{\bdot{\mathscr{W}}_{4p}^k} \lesssim \varepsilon^k.$ Therefore we get
\begin{equs}
A \lesssim \sum_{i+j+k = n}\|\phi^1 - \phi^2\|_{\bdot{\mathscr{W}}_{4p}^i}\Big(\|\phi^1 - \phi^3\|_{\bdot{\mathscr{W}}_{4p}^j}+\|\phi^2 - \phi^4\|_{\bdot{\mathscr{W}}_{4p}^j}\Big)\varepsilon^{k}.
\end{equs}
Finally,
\begin{equs}
B &\lesssim \int_0^1 \sum_{i=0}^{n-1}\|\phi^1 - \phi^2 - \phi^3 + \phi^4\|_{\bdot{\mathscr{W}}_{2p}^i}\|\nabla f(\Theta_1(\theta))\|_{\bdot{\mathscr{W}}_{2p}^{n-i}} d \theta +  \|\phi^1 - \phi^2 - \phi^3 + \phi^4\|_{\bdot{\mathscr{W}}_{p}^n}\|f'\|_{\mathbb{B}} \\ 
&\lesssim\sum_{i=0}^{n-1}\|\phi^1 - \phi^2 - \phi^2 + \phi^4\|_{\bdot{\mathscr{W}}_{2p}^i}\varepsilon^{n-i} + \|\phi^1 - \phi^2 - \phi^3 + \phi^4\|_{\bdot{\mathscr{W}}_{p}^n}
\end{equs}
where the last inequality again follows from point (\ref{fofphi1bound}). Hence the proof is finished.
\end{proof}

\begin{lemma}\label{DXperYlemma}
Consider constants $\varepsilon,c>0$, $n \in \mathbb{Z}_{\geq 0}$. Let $X \in \mathscr{W}^{n+1}$, $Y \in \mathscr{W}^n$ with $Y \geq 0$.
Suppose that for all $p \in (2,\infty)$ there exists a constant $N_0 = N_0>0$ such that for all $k \in \{1,\dots, n+1\}$, $l \in \{0,\dots, n\}$ we have
\begin{equs}
\|X\|_{\bdot{\mathscr{W}}_p^k} &\leq N_0\varepsilon^k,\qquad
\|Y\|_{\bdot{\mathscr{W}}_p^{l}} &\leq N_0 \varepsilon^{c+l}, \qquad \mathbb{E}[Y^{-p} ]\leq N_0 \varepsilon^{-cp}. \label{XYassumptions}
\end{equs}
Define an $H$-valued random variable by
$w := \frac{\mathscr{D}X}{Y}$.
Then, for each $p \in [1,\infty)$, there exists a constant $N = N(N_0,n,p)>0$ such that
$$\|\mathscr{D}^n w \|_{L_p (\Omega, H^{\otimes (n+1)})} \leq N \varepsilon^{n+1 -c}.$$
\end{lemma}
\begin{proof}
We may assume that $p>2$. By (\ref{XYassumptions}) and  a simple approximation argument (shifting $Y$ away from $0$) we have that $Y \in \mathscr{W}^n$, and for $m \in \{1,\dots, n\}$, $\gamma \in ([0,1]\times \mathbb{T})^m$ we have
\begin{equs}
\mathscr{D}^m_\gamma(Y)^{-1} = \sum_{k=1}^m \frac{(-1)^k k!}{(Y )^{k+1}}\sum_{\lambda \in \Lambda(n,k)}\sum_{\mathcal{P}(n, \lambda)}\prod_{j=1}^k \mathscr{D}^{\lambda_j}_{(\theta_1^j,\zeta_1^j),\dots(\theta_{\lambda_j}^j,\zeta_{\lambda_j}^j)}Y.
\end{equs}
Thus by (\ref{XYassumptions}) we have
\begin{equs}
\|\mathscr{D}^m (Y^{-1})\|_{L_p(\Omega; H^{\otimes m})}&\lesssim \sum_{k=1}^m \varepsilon^{-c(k+1)} \sum_{\lambda \in \Lambda(m,k)}\sum_{\mathcal{P}(m,\lambda)}\prod_{j=1}^k \varepsilon^{c + \lambda_j}\\
&\lesssim \sum_{k=1}^m \varepsilon^{-c(k+1)} \sum_{\lambda \in \Lambda(m,k)}\sum_{\mathcal{P}(m,\lambda)} \varepsilon^{\sum_{j=1}^k(c+\lambda_j)}\\
&\lesssim \sum_{k=1}^m\varepsilon^{-c(k+1) + ck + m} \lesssim \varepsilon^{m-c}.
 \label{DnDThetaminustwo}
\end{equs}
For $\gamma \in ([0,1]\times \mathbb{T})^n$ and for $\eta \in [0,1]\times \mathbb{T}$, using the Leibniz rule, we have
\begin{align}\label{leibnizforw}
\mathscr{D}_\gamma^n w(\eta) = \sum_{\lambda_1 + \lambda_2 =n}\sum_{(\gamma_1 , \gamma_2) \in \mathcal{P}(n,2)} \mathscr{D}_{\gamma_1}^{\lambda_1}(\mathscr{D}_\eta X) \mathscr{D}_{\gamma_2}^{\lambda_2}(Y^{-1}).
\end{align}
Using (\ref{leibnizforw}), (\ref{XYassumptions}) and (\ref{DnDThetaminustwo}) we get that
\begin{align*}
\| \| \mathscr{D}^n w \|_{H^{\otimes (n+1)}}\|_{L_p} &\lesssim\sum_{ \lambda_1 + \lambda_2 =n}\sum_{(\gamma_1 , \gamma_2) \in \mathcal{P}(n,2)} \| \| \mathscr{D}^{\lambda_1 +1}X \|_{H^{\otimes (\lambda_1 +1)}}\|_{L_{2p}}\| \| \mathscr{D}^{\lambda_2}(Y^{-1})\|_{H^{\otimes \lambda_2}}\|_{L_{2p}}\\
&\lesssim \sum_{= \lambda_1 + \lambda_2 =n}\sum_{(\gamma_1 , \gamma_2) \in \mathcal{P}(n,2)} \varepsilon^{\lambda_1 +1}\varepsilon^{\lambda_2 -c}\lesssim \varepsilon^{n+1 -c}
\end{align*}
as required.
\end{proof}

\section{Malliavin regularity for the driftless equation}\label{malliavinsection}
This section is concerned with the solution of the  driftless multiplicative stochastic heat equation 
\begin{equs}\label{driftlessSHE}
(\partial_t - \Delta)\phi &= \sigma(\phi) \xi, \qquad \phi(0,\cdot)&=\phi_0
\end{equs}
and its Malliavin derivatives. 
Herein, $\phi_0 \in C(\mathbb{T})$ is fixed and the solution $\phi: \Omega \times [0,1]\times \mathbb{T} \to \mathbb{R}$ is a $\mathscr{P}\otimes\mathscr{B}(\mathbb{T})$-measurable random field,  a.s.  continuous on \([0,1]\times \mathbb{T}\),  and satisfies the following equation almost surely
\begin{equs}\label{integralSHE}
\phi(t,x) =P_t\phi_0(x)+ \int_0^t \int_{\mathbb{T}}p_{t-r}(x,y) \sigma(\phi(r,y))\xi(dy,dr), \quad \forall (t,x) \in [0,1]\times \mathbb{T}.
\end{equs}

\subsection{Moment bounds for Malliavin derivatives}\label{posimomentboundsection}

We will show the following result.
\begin{lemma}\label{posimomentbound} Let $\phi$ be the solution of (\ref{driftlessSHE}).
For any $n \in \mathbb{Z}_{\geq 0}$ and $p \in [1,\infty)$, if in addition $\sigma \in C^n$,  then there exists some constant $N = N(n,p,\| \sigma\|_{C^n})$ such that for all $t \in [0,1]$ we have
$$\sup_{x \in \mathbb{T}}\| \phi(t,x)\|_{\bdot{\mathscr{W}}_p^n} \leq N(1+\mathbf{1}_{n=0}\|u_0\|_{\mathbb{B}(\mathbb{T})}) t^{n/4}.$$
\end{lemma}

To show this, let us recall the  following non-quantitative result from  \cite[Proposition 4.3]{bally1998malliavin}.
\begin{proposition}[Boundedness and Malliavin differentiability]\label{malldiffability}
Let $\phi$ be the solution of (\ref{driftlessSHE}). For any $n \in \mathbb{N}$, $p \in [1,\infty)$, if $\sigma \in C^n$, then we have
$$\sup_{(t,x) \in [0,1]\times\mathbb{T}}\| \phi(t,x)\|_{{\mathscr{W}}_p^n} <\infty.$$
\end{proposition}

We will also need the following result from \cite[Lemma 5.6]{chen2021regularity}:

\begin{proposition}\label{CHN3}
Let $\phi$ be the solution of (\ref{driftlessSHE}), and let $n \in \mathbb{N}, p \in [1,\infty), \sigma \in C^n.$
For all $(t,x) \in [0,1]\times \mathbb{T}$ and for almost every $\gamma = (\theta_i,\zeta_i)_{i=1}^n \in ([0,1]\times \mathbb{T})^n,$  we have
\begin{equs}
\mathscr{D}_\gamma^n \phi(t,x) &= \mathbf{1}_{[0,t]}(\theta^*)\sum_{k=1}^n p_{t-\theta_k}(x,\zeta_k) \mathscr{D}_{\hat{\gamma}_k}^{n-1}[\sigma(\phi(\theta_k,\zeta_k))]\\
&\qquad+\mathbf{1}_{[0,t]}(\theta^*)\int_{{\theta}^*}^t \int_{\mathbb{T}}p_{t-r}(x,y) \mathscr{D}_{\gamma}^n[\sigma(\phi(r,y))] \xi(dy,dr),                          \label{eq:Malliavin_der_solution}
\end{equs}
where $\hat{\gamma}_k$ is defined by (\ref{gammahatdef}) and $\theta^* := \max_{k \in \{1,\dots, n\}} \theta_k$.
\end{proposition}
\begin{remark}
Note that in the above equation the stochastic integral can be taken over the time interval $[0,t]$ rather than $[\theta^*,t]$, since for $r \leq \theta^*$ the Malliavin derivative $\mathscr{D}_\gamma^n \sigma(\phi(r,y))$ is zero (see \cite[Remark 5.1]{sanz2004course}).
\end{remark}

We can now proceed with the proof of \cref{posimomentbound}. 

\begin{proof}[Proof of \cref{posimomentbound}]
We will prove the result by induction. By the BDG\footnote{We call the Burkholder-Davis-Gundy inequality ``BDG inequality'' for brevity.} inequality and by the boundedness of $\sigma$, the result holds for the case $n=0$. Suppose that the result holds for the first $(n-1)$ Malliavin derivatives. We aim to show that the result also holds for the $n$-th Malliavin derivative. 
Assume without loss of generality that $p \geq 2$. By \eqref{eq:Malliavin_der_solution} and  the BDG inequality, we get
\begin{equs}\|\phi(t,x)\|_{\bdot{\mathscr{W}}_p^n}^2 &\lesssim \| p_{t-\cdot}(x,\cdot)\mathscr{D}^{n-1}\sigma(\phi(\cdot,\cdot))\mathbf{1}_{[0,t]}(\cdot)\|_{L_p(\Omega; H^{n})}^2\\
&\qquad\qquad+ \int_{0}^t \int_{\mathbb{T}} p_{t-r}^2(x,y) \|\sigma(\phi(r,y))\|_{\bdot{\mathscr{W}}_p^n}^2 dydr\\
&=: A(t,x) + B(t,x). \label{phiisAplusB}
\end{equs}
We proceed with proving that
\begin{equs}
A(t,x) \lesssim t^{n/2}. \label{Abound}
\end{equs}
 Indeed, in the $n=1$ case, we have by the boundedness of $\sigma$ that
\begin{equs}
A(t,x) = \|p_{t-\cdot}(x,\cdot)\sigma(\phi(\cdot,\cdot))\mathbf{1}_{[0,t]}(\cdot)\|_{L_p(\Omega, H)}^2 \lesssim \|p_{t-\cdot}(x,\cdot)\mathbf{1}_{[0,t]}(\cdot)\|_{H}^2 \lesssim t^{1/2}.
\end{equs}
Moreover in the $n \geq 2$ case,
by point (\ref{fofphi1bound}) of \cref{Dsigmaubuckling} and by the induction  hypothesis we have for $(\theta,\zeta) \in [0,t]\times \mathbb{T}$ that
\begin{align*}
\|  \sigma(\phi(\theta, \zeta))\|_{\bdot{\mathscr{W}}_p^{n-1}} &\lesssim \theta^{(n-1)/4} + \|\phi(\theta,\zeta)\|_{\bdot{\mathscr{W}}_{p}^{n-1}} \lesssim \theta^{(n-1)/4}\lesssim t^{(n-1)/4},
\end{align*}
and thus using Minkowski's inequality and the above bound, we get that
\begin{equs}
A(t,x)   &= \| \| p_{t-\cdot}(x,\cdot) \| \mathscr{D}^{n-1}\sigma(\phi(\cdot,\cdot))\|_{H^{\otimes(n-1)}}\mathbf{1}_{[0,t]}(\cdot)\|_{H}\|_{L_p}^2\\
&\leq \big\| p_{t-\cdot}(x,\cdot)\|\sigma(\phi(\cdot,\cdot))\|_{\bdot{\mathscr{W}}_p^{n-1}}\mathbf{1}_{[0,t]}(\cdot)\big\|_{H}^2\\
&\lesssim t^{(n-1)/2}\| p_{t-\cdot}(x,\cdot)\mathbf{1}_{[0,t]}(\cdot)\|_{H}^2 \lesssim t^{n/2}
\end{equs}
as required. Hence (\ref{Abound}) is proven.
We now proceed by bounding $B$. By point (\ref{fofphi1bound}) of \cref{Dsigmaubuckling} and by the induction hypothesis we have
\begin{equs}
B(t,x) &\lesssim  \int_{0}^t \int_{\mathbb{T}} p_{t-r}^2(x,y)\Big( r^{n/4} + \mathbb\|\phi(r,y)\|_{\bdot{\mathscr{W}}_p^{n}} \Big)^2dydr\\
&\lesssim \int_{0}^t \int_{\mathbb{T}} p_{t-r}^2(x,y) r^{n/2}dydr + \int_{0}^t \int_{\mathbb{T}} p_{t-r}^2(x,y)\|\phi(r,y)\|_{\bdot{\mathscr{W}}_p^{n}}^2 dydr\\
&\lesssim t^{(n+1)/2} + \int_{0}^t \int_{\mathbb{T}} p_{t-r}^2(x,y)\|\phi(r,y)\|_{\bdot{\mathscr{W}}_p^{n}}^2 dydr.\label{Bbound}
\end{equs}
By (\ref{phiisAplusB}), and by our bounds (\ref{Abound}), (\ref{Bbound}) on $A,B$, we conclude that
$$\|\phi(t,x)\|_{\bdot{\mathscr{W}}_p^n}^2 \lesssim t^{n/2} + \int_{0}^t \int_{\mathbb{T}} p_{t-r}^2(x,y)\|\phi(r,y)\|_{\bdot{\mathscr{W}}_p^n}^2 dydr.$$
By Proposition \ref{malldiffability} we have that $\sup_{(r,y) }\|\phi(r,y)\|_{\bdot{\mathscr{W}}_p^n}<\infty.$ Therefore by \cref{Gronwalltypelemma}, the statement we aim to show also holds for the $n$-th Malliavin derivative. Thus the proof is finished.
\end{proof}

\begin{lemma}\label{mallimatrixbound}
Let $n \in \mathbb{Z}_{\geq 0}$, $p \in [1,\infty)$, $\sigma \in C^{n+1}$, and let $\phi$ solve (\ref{driftlessSHE}). There exists some constant $N =N(n,p, \|\sigma\|_{C^{n+1}})$ such that for all $t \in (0,1]$ we have
$$\sup_{x \in \mathbb{T}}\Big\| \| \mathscr{D}\phi(t,x)\|_{H}^2 \Big\|_{\bdot{\mathscr{W}}_p^n} \leq N t^{(n+2)/4}.$$
\end{lemma}
\begin{proof}
By using the Minkowski inequality, the Leibniz rule, H\"older's inequality, and \cref{posimomentbound},
we can see that
\begin{equs}
&\| \| \mathscr{D}\phi(t,x)\|_{H}^2 \|_{\bdot{\mathscr{W}}_p^n}\\ &= \Big\| \Big\| \mathscr{D}^n\int_0^1 \int_{\mathbb{T}}(\mathscr{D}_{\theta,\zeta}\phi(t,x))^2 d\zeta d\theta \Big\|_{H^{\otimes n}}\Big\|_{L_p}\\
& \leq \Big\| \int_0^1 \int_{\mathbb{T}}\|\mathscr{D}^n\big(\mathscr{D}_{\theta,\zeta}\phi(t,x)\mathscr{D}_{\theta,\zeta}\phi(t,x) \big)\|_{H^{\otimes n}} d\zeta d\theta\Big\|_{L_p}\\
&\leq \sum_{i=0}^n \Big\|\int_0^1 \int_{\mathbb{T}}\|\mathscr{D}^i(\mathscr{D}_{\theta,\zeta}\phi(t,x))\mathscr{D}^{n-i}(\mathscr{D}_{\theta,\zeta}\phi(t,x))\|_{H^{\otimes n}} d\zeta d\theta \Big\|_{L_p}\\
& = \sum_{i=0}^n \Big\| \int_0^1 \int_{\mathbb{T}}\|\mathscr{D}^i\mathscr{D}_{\theta,\zeta}\phi(t,x)\|_{H^{\otimes i}}\|\mathscr{D}^{n-i}\mathscr{D}_{\theta,\zeta}\phi(t,x)\|_{H^{\otimes(n-i)}}d\zeta d\theta \Big\|_{L_p}\\
&\leq \sum_{i=0}^n\Big\|\Big(\int_0^1\int_{\mathbb{T}}\|\mathscr{D}^{i}\mathscr{D}_{\theta,\zeta}\phi(t,x)\|_{H^{\otimes i}}^2 d\zeta d\theta\Big)^{1/2}\Big(\int_0^1\int_{\mathbb{T}}\|\mathscr{D}^{n-i}\mathscr{D}_{\bar{\theta},\bar{\zeta}}\phi(t,x)\|_{H^{\otimes (n-i)}
}^2 d\bar{\zeta}
d\bar{\theta}\Big)^{1/2}\Big\|_{L_p}\\
& = \sum_{i=0}^n\Big\| \|\mathscr{D}^{i+1}\phi(t,x)\|_{H^{\otimes (i+1)}} \|\mathscr{D}^{n-i+1}\phi(t,x)\|_{H^{\otimes (n-i+1)}} \Big\|_{L_p}\\
&\leq \sum_{i=0}^n \|\phi(t,x)\|_{\bdot{\mathscr{W}}_{2p}^{i+1}}\|\phi(t,x)\|_{\bdot{\mathscr{W}}_{2p}^{n-i+1}} \lesssim \sum_{i=0}^n
t^{(i+1)/4}t^{(n-i+1)/4} \lesssim t^{(n+2)/4}\end{equs}
as required.
\end{proof}

\subsection{Lipschitz regularity with respect to the initial condition}\label{lipschitznesssection}
For any $z \in C(\mathbb{T})$, let $\phi^{z}$  denote the solution of (\ref{driftlessSHE}) with \(\phi_0=z\). For $n \in \mathbb{N}$, $\sigma \in C^n$, $q \in [1, \infty)$, $(z_1,z_2) \in (C(\mathbb{T}))^2$, $(t,x) \in [0,1]\times \mathbb{T}$, we define 
\begin{equs}
F_{q,n}^{(2)}(t,x,z_1,z_2) &:= \|\phi^{z_1}(t,x) - \phi^{z_2}(t,x)\|_{\bdot{\mathscr{W}}_q^n}, \label{F2def}\\
\Sigma_{q,n}^{(2)}(t,x,z_1,z_2) &:= \|\sigma(\phi^{z_1}(t,x)) - \sigma(\phi^{z_2}(t,x))\|_{\bdot{\mathscr{W}}_q^n} \label{Sigma2def}.
\end{equs}
The main result of this section is the following:
\begin{lemma}\label{lipschitzflowgeneral}
Let $n \in \mathbb{Z}_{\geq 0}$ and assume that $\sigma \in C^{n+1}$.
 For all $q \in [2, \infty)$, $(t,x) \in [0,1]\times \mathbb{T}$, we have that
  $$F_{q,n}^{(2)}(t,x,\cdot, \cdot) \in C((C(\mathbb{T}))^2).$$
  Moreover for all $q,p_1 \in [2, \infty)$ and $p_2 \in [2, \infty]$ there exists a constant $N = N(n,p_1,p_2,q,\|\sigma\|_{C^{n+1}})$ such that for any $\sigma$-algebra $\mathcal{G} \subset \mathcal{F}$ and any random variables $Z,\bar{Z} \in L_p(\Omega, C(\mathbb{T}))$ and for all $(t,x) \in [0,1]\times \mathbb{T}$ we have
$$\big\|F_{q,n}^{(2)}(t,x,Z,\bar{Z})\big\|_{L^{\mathscr{G}}_{p_1,p_2}}\leq N t^{n/4}\sup_{x \in \mathbb{T}}\|Z(x)-\bar{Z}(x)\|_{L_{p_1,p_2}^{\mathscr{G}}}.$$
\end{lemma}
\begin{proof} The result will be proven by induction.

\underline{Step 1:} We show that the statement holds for $n=0$.
For $z,\bar{z} \in C(\mathbb{T})$, we have by (\ref{integralSHE}) that
\begin{equs}
\phi^z(t,x) - \phi^{\bar{z}}(t,x) = P_t(z-\bar{z})(x) + \int_0^t\int_{\mathbb{T}} p_{t-r}(x,y)(\sigma(\phi^z(r,y)) - \sigma(\phi^{\bar{z}}(r,y))\xi(dy,dr).
\end{equs}
Therefore, by the BDG inequality we get
\begin{equs}&\| \phi^z(t,x) - \phi^{\bar{z}}(t,x)\|_{L_q}^2\\
&\qquad\lesssim |P_{t}(z-\bar{z})(x)|^2 + \int_0^t \int_{\mathbb{T}} p_{t-r}^2(x,y)\| \phi^z(r,y) - \phi^{\bar{z}}(r,y)\|_{L_q}^2 dy dr. \label{gzerobuckling}\end{equs}
 Notice that by the triangle inequality and Proposition \ref{malldiffability}, it follows that 
the norm in the  integrand is bounded in $(r,y)$.  Hence
 using \cref{Gronwalltypelemma}, we conclude that
\begin{equs}F_{q,0}^{(2)}(t,x,z,\bar{z}) =\|\phi^{z}(t,x) - \phi^{\bar{z}}(t,x)\|_{L_q} &\lesssim \sup_{(t,x) \in [0,1]\times \mathbb{T}}|P_{t}(z-\bar{z})(x)|\\
&\lesssim \sup_{x \in \mathbb{T}}|z(x) - \bar{z}(x)|. \label{deterministiclipscitzg0}\end{equs}
By the triangle inequality, it follows that $F^{(2)}_{q,0}(t,x,\cdot,\cdot) \in C((C(\mathbb{T}))^2, \mathbb{R}).$ This implies that $F_{q,0}^{(2)}(t,x,Z,\bar{Z})$ is indeed defined as a random variable.
We begin by showing that the desired inequality holds for the case when $\|Z\|_{\mathbb{B}},\|\bar{Z}\|_{\mathbb{B}} < N$ almost surely, for some constant $N<\infty$.
By evaluating (\ref{gzerobuckling}) at $(z,\bar{z}) = (Z,\bar{Z})$, and then taking the $L_{p_1,p_2}^{\mathscr{G}}$-norm of the square root of both sides, we get
\begin{equs}
&\|F_{q,0}^{(2)}(t,x,Z,\bar{Z})\|_{L_{p_1,p_2}^{\mathscr{G}}}\\
&\lesssim \|P_{t}(Z-\bar{Z})(x)\|_{L_{p_1,p_2}^{\mathscr{G}}}+ \Big\|\int_0^t \int_{\mathbb{T}} p_{t-r}^2(x,y)|F_{q,0}^{(2)}(r,y,Z,\bar{Z})|^2 dydr\Big\|_{L_{\frac{p_1}{2},\frac{p_2}{2}}^{\mathscr{G}}}^{1/2}\\
&\lesssim \sup_{w \in \mathbb{T}}\|Z(w)-\bar{Z}(w)\|_{L_{p_1, p_2}^\mathscr{G}} + \Big(\int_0^t \int_{\mathbb{T}} p_{t-r}^2(x,y)\|F_{q,0}^{(2)}(r,y,Z,\bar{Z})\|_{L_{p_1,p_2}^{\mathscr{G}}}^2 dydr\Big)^{1/2},
\end{equs}
where to obtain the last expression, we used Minkowski's inequality.
Hence we may conclude that
\begin{equs}\|F_{q,0}^{(2)}(t,x,Z,\bar{Z})\|_{L_{p_1,p_2}^{\mathscr{G}}}^2 &\lesssim \Big(\sup_{w \in \mathbb{T}}\|Z(w)-\bar{Z}(w)\|_{L_{p_1,p_2}^{\mathscr{G}}}\Big)^2\\
&\qquad\qquad+ \int_0^t \int_{\mathbb{T}} p_{t-r}^2(x,y)\|F_{q,0}^{(2)}(r,y,Z,\bar{Z})\|_{L_{p_1,p_2}^{\mathscr{G}}}^2 dydr.
\end{equs}
Note that by (\ref{deterministiclipscitzg0}) and due to the fact that $\|Z\|_{\mathbb{B}},\|\bar{Z}\|_{\mathbb{B}}\leq N $, we have
$\sup_{r,y}\|F_{q,0}^{(2)}(r,y,Z,\bar{Z})\|_{L_{p_1,p_2}^{\mathscr{G}}} \lesssim 2N $.
Hence by the above inequality and by \cref{Gronwalltypelemma} we get
$$\|F^{(2)}_{q,0}(t,x,Z,\bar{Z})\|_{L_{p_1,p_2}^{\mathscr{G}}} \lesssim \sup_{w \in \mathbb{T}}\|Z(w) - \bar{Z}(w)\|_{L_{p_1,p_2}^{\mathscr{G}}}.$$
We now remove the assumption that $\|Z\|_{\mathbb{B}},\|\bar{Z}\|_{\mathbb{B}}<N$. Indeed, if this does not hold, then we can construct the sequence of truncations $(Z_N,{\bar{Z}}_N) := ((N\wedge Z) \vee (-N), (N\wedge \bar{Z}) \vee (-N))$. Then using the continuity of $F_{q,0}^{(2)}$ and Fatou's lemma, the above inequality, and the fact that truncation is a Lipschitz operation, we get
\begin{equs}
\|F_{q,0}^{(2)}(t,x,Z,\bar{Z})\|_{L_{p,\infty}^{\mathscr{G}}} &\leq \liminf_{N \to \infty} \|F_{q,0}^{(2)}(t,x,Z_N,\bar{Z}_N)\|_{L_{p_1,p_2}^{\mathscr{G}}}\\
&\lesssim \liminf_{N \to \infty} \sup_{w \in \mathbb{T}}\|Z_N(w) - \bar{Z}_N(w)\|_{L_{p_1,p_2}^{\mathscr{G}}}\\
&\leq \sup_{w \in \mathbb{T}}\|Z(w) - \bar{Z}(w)\|_{L_{p_1,p_2}^{\mathscr{G}}}.
\end{equs}
This finishes the proof of the case $n=0$.

\underline{Step 2:}
Suppose that the result holds for $F_{q,0}^{(2)},\dots, F_{q,n-1}^{(2)}$ for some $n \in \mathbb{N}$. We aim to show that it also holds for $F_{q,n}^{(2)}$. We  first assume that for all $(t,x)\in [0,1]\times \mathbb{T}$, $F_{q,n}^{(2)}(t,x,z,\bar{z})$ is continuous in the $(z,\bar{z})$ variable and later we will show that this is indeed the case. By Proposition \ref{CHN3}, we have
\begin{equs}
F^{(2)}_{q,n}(t,x,z,\bar{z}) &=\|\mathscr{D}^n[\phi^z(t,x) - \phi^{\bar{z}}(t,x)]\|_{L_q(\Omega, H^{\otimes n})}\\
&\lesssim \Big\| p_{t-\cdot}(x,\cdot)\mathscr{D}^{n-1}\big( \sigma(\phi^z(\cdot,\cdot)) - \sigma(\phi^{\bar{z}}(\cdot,\cdot))\big)\mathbf{1}_{[0,t]}\Big\|_{L_q(\Omega; H^{\otimes n})}\\
&\qquad+ \Big\|\int_0^t \int_{\mathbb{T}} p_{t-r}(x,y) \mathscr{D}^n\big( \sigma(\phi^z(r,y)) - \sigma(\phi^{\bar{z}}(r,y))\big)\xi(dy,dr)\Big\|_{L_q(\Omega; H^{\otimes n})}\\
& =: A(z,\bar{z})+B(z,\bar{z}). \label{FisboundedbyAplusB}
\end{equs}
  By \cref{posimomentbound} we may apply point (\ref{fphi1minusfphi2bound}) of \cref{Dsigmaubuckling} with $\varepsilon = s^{1/4}$, to see that for any $m \leq n$ and $(s,y) \in [0,1]\times \mathbb{T}$ we have
\begin{equs}\Sigma_{q,m}^{(2)}(s,y,Z,\bar{Z})&=\|\sigma(\phi^z(s,y)) - \sigma(\phi^{\bar{z}}(s,y))\|_{\bdot{\mathscr{W}}_q^m}\\
&\lesssim \sum_{i=0}^{m-1} s^{(m-i)/4}\|\phi^z(s,y) - \phi^{\bar{z}}(s,y)\|_{\bdot{\mathscr{W}}_{2q}^i} + \|\phi^z(s,y)-\phi^{\bar{z}}(s,y)\|_{\bdot{\mathscr{W}}_q^m}.
\end{equs}
Thus by using the induction  hypothesis and the definition of $F^{(2)}$, we can see that
\begin{equs}
\| \Sigma_{q,m}^{(2)}(s,y,Z,\bar{Z})\|_{L_{p_1,p_2}^{\mathscr{G}}} \lesssim  s^{ m/4}\sup_{y \in \mathbb{T}}\|Z(y) - \bar{Z}(y)\|_{L_{p_1,p_2}^{\mathscr{G}}}+ \|F_{q,m}^{(2)}(s,y,Z,\bar{Z})\|_{L_{p_1,p_2}^{\mathscr{G}}}. \label{diffofsigmaus}
\end{equs}
We can bound $A$ by applying this result (with $m:=n-1$) as follows:
\begin{equs}
&\|A(Z,\bar{Z})\|_{L_{p_1,p_2}^{\mathscr{G}}}\\
&\lesssim\big\|p_{t-\cdot}(x,\cdot)\big\|\Sigma_{q,n-1}^{(2)}(\cdot,\cdot,Z,\bar{Z})\big\|_{L_{p_1,p_2}^{\mathscr{G}}} \mathbf{1}_{[0,t]}\big\|_{H}\\
&\lesssim\Big\| p_{t-\cdot}(x,\cdot)\sup_{(\theta,\zeta) \in [0,t]\times \mathbb{T}}\Big( \theta^{\frac{n-1}{4}}\sup_{w \in \mathbb{T}}\|Z(w) -\bar{Z}(w)\|_{L_{p_1,p_2}^{\mathscr{G}}} + \|F_{q,n-1}^{(2)}(\theta,\zeta,Z,\bar{Z})\|_{L_{p_1,p_2}^{\mathscr{G}}} \Big)\mathbf{1}_{[0,t]}\Big\|_H\\
&\lesssim t^{n/4}\sup_{w \in \mathbb{T}}\|Z(w) -\bar{Z}(w)\|_{L_{p_1,p_2}^\mathscr{G}} \label{Aboundforlipschitzness}
\end{equs}
where for the last inequality we used the induction  hypothesis. Now we also bound $B$. To this end, note that by the BDG inequality  we have
\begin{align*}
B(z,\bar{z})
&\lesssim \Big(\int_0^t \int_{\mathbb{T}}p_{t-r}^2(x,y) \|\sigma(\phi^z(r,y)) - \sigma(\phi^{\bar{z}}(r,y))\|_{\bdot{\mathscr{W}}_q^n}^2 dydr\Big)^{1/2},
\end{align*}
and thus using (\ref{diffofsigmaus}) we get
\begin{equs}&\|B(Z,\bar{Z})\|_{L_{p_1,p_2}^{\mathscr{G}}}^2\\
&\lesssim \int_0^t \int_{\mathbb{T}} p_{t-r}^2(x,y)\big\|\Sigma_{q,n}^{(2)}(r,y,Z,\bar{Z})\big\|_{L_{p_1,p_2}^{\mathscr{G}}}^2 dydr\\
&\lesssim \int_0^t \int_{\mathbb{T}}p_{t-r}^2(x,y)\Big(t^{n/4}\sup_{w \in \mathbb{T}}\|Z(w) - \bar{Z}(w)\|_{L_{p_1,p_2}^{\mathscr{G}}}+\|F_{q,n}^{(2)}(r,y,Z,\bar{Z})\|_{L_{p_1,p_2}^{\mathscr{G}}} \Big)^2dydr\\
&\lesssim t^{(n+1)/2}\sup_{w \in \mathbb{T}}\|Z(w) - \bar{Z}(w)\|_{L_{p_1,p_2}^{\mathscr{G}}}^2  + \int_0^t \int_{\mathbb{T}} p_{t-r}^2(x,y)\|F_{q,n}^{(2)}(r,y,Z,\bar{Z})\|_{L_{p_1,p_2}^{\mathscr{G}}}^2 dydr. \quad\label{Bboundforlipschitzness}
\end{equs}
By putting the bounds (\ref{Aboundforlipschitzness}) and (\ref{Bboundforlipschitzness}) into (\ref{FisboundedbyAplusB}) we can see that
\begin{equs}
\|F_{q,n}^{(2)}(t,x,Z,\bar{Z})\|_{L_{p_1,p_2}^{\mathscr{G}}}^2 &\lesssim t^{n/2}\sup_{w \in \mathbb{T}}\|Z(w) - \bar{Z}(w)\|_{L_{p_1,p_2}^{\mathscr{G}}}^2\\
&\qquad+ \int_0^t \int_{\mathbb{T}} p_{t-r}^2(x,y)\|F_{q,n}^{(2)}(r,y,Z,\bar{Z})\|_{L_{p_1,p_2}^{\mathscr{G}}}^2 dydr. \label{lipschitznessbucklingforgn}
\end{equs}
If $(Z,\bar{Z}) = (z,\bar{z}) \in (C(\mathbb{T}))^2$ is deterministic, then we may repeat the proof without assuming that $F_{q,n}^{(2)}(t,x,z,\bar{z})$ is continuous in $(z,\bar{z})$. The above inequality then simply states that
$$|F_{q,n}^{(2)}(t,x,z,\bar{z})|^2 \lesssim t^{n/2}\sup_{w \in \mathbb{T}}|z(w) - \bar{z}(w)|^2 + \int_0^t \int_{\mathbb{T}} p^2_{t-r}(x,y)|F_{q,n}^{(2)}(r,y,z,\bar{z})|^2dydr$$
Note that by the definition of $F^{(2)}$ and by Proposition \ref{malldiffability} we have
$$\sup_{(t,x)\in[0,1]\times \mathbb{T}}|F_{q,n}^{(2)}(t,x,z,\bar{z})| \leq \sup_{(t,x) \in [0,1]\times \mathbb{T}}\|\phi^z(t,z)\|_{\bdot{\mathscr{W}}_q^n} + \sup_{(t,x) \in [0,1]\times \mathbb{T}}\|\phi^{\bar{z}}(t,z)\|_{\bdot{\mathscr{W}}_q^n} <\infty,$$
so using \cref{Gronwalltypelemma} we get
$$\|\phi^z(t,x) - \phi^{\bar{z}}(t,x)\|_{\bdot{\mathscr{W}}_q^n} =|F^{(2)}_{q,n}(t,x,z,\bar{z})| \lesssim t^{n/4}\sup_{x \in \mathbb{T}}|z(x) -\bar{z}(x)| \lesssim \sup_{x \in \mathbb{T}}|z(x) - \bar{z}(x)|.$$
From this it easily follows that for all $(t,x) \in [0,1]\times \mathbb{T}$ the map $F_{q,n}^{(2)}(t,x,\cdot,\cdot)$ is of class $C((C(\mathbb{T}))^2, \mathbb{R})$.
Now going back to (\ref{lipschitznessbucklingforgn}), if $\|Z\|_{\mathbb{B}}, \|\bar{Z}\|_{\mathbb{B}} \leq N$ for some given $N >0$ then the desired result follows by \cref{Gronwalltypelemma}.
For the general case we can repeat the truncation argument from the $n=0$ case to finish the proof.
\end{proof}

\subsection{Nondegeneracy of convex combinations of solutions}
\label{nondegeneracysection}
Throughout the section we assume that $\sigma \in C^1$ such that there exists a constant $\mu>0$ such that for all $x \in \mathbb{R}$ we have $\sigma^2(x) \geq \mu^2$. 
Let $\phi^1,\dots, \phi^K$ solve the driftless equation (\ref{driftlessSHE}) with initial conditions $\phi^1_0,\dots, \phi^K_0$ respectively. Consider the convex combination
\begin{equs}\label{convexcombo}
\Theta(t,x) := \sum_{i=1}^K c_i \phi^i(t,x).
\end{equs}
 with
$\sum_{i=1}^K c_i =1$ and $ c_1,\dots, c_K \in [0,1]$.
For a smooth map $g$ and a nonnegative integer $n$, we aim to obtain estimates on the expectation of $\nabla^n g(\Theta(t,x))$ which depend only on a Besov--H\"older norm of $g$ with a negative index, see \cref{mainmalliavintool} below. 

 The following lemma quantifies  Theorem 4.5 in the chapter by Nualart in \cite{dalang2009minicourse}.
\begin{lemma}\label{negativemoments} For any $p \in (2,\infty)$ there exists some constant $N = N(p,\|\sigma\|_{C^1}, \mu)$, such that for all $t \in [0,1]$ we have
$$\sup_{x \in \mathbb{T}}\mathbb{E}\| \mathscr{D}\Theta(t,x)\|_H^{-p} \leq N t^{-p/4}.$$
\end{lemma}

\begin{proof}
By Proposition \ref{CHN3} we have for $(t,x), (\theta,\zeta) \in [0,1]\times \mathbb{T}$ that
\begin{align*}
\mathscr{D}_{\theta,\zeta} \Theta(t,x) &= \mathbf{1}_{[0,t]}(\theta)p_{t-\theta}(x,\zeta)\sum_{i=1}^K c_i \sigma(\phi^i(\theta,\zeta))\\
&\qquad+ \mathbf{1}_{[0,t]}(\theta)\int_0^t \int_{\mathbb{T}} p_{t-r}(x,y)\Big(\sum_{i=1}^K c_i \sigma'(\phi^i(r,y)) \mathscr{D}_{\theta,\zeta} \phi^i(r,y) \Big)\xi(dy,dr)\\
& =: A(\theta,\zeta)+B(\theta,\zeta).
\end{align*}
From this we can see that
\begin{align*}\int_0^t \int_{\mathbb{T}} | \mathscr{D}_{\theta,\zeta} \Theta(t,x)|^2 d\zeta d\theta &\geq \frac{1}{2}\int_{t-\delta}^t \int_{\mathbb{T}} | A(\theta,\zeta)|^2 d\zeta d\theta  - \int_{t-\delta}^t \int_{\mathbb{T}} | B(\theta,\zeta)|^2 d\zeta d\theta\\
& =: I_A^\delta - I_B^\delta.
\end{align*}
So since $|A(\theta,\zeta)| \geq \mathbf{1}_{[0,t]}(\theta)p_{t-\theta}(x,\zeta)\mu$, and by the properties of the heat kernel, it follows that
\begin{equs}I_A^\delta = \frac{1}{2}\int_{t-\delta}^t \int_{\mathbb{T}}|A(\theta,\zeta)|^2 d\zeta d\theta  &\geq \mu^2 \int_{t-\delta}^t \int_{\mathbb{T}}|p_{t-\theta}(x,\zeta)|^2 d\zeta d \theta
\geq  k\mu^2 \delta^{1/2}
\end{equs}
for some universal constant $k>0$. Thus
$$I_A^\delta \geq c_0 \delta^{1/2} \quad \text{with} \quad c_0   = k \mu^2.$$
Therefore for $\varepsilon \in (0, c_0\delta^{1/2})$ we have
\begin{equs}
\mathbb{P}\Big(\int_0^t \int_{\mathbb{T}} | \mathscr{D}_{\theta,\zeta} \Theta(t,x)|^2 d\zeta d\theta < \varepsilon \Big)
&\leq \mathbb{P}(I_A^\delta - I_B^\delta <\varepsilon)\\
&\leq  \mathbb{P}(I_B^\delta >c_0\delta^{1/2} - \varepsilon)\\
&\leq (c_0\delta^{1/2} - \varepsilon)^{-p} \mathbb{E}|I_B^\delta|^p
\label{nualartstylebound}
\end{equs}
where the last inequality holds by Markov's inequality. We will now need to bound the expectation in the last line. Note that
\begin{align*}
\mathbb{E}|I_B^\delta|^p &= \mathbb{E}\Big|\int_{t-\delta}^t \int_{\mathbb{T}} \Big| \int_0^t \int_{\mathbb{T}} p_{t-r}(x,y)\Big(\sum_{i=1}^K c_i \sigma'(\phi^i(r,y)) \mathscr{D}_{\theta,\zeta} \phi^i(r,y) \Big)\xi(dy,dr)\Big|^2 d\zeta d\theta\Big|^p\\
&= \mathbb{E}\Big\|\int_0^t \int_{\mathbb{T}} p_{t-r}(x,y)\Big(\sum_{i=1}^K c_i \sigma'(\phi^i(r,y)) \mathscr{D} \phi^i(r,y) \Big)\xi(dy,dr)\Big\|_{L_2([t-\delta,t]\times \mathbb{T})}^{2p}\\
&\lesssim \mathbb{E}\Big(\int_{t-\delta}^t \int_{\mathbb{T}}|p_{t-r}(x,y)|^2 \sum_{i=1}^K \| \mathscr{D} \phi^i(r,y)\|_{L_2([t-\delta,t]\times \mathbb{T})}^2 dydr\Big)^{p}
\end{align*}
where we used the BDG inequality, and the fact that $\| \mathscr{D} \phi^i(r,y)\|_{L_2([t-\delta,t]\times \mathbb{T})} =0$ for $r<t-\delta$. Noting that $r-\delta<t-\delta$, and that $\mathscr{D}_{\theta,\zeta} \phi^i(r,y) =0$ for $\theta>r$, we may bound the $L_2([t-\delta,t]\times \mathbb{T})$ norm in the expression by the $L_2([r-\delta,r]\times \mathbb{T})$-norm, and write
\begin{equs}
\mathbb{E}|I_B^\delta|^p &\lesssim \sum_{i=1}^K \mathbb{E}\Big(\int_{t-\delta}^t \int_{\mathbb{T}} |p_{t-r}(x,y)|^2 \| \mathscr{D} \phi^i(r,y)\|_{L_2([r-\delta,r]\times \mathbb{T})}^2 dydr\Big)^{p} \nonumber\\
&\leq \delta^{p/2 -1}\sum_{i=1}^K \int_{t-\delta}^t \int_{\mathbb{T}} p_{t-r}(x,y) \mathbb{E}\| \mathscr{D} \phi^i(r,y) \|_{L_2([r-\delta , r]\times \mathbb{T})}^{2p} dydr \nonumber\\
&=: \delta^{p/2-1}\sum_{i=1}^K G_i\label{sumofFG}
\end{equs}
where we used \cref{holdertrick} with $\gamma = 2$ and $\delta =1$.
To bound $G_i$, we will need to bound $\| \mathscr{D} \phi^i(r,y)\|_{L_2([r-\delta, r]\times \mathbb{T})}.$ To this end, note that for $(\theta,\zeta) \in [0,t]\times \mathbb{T}$ we have
\begin{equs}\mathscr{D} \phi^i(t,x) &= p_{t-\theta}(x,\zeta)\sigma(\phi^i(\theta,\zeta))+ \int_0^t \int_{\mathbb{T}} p_{t-r}(x,y) \sigma'(\phi^i(r,y)) \mathscr{D}_{\theta,\zeta}\phi^i(r,y) \xi(dy,dr).
\end{equs}
Therefore using the BDG inequality, we get
\begin{align*}
\mathbb{E}\| \mathscr{D} \phi^i(t,x)\|_{L_2([t-\delta , t] \times \mathbb{T})}^q &\lesssim \|p_{t-\cdot}(x,\cdot)\|_{L_2([t-\delta , t] \times \mathbb{T})}^q \\
&\qquad+ \mathbb{E}\Big(\int_0^t \int_{\mathbb{T}} |p_{t-r}(x,y)|^2 \| \mathscr{D} \phi^i(r,y)\|_{L_2([t-\delta , t] \times \mathbb{T})}^2 dydr\Big)^{q/2}\\
&=: \bar{A} + \bar{B}.
\end{align*}
We have $\bar{A} \lesssim \delta^{q/4}$. Moreover by applying \cref{holdertrick} with $\gamma = \delta =2$ to $\bar{B}$ and noting again that $\|\mathscr{D} \phi^i(r,y)\|_{L_2([t-\delta,t]\times \mathbb{T})} \leq \|\mathscr{D} \phi^i(r,y)\|_{L_2([r-\delta,r]\times \mathbb{T})}$ we get
\begin{align*}
\bar{B}&\lesssim t^{(p-1)/2}\int_0^t \int_{\mathbb{T}} |p_{t-r}(x,y)|^2 \mathbb{E}\| \mathscr{D} \phi^i(r,y)\|_{L_2([r-\delta , r] \times \mathbb{T})}^{q} dydr
\end{align*}
for $q>2$.
Therefore we obtain
$$\mathbb{E}\| \mathscr{D} \phi^i(t,x)\|_{L_2([r-\delta , r] \times \mathbb{T})}^q \lesssim \delta^{q/4} + \int_0^t \int_{\mathbb{T}} |p_{t-r}(x,y)|^2 \mathbb{E}\| \mathscr{D} \phi^i(r,y)\|_{L_2([r-\delta , r] \times \mathbb{T})}^{q} dydr.$$
By \cref{malldiffability} we can see that the $q$-th moment in the integrand is bounded in $(r,y)$. 
Hence by \cref{Gronwalltypelemma} it follows that
$$\mathbb{E}\| \mathscr{D} \phi^i(t,x)\|_{L_2([r-\delta,r]\times \mathbb{T})}^q \lesssim \delta^{q/4}.$$
Applying this with $q =2p$ to bound $G_i$, we get
\begin{equs}
G_i \lesssim \int_{t-\delta}^t \int_{\mathbb{T}}p_{t-r}(x,y) \delta^{p/2}dydr
\lesssim \delta^{p/2+1}.\label{Gboundconvexcombo}
\end{equs}
Now putting (\ref{Gboundconvexcombo}) into (\ref{sumofFG}), we get
$\mathbb{E}|I_B^\delta|^p \lesssim \delta^p$.
Putting this into (\ref{nualartstylebound}) we see that for all $\delta \in [0, t]$ and all $\eps \in (0, c_0\delta^{1/2})$, we have
$$\mathbb{P}(\|\mathscr{D} \Theta(t,x)\|_{H}^2 <\varepsilon) \lesssim (c_0 \delta^{1/2}- \varepsilon)^{-p}\delta^{p}.$$
So if $\eps \in (0, (c_0/2) \sqrt{t})$, we can choose 
$\delta(\varepsilon) := \frac{4}{c_0^2}\varepsilon^2$, 
to get
$$\mathbb{P}(\|\mathscr{D} \Theta(t,x)\|_{H}^2 <\varepsilon) \lesssim  \varepsilon^p.$$
Let $L:= (2/c_0)^{p/2} t^{-p/4}$. Notice that if $\gamma>L$, then $\gamma^{-2/p} \in  (0, (c_0/2) \sqrt{t})$, and consequently we have 
$$\mathbb{P}(\|\mathscr{D} \Theta(t,x)\|_{H}^2 < \gamma^{-2/p}) \lesssim \gamma^{-2} .$$
Hence, we have
\begin{equs}
\mathbb{E}\|\mathscr{D}\Theta(t,x)\|_{H}^{-p} = \int_0^\infty \mathbb{P}\big(\|\mathscr{D} \Theta\|_H^{-p} \geq \gamma\big) d\gamma & \leq  L+ \int_L^\infty \mathbb{P}(\|\mathscr{D}\Theta\|_H^2  < \gamma^{-2/p}) d\gamma\\
&\lesssim L + \int_{L}^\infty \gamma^{-2} d\gamma\\
&\lesssim L + L^{-1} \lesssim t^{-p/4},
\end{equs}
which finishes the proof. 
\end{proof}

 For $(t,x) \in [0,1]\times \mathbb{T}$, we consider the $H$-valued random variables $w_{t,x}$ which are given for all $(\theta,\zeta) \in [0,1]\times \mathbb{T}$ by
$$w_{t,x}(\theta,\zeta) := \frac{\mathscr{D}_{\theta,\zeta}\Theta(t,x)}{\|\mathscr{D}\Theta(t,x)\|_{H}^2}.$$ 
For given $n \in \mathbb{N}$ and for a $C([0,1]\times \mathbb{T})$-valued random variable $G_0$ such that for all $(t,x) \in [0,1]\times \mathbb{T}$ $G_0(t,x) \in \mathscr{W}^n$, we may define iterated Skorokhod integrals for all $k \in \{0,\dots, n-1\}$ and $(t,x) \in [0,1]\times \mathbb{T}$ recursively by
\begin{align*}
G_{k+1}(t,x) &= \delta(w_{t,x}  G_k(t,x)).
\end{align*}
Then by \cref{repeatedMalliavinbyparts}, for any $f \in C^\infty$ we have the integration-by-parts formula
$$\mathbb{E}\big(\nabla^k f(\Theta) G_0\big) = \mathbb{E}\big(f(\Theta)G_k\big)$$
To bound the iterations $(G_k)_{k \in \{0,\dots, n\}}$, we will need the following bounds on $w$ and its Malliavin derivatives.
\begin{lemma}\label{wbounds}
Let $p \in [1,\infty)$, $n \in \mathbb{Z}_{\geq 0}$, and  $\sigma \in C^{n+1}$. Then there exists a constant $N = N(p,n, \|\sigma\|_{C^{n+1}}, \mu)$ such that for all $(t,x) \in [0,1]\times \mathbb{T}$, we have
$$\|  \mathscr{D}^n w_{t,x}\|_{L_p(\Omega;H^{\otimes (n+1)})}  \leq N t^{(n-1)/4}.$$
\end{lemma}
\begin{proof}
Fix $(t,x) \in [0,1]\times \mathbb{T}$ and let $X := \Theta(t,x)$, $Y := \|\mathscr{D}\Theta(t,x)\|_H^2$ and $w := (\mathscr{D}X)/Y$. We may assume that $p>2$.
By \cref{posimomentbound}, \cref{mallimatrixbound} and \cref{negativemoments} respectively, we have
$$\|X\|_{\bdot{\mathscr{W}}_p^k} \lesssim t^{\frac{k}{4}}, \qquad \|Y\|_{\bdot{\mathscr{W}}_p^k} \lesssim t^{\frac{2+k}{4}}, \qquad \mathbb{E}(Y)^{-p} \lesssim t^{-\frac{2p}{4}}.$$
Therefore by \cref{DXperYlemma} (with $\varepsilon := t^{1/4}$ and $c = 2$) to obtain
$$\|\mathscr{D}^n w\|_{L_p(\Omega; H^{\otimes (n+1)}} \lesssim (t^{1/4})^{n+1 -2} \lesssim t^{(n-1)/4}$$
as required.
\end{proof}

\begin{lemma}\label{iteratedskorokhod}
Let $n \in \mathbb{Z}_{\geq 0}$, and $\sigma \in C^n$. Then for each $k,m \in \mathbb{Z}_{\geq 0}$ such that $k+m \leq n$ and for all $p \in [1,\infty)$ there exists a constant $N = N(k,m,p, \| \sigma\|_{C^n}, \mu)$  such that with $q := 2^m p$ we have for all $(t,x) \in [0,1]\times \mathbb{T}$ that
\begin{equs}\| G_m(t,x) \|_{\mathscr{W}_p^k} \leq Nt^{-m/4}\|G_0(t,x)\|_{\mathscr{W}_{q}^{k+m}}.
\end{equs}
\end{lemma}
\begin{proof}
For notational convenience, fix $(t,x) \in [0,1]\times \mathbb{T}$, set $w := w_{t,x}$, and for $i=1,\dots,n$ set $G_i:= G_i(t,x)$. The proof will be done by induction with respect to the $m$ variable. For $m = 0$ the statement is obviously true. Now suppose that the statement is true for some $m \leq n-1$. That is, we suppose that for all $l \in \mathbb{Z}_{\geq 0}$ such that $l+m \leq n$ we have
$$\|G_m\|_{\mathscr{W}_p^l} \lesssim t^{-m/4}\|G_0\|_{\mathscr{W}_{2^m p}^{l+m}}.$$
We show that the statement is also true for $m+1$, i.e. that for all $k \in \mathbb{Z}_{\geq 0}$ such that $k+(m+1) \leq n$, we have
\begin{equs}\|G_{m+1}\|_{\mathscr{W}_p^k} \lesssim t^{-(m+1)/4}\|G_0\|_{\mathscr{W}_{2^{m+1}p}^{k+m+1}}.
\label{Gmplus1bound}
\end{equs}
Let $k \in \mathbb{Z}_{\geq 0}$, such that $k+m+1 \leq n$.
Since the divergence $\delta: \mathscr{W}_p^{k+1} \to \mathscr{W}_p^k$ is a bounded linear operator (see
\cite[Proposition 1.5.7 and point 1 of remarks of Chapter 1]{nualart2006malliavin}), we have
\begin{equs}
\| G_{m+1}\|_{\mathscr{W}_p^k} = \| \delta(w G_{m}) \|_{\mathscr{W}_p^k} &\lesssim \| w G_{m} \|_{\mathscr{W}_p^{k+1}(H)} \\ &\lesssim   \sum_{i=0}^{k+1}  \| \mathscr{D}^i(wG_{m}) \|_{L_p(\Omega;H^{\otimes (i+1)})} \\
 &\lesssim \sum_{i=0}^{k+1}\sum_{\lambda_1 + \lambda_2 = i} \| \mathscr{D}^{\lambda_1}w\|_{L_{2p}(\Omega; H^{\otimes(\lambda_1 +1)})}\|\mathscr{D}^{\lambda_2} G_m\|_{L_{2p}(\Omega; H^{\otimes \lambda_2})} .\label{sumofprodsofwandG}
 \end{equs}
By \cref{wbounds}, and since $\lambda_1\geq 0$, we have 
\begin{equs}\| \mathscr{D}^{\lambda_1}w\|_{L_{2p}(\Omega; H^{\otimes(\lambda_1 +1)})} \lesssim t^{(\lambda _1 -1)/4} \lesssim t^{-1/4}. \label{DwboundforGm}
\end{equs}
Moreover since $\lambda_2 + m \leq k+1+m \leq n$, by the induction  hypothesis we have
\begin{equs}
\|\mathscr{D}^{\lambda_2} G_m\|_{L_{2p}(\Omega; H^{\otimes \lambda_2})} \leq \|G_m\|_{\mathscr{W}_{2p}^{\lambda_2}} \lesssim t^{-m/4}\|G_0\|_{\mathscr{W}_{2^{m +1} p}^{\lambda_2+m}} \leq t^{-m/4}\|G_0\|_{\mathscr{W}_{2^{m+1}p}^{k+1+m}} .\label{DGmbound}
\end{equs}
Now putting (\ref{DwboundforGm}) and (\ref{DGmbound}) into (\ref{sumofprodsofwandG}), we get
$$\|G_{m+1}\|_{\mathscr{W}_p^k} \lesssim t^{-1/4}t^{-m/4}\|G_0\|_{\mathscr{W}_{2^{m+1}p}^{k+1+m}}.$$
Hence (\ref{Gmplus1bound}) holds, and the proof is finished.
\end{proof}

\begin{lemma}\label{mainmalliavintool} Let $n \in \mathbb{Z}_{\geq 0}$, $\sigma \in C^{n+1}$, $\beta \in (-2,-1)\cup(-1,0)$, and set $q:= 2^{n+2}\mathbf{1}_{(-1,0)}(\beta) + 2^{n+3}\mathbf{1}_{(-2,-1)}(\beta) $ and $m :=  (n+1)\mathbf{1}_{(-1,0)}(\beta) + (n+2)\mathbf{1}_{(-2,-1)}(\beta)$.  There exists a constant $N = N(n,\beta, \|\sigma\|_{C^{n+1}}, \mu)$ such that for all $g \in C^\infty$, $(t,x) \in [0,1]\times \mathbb{T}$, we have
$$\big|\mathbb{E}\big(\nabla^n g(\Theta(t,x)) G_0(t,x)\big)\big| \leq  N\|g\|_{C^\beta} t^{\frac{\beta-n}{4}}\|G_0(t,x)\|_{\mathscr{W}_{q}^{m}}.$$
\end{lemma}
\begin{proof}
 Let $f \in C^\infty$ be the solution of
$(1-\Delta)f = g$ and let $(t,x) \in [0,1]\times \mathbb{T}$.
By  \cref{repeatedMalliavinbyparts} and by the definition of $f$, we get
\begin{equs}
|\mathbb{E}(\nabla^n g(\Theta(t,x)) G_0(t,x))| &= |\mathbb{E}(  g(\Theta(t,x)) G_n(t,x))|\\
&\leq | \mathbb{E}(f(\Theta(t,x))G_n(t,x))| +|\mathbb{E}(\Delta f(\Theta(t,x))G_n(t,x))| =:A+B.
\end{equs}
It follows easily from \cref{iteratedskorokhod} and Proposition \ref{laplacianlemma} that 
\begin{equs}A \leq \|f\|_{\mathbb{B}}\|G_n(t,x)\|_{L_1} \leq \|f\|_{C^{2+\beta}}t^{-n/4}\|G_0(t,x)\|_{\mathscr{W}_{2^n}^n} \lesssim \|g\|_{C^{\alpha}} t^{-n/4 + \beta/4}\|G_0(t,x)\|_{\mathscr{W}_{q}^{m}}.\quad\quad
\label{Aboundformallbyparts}
\end{equs}
It remains to be shown that the desired bound also holds on $B$. To this end, we first note that by Jensen's inequality and by the BDG inequality for $\gamma \in (0,1)$ we have
\begin{equs}
\Big\| \Big(\Theta(t,x) - \sum_{i=1}^K c_i P_t \phi^i(0,\cdot)(x)\Big)^{\gamma}\Big\|_{L_2} & \leq \Big\|\sum_{i=1}^K c_i\Big(\phi^i(t,x) - P_t \phi^i(0,\cdot)(x)\Big)\Big\|_{L_{2}}^{\gamma}\\
&\leq \sum_{i=1}^K \Big\|\int_0^t \int_{\mathbb{T}} p_{t-r}(x,y)\sigma(\phi^i(r,y)) \xi(dy,dr)\Big\|_{L_{2}}^{\gamma}\\
&\lesssim t^{\gamma/4}. \label{ThetaminusICrate}
\end{equs}
We first consider the case that $\beta \in (-1,0)$. By \cref{repeatedMalliavinbyparts}, the fact that $\mathbb{E}G_{n+1} =0$, and (\ref{ThetaminusICrate}) with $\gamma = 1+\beta \in (0,1)$, we get 
\begin{equs}
B   &= |\mathbb{E}(\nabla f(\Theta(t,x))G_{n+1}(t,x))|\\
& = \Big|\mathbb{E}\Big(\Big(\nabla f(\Theta(t,x)) - \nabla f\Big(\sum_{i=1}^K c_i P_t \phi^i(0,\cdot)(x)\Big)\Big) G_{n+1}(t,x))\Big)\Big|\\
& \lesssim \|\nabla f\|_{C^{ 1+\beta}}\Big\| \Big(\Theta(t,x) - \sum_{i=1}^K c_i P_t \phi^i(0,\cdot)(x)\Big)^{1+\beta}\Big\|_{L_2}\|G_{n+1}(t,x)\|_{L_2}\\
&\lesssim \|f\|_{C^{2+\beta}} t^{\frac{1+\beta}{4}}\|G_{n+1}(t,x)\|_{L_2}\\
&\lesssim\|g\|_{C^\beta}t^{\frac{\beta - n}{4}}\|G_0(t,x)\|_{\mathscr{W}_{q}^{n+1}},\label{Bboundformallbyparts}
\end{equs}
where for the last inequality we used Proposition \ref{laplacianlemma} and \cref{iteratedskorokhod}.
We now also deal with the case when $\beta \in (-2,-1)$. Repeating the same steps with one more iteration of Malliavin integration by parts and with $\gamma = 2+\beta \in (0,1)$, we can see that
\begin{equs}
B &= |\mathbb{E}(f(\Theta(t,x))G_{n+2}(t,x))|\\
&\lesssim\| f\|_{C^{ 2+\beta}}\Big\| \Big(\Theta(t,x) - \sum_{i=1}^K c_i P_t \phi^i(0,\cdot)(x)\Big)^{2+\beta}\Big\|_{L_2}\|G_{n+2}(t,x)\|_{L_2}\\
&\lesssim \|g\|_{C^\beta}t^{\frac{\beta - n}{4}}\|G_0(t,x)\|_{\mathscr{W}_{q}^{n+2}}.\qquad\label{Bboundformallbypartsnew}
\end{equs}
By (\ref{Bboundformallbyparts}) and (\ref{Bboundformallbypartsnew}) we can see that for all $\beta \in (-2,-1) \cup (-1,0)$ we have
\begin{equs}B \lesssim \|g\|_{C^\beta}t^{\frac{\beta - n}{4}}\|G_0(t,x)\|_{\mathscr{W}_{q}^{m}}.\label{Bboundsummarised}
\end{equs}
By the bounds (\ref{Aboundformallbyparts}) and (\ref{Bboundsummarised}) on $A$ and $B$ respectively, the proof is finished.
\end{proof}
 Let $s \geq 0$ and suppose that $Z:\Omega\times \mathbb{T} \to \mathbb{R}$ is an $\mathscr{F}_s \otimes \mathscr{B}(\mathbb{T})$-measurable map, such that $Z(x)$ is continuous in $x$ and that $\sup_{x \in \mathbb{T}}\|Z(x)\|_{L_2} <\infty$. Let  $\phi^{Z,s}$ denote the solution of
\begin{equs}
(\partial_t  - \Delta )\phi^{Z,s} = \sigma(\phi^{Z,s})\xi \qquad \text{in $(s, 1) \times \mathbb{T}$},  \qquad \phi^{Z,s}_s = Z. \label{phidef0}
\end{equs}
For $(t,x) \in [s,1]\times \mathbb{T}$, the solution satisfies the integral equation
\begin{equs}\phi^{Z,s}(t,x) = P_{t-s}Z(x) + \int_s^t \int_{\mathbb{T}}p_{t-r}(x,y) \sigma(\phi^{Z,s}(r,y)) \xi(dy,dr). \label{phidef1}
\end{equs}
We will moreover use the shorthand
\begin{equs}\phi^Z(t,x) := \phi^{Z,0}(t,x). \label{phidef2}\end{equs}

\begin{lemma}\label{malliavinbypartsnew}
Let $K \in \mathbb{N}$, and for $z_1,\dots, z_K \in C(\mathbb{T})$, $(s,t)\in [0,1]_{\leq}^2$, $x \in \mathbb{T}$, consider the convex combination 
\begin{equs}
\Phi^{z_1,\dots,z_K,s}(t,x)&:= \sum_{i=1}^K c_i \phi^{z_i,s}(t,x).
\end{equs} 
 Let $h \in C(\mathbb{R}^K)$, and for $q \in [1,\infty)$ and $i \in \mathbb{Z}_{\geq 0}$ define $H_{q,i}: (C(\mathbb{T}))^K \to \mathbb{R}$ by
$$H_{q,i}(z_1,\dots, z_K) := \| h\big(\phi^{z_1}(t-s,x),\dots, \phi^{z_K}(t-s,x)\big)\|_{\bdot{\mathscr{W}}_q^i}.$$
Let $Z_1,\dots,Z_K$ be $\mathscr{F}_s$-measurable $C(\mathbb{T})$-valued random variables and $g \in C^\infty(\mathbb{R})$. For all $n \in \mathbb{Z}_{\geq 0}$, $\beta \in (-2,-1)\cup(-1,0)$ there exists a constant $N = N(n, \beta, \|\sigma\|_{C^{n+1}},\mu)$ such that with $q := 2^{n+2}\mathbf{1}_{(-1,0)}(\beta) + 2^{n+3}\mathbf{1}_{(-2,-1)}(\beta) $ and $m:= (n+1)\mathbf{1}_{(-1,0)}(\beta) + (n+2)\mathbf{1}_{(-2,-1)}(\beta)$ we have
\begin{align*}
&\Big|\mathbb{E}^s\Big( \nabla^n g\Big(\Phi^{Z_1,\dots, Z_K,s}(t,x)\Big)  h\big(\phi^{Z_1,s}(t,x),\dots, \phi^{Z_K,s}(t,x)\big) \Big)\Big| \\&
\qquad\qquad\qquad\qquad\qquad\qquad\leq N\|g\|_{C^\beta} (t-s)^{\frac{\beta -n}{4}}\sum_{i=0}^{m}H_{q,i}\big(Z_1,\dots, Z_K\big).
\end{align*}
\end{lemma}

\begin{proof}
By \cref{markovproperty} we have
\begin{equs}
\Big|\mathbb{E}^s\Big( \nabla^n g\big(\Phi^{Z_1,\dots, Z_K,s}(t,x)\big)  h\big(\phi^{Z_1,s}(t,x),\dots, \phi^{Z_K,s}(t,x)\big) \Big)\Big| = G(Z_1,\dots,Z_K), \label{GZ1dotsZK}
\end{equs}
where for $z_1, \dots, z_K \in C(\mathbb{T})$ we define
$$G(z_1,\dots,z_K) := \Big|\mathbb{E}\Big(\nabla^n g\Big(\sum_{i=1}^K c_i \phi^{z_i}(t-s,x)\Big) h(\phi^{z_i}(t-s,x),\dots, \phi^{z_K}(t-s,x))\Big)\Big|.$$
By \cref{mainmalliavintool} we have
\begin{equs}G(z_1,\dots, z_K) &\lesssim \|g\|_{C^\beta}(t-s)^{\frac{\beta -n}{4}}\|h(\phi^{z_1}(t-s,x),\dots,\phi^{z_M}(t-s,x))\|_{\mathscr{W}_q^{m}}\\
&\lesssim \|g\|_{C^\beta}(t-s)^{\frac{\beta-n}{4}}\sum_{i=0}^{m}\|h(\phi^{z_1}(t-s,x),\dots, \phi^{z_M}(t-s,x))\|_{\bdot{\mathscr{W}}_q^{i}}\\
& = \|g\|_{C^\beta}(t-s)^{\frac{\beta -n}{4}}\sum_{i=0}^{m}H_{q,i}(z_1,\dots,z_K). \label{Hqnplus1}
\end{equs}
Now putting (\ref{Hqnplus1}) into (\ref{GZ1dotsZK}), the desired result follows.
\end{proof}

\section{Driftless approximation}\label{driftilessapproxsection}
In this section we deal with the approximation of the solution $u(t,x)$ by $\phi^{u(s, \cdot),s}(t,x)$. The main results of this section are  \cref{driftlessapprox} and  \cref{4pointestimate}. 

\begin{lemma}[Boundedness of regularised solutions
]\label{boundednessofsolutions}
Let $u$ be a regularised solution of (\ref{SHE}) with initial condition $u(0,\cdot) = u_0 \in C(\mathbb{T})$ and let $p \in [2,\infty)$. There exists a constant $N = N(p,\|\sigma\|_{\mathbb{B}})$ such that for all $(t,x) \in [0,1]\times \mathbb{T}$ we have
\begin{equs}\label{boundednessofu}\|u(t,x)\|_{L_p} \leq N\Big(\|u_0\|_{\mathbb{B}(\mathbb{T})}+ \|D^u_t(x)\|_{L_p} + t^{1/4}\Big).
\end{equs}
Consequently, if $u \in \mathscr{U}_p^0$, then 
\begin{equs}\sup_{(t,x) \in [0,1]\times \mathbb{T}}\|u(t,x)\|_{L_p} <\infty. \label{uisfinite}
\end{equs}
\end{lemma}
\begin{proof}
From (\ref{integraleqn}) and the BDG inequality we can see that
\begin{align*}\|u(t,x)\|_{L_p}^2 &\lesssim \|u_0\|_{\mathbb{B}(\mathbb{T})}^2 + \|D_t^u(x)\|_{L_p}^2 + \int_0^t \int_{\mathbb{T}}p_{t-r}^2(x,y)\|\sigma\|_{\mathbb{B}}^2dydr,
\end{align*}
and thus the inequality (\ref{boundednessofu}) follows.
Now suppose that $u \in \mathscr{U}_p^0$. Then noting that $D_t^u = D_t^u - P_{t-0}D_0^u$, we have
$$\sup_{(t,x) \in [0,1]\times \mathbb{T}}\|D^u_t(x)\|_{L_p}    \leq [D^u]_{\mathscr{V}_p^0} <\infty.$$
Also, note that by Assumption \ref{assumptions} the initial condition $u(0,\cdot)$ is bounded. Therefore all terms on the right hand side of (\ref{boundednessofu}) are bounded in $(t,x)$, and thus (\ref{uisfinite}) follows.
\end{proof}

 Recall that the random fields $\phi^{z,s}(t,x)$ and $\phi^z(t,x)$  are defined by (\ref{phidef1}) and (\ref{phidef2}) respectively.
Let $\sigma \in C^1$, $\alpha\in(-1,0)$, $\beta \in [0,1]$, $p \in [1,\infty)$, and for $i=1,2$ let $b^i \in C^{\alpha}$, and suppose that  $u^i \in \mathscr{U}_p^\beta$ are regularised solutions of the stochastic reaction--diffusion equations \begin{equs}(\partial_t - \Delta)u^i = b^i(u^i) + \sigma(u^i)\xi \label{u1u2equations}.
\end{equs}
 For $(S,T) \in [0,1]_{\leq}^2$ we define the \emph{$\mathscr{S}_p^{\beta}[S,T]$-bracket} of $u^1$ and $u^2$ by
\begin{equs}\label{Spbracket} \mkern-35mu
[ u^1,u^2]_{\mathscr{S}^{\beta}_p[S,T]} := \sup_{x \in \mathbb{T}}\sup_{(s,t)\in [S,T]_{<}^2}\frac{ \| u^1(t,x) - \phi^{u^1(s,\cdot),s}(t,x) - u^2(t,x) + \phi^{u^2(s,\cdot),s}(t,x) \|_{L_p}}{|t-s|^{\beta}}.
\end{equs}
\begin{remark}
Note that for all $s \in [0,1]$, by definition, the random field $u^i(s,x)$ is continuous in $x$, and by \cref{boundednessofsolutions} we have $\sup_{x \in \mathbb{T}}\|u^i(s,x)\|_{L_p} <\infty$. Therefore the equation (\ref{phidef0}) starting from $u^i(s,\cdot)$ does indeed have a unique solution, which is denoted by $\phi^{u^i(s,\cdot),s}(t,x)$.
Consequentially, the expression (\ref{Spbracket}) is well-defined.
\end{remark}
 For brevity, we will use the convention 
$$[u^1,u^2]_{\mathscr{S}_p^{\beta}} := [u^1,u^2]_{\mathscr{S}_p^{\beta}[0,1]}.$$
Moreover recalling the definition of  the $\mathscr{V}_p^\beta[S,T]$-bracket from Definition \ref{defofclassV}, we set \begin{equs}\,
[f]_{\mathscr{V}_p^\beta} :=[f]_{\mathscr{V}_p^\beta[0,1]}.
\end{equs}
By the triangle inequality and by \cref{lipschitzflowgeneral} the following result holds:
\begin{lemma}\label{bracketlemma} 
Let $\sigma \in C^1$, $b^1,b^2 \in C^\alpha$, $\beta \in (0,1]$, $p \in [2,\infty)$ and let $u^1,u^2$ be regularised solutions of (\ref{u1u2equations}) in the class $\mathscr{U}_p^\beta$. There exists some constant $N =N(p, \|\sigma\|_{C^1}, \alpha, \beta)$ such that  for all $(s,t) \in [0,1]_{\leq}^2$ we have
$$\|u^1(t,\cdot) - u^2(t,\cdot)\|_{\mathbb{B}(\mathbb{T}, L_p)} \leq  [u^1,u^2]_{\mathscr{S}_p^{\beta}[s,t]}(t-s)^\beta + N\|u^1(s,\cdot)  - u^2(s,\cdot)\|_{\mathbb{B}(\mathbb{T}, L_p)}.$$
\end{lemma}

\begin{lemma}[Driftless approximation]\label{driftlessapprox} 
Let $\sigma \in C^1$, $\alpha \in (-1,0)$, $b \in C^{\alpha}$. Let $p \in [2,\infty)$, $\beta \in [0, 1+ \frac{\alpha}{4}]$ and let $u$ be a regularised solution of (\ref{SHE}) in the class $\mathscr{U}_p^\beta$. Then there exists a constant $N = N(p, \|\sigma\|_{C^1}, \alpha, \beta)$ such that for all $0\leq s \leq t \leq 1$, we have 
$$\sup_{x \in \mathbb{T}}\|u(t,x) - \phi^{u(s,\cdot),s}(t,x)\|_{L_{p,\infty}^{\mathscr{F}_s}} \leq N [D^u]_{\mathscr{V}_p^\beta[s,t]}(t-s)^{\beta}.$$
\end{lemma}
\begin{proof}
By splitting the stochastic integral in (\ref{integralSHE}) at time $s$, we have
\begin{equs}u(t,x) &=P_tu(0,\cdot)(x) + D_t^u(x) + \int_0^s \int_{\mathbb{T}} p_{t-r}(x,y)\sigma(u(r,y))\xi(dy,dr)\\
&\qquad\qquad\qquad\qquad\qquad+ \int_s^t \int_{\mathbb{T}} p_{t-r}(x,y)\sigma(u(r,y))\xi(dy,dr). \label{utxtocompare}
\end{equs}
Moreover using (\ref{phidef1}) to compute $\phi^{u(s,\cdot),s}(t,x)$ and then (\ref{integraleqn}) to express $u(s,\cdot)$, we get 
\begin{equs}
\phi^{u(s,\cdot),s}(t,x)&= P_{t-s}\Big(P_s u(0,\cdot)+ D_s^u(\cdot) + \int_0^s \int_{\mathbb{T}}p_{s-r}(\cdot,y)\sigma(u(r,y))\xi(dy,dr)\Big)(x)\\
&\qquad\qquad\qquad\qquad\qquad+ \int_s^t \int_{\mathbb{T}}p_{t-r}(x,y)\sigma(\phi^{u(s,\cdot),s}(r,y))\xi(dy,dr)\\
&=P_{t}u(0,\cdot)(x) + P_{t-s}D_s^u(x) + \int_0^s\int_{\mathbb{T}} p_{t-r}(x,y)\sigma(u(r,y))\xi(dy,dr) \\
&\qquad\qquad\qquad\qquad\qquad+ \int_s^t \int_{\mathbb{T}}p_{t-r}(x,y)\sigma(\phi^{u(s,\cdot),s}(r,y))\xi(dy,dr) \label{phiustxtocompare}
\end{equs}
where the last equality follows from the semigroup property that $P_{t-s}P_s = P_t$. Comparing (\ref{utxtocompare}) and (\ref{phiustxtocompare}), we can see that the error of the driftless approximation is given by
\begin{equs}
u(t,x) - \phi^{u(s,\cdot),s}(t,x) &= D_t^u(x) - P_{t-s}D_s^u(x)  \\
&\quad+ 
\int_s^t \int_{\mathbb{T}}p_{t-r}(x,y)\big( \sigma(u(r,y)) -\sigma(\phi^{u(s,\cdot),s}(r,y))\big)\xi(dy,dr).
\end{equs}
Hence by the conditional BDG inequality (see \cref{conditionalBDGforstochintegral}), we get
\begin{equs}
&\|u(t,x) - \phi^{u(s,\cdot),s}(t,x)\|_{L_p | \mathscr{F}_s}^2 \lesssim \|D_t^u(x) - P_{t-s}D_s^u(x)\|_{L_p|\mathscr{F}_s}^2\\
&\qquad\qquad\qquad\qquad+ \int_s^t \int_{\mathbb{T}}p^2_{t-r}(x,y)\big\| \sigma(u(r,y)) -\sigma(\phi^{u(s,\cdot),s}(r,y))\big\|_{L_p | \mathscr{F}_s}^2dydr. \label{uminusphiafterBDG}
\end{equs}
Therefore
\begin{equs}
\|u(t,x) - \phi^{u(s,\cdot),s}(t,x)\|_{L_{p,\infty}^{\mathscr{F}_s}}^2 &\lesssim [D^u]_{\mathscr{V}_p^\beta[s,t]}^2(t-s)^{2\beta}  \\
&+\int_s^t \int_{\mathbb{T}}p^2_{t-r}(x,y)\|u(r,y) - \phi^{u(s,\cdot),s}(r,y)\|_{L_{p,\infty}^{\mathscr{F}_s}}^2 dydr. \label{uminusphitobuckle}
\end{equs}
From (\ref{uminusphiafterBDG}) and the fact that $u \in \mathscr{U}_p^\beta$, we see that 
\begin{equs}
&\sup_{(t,x)\in [0,1]\times \mathbb{T}}\|u(t,x) - \phi^{u(s,\cdot),s}(t,x)\|_{L_{p,\infty}^{\mathscr{F}_s}} \lesssim [D^u]_{\mathscr{V}_p^0} + \|\sigma\|_{\mathbb{B}} <\infty.
\end{equs}
Hence by (\ref{uminusphitobuckle}) and by \cref{Gronwalltypelemma}, we obtain that
$$\sup_{x \in \mathbb{T}}\|u(t,x) - \phi^{u(s,\cdot),s}(t,x)\|_{L_{p,\infty}^{\mathscr{F}_s}}^2 \lesssim [D^u]^2_{\mathscr{V}_p^\beta[s,t]}(t-s)^{2\beta},$$
which implies the desired result.
\end{proof}

\begin{assumption}\label{4pointassumption}
 Let $\alpha \in (-1,0)$, $b \in C^\alpha$, $n \in \mathbb{Z}_{\geq 0}$, and let $\sigma \in C^{n+2}$ such that there exists a constant $\mu > 0$ such that for all $x \in \mathbb{R}$, $\sigma^2(x) > \mu^2$.
 Let $\beta \in [\frac{1}{2} ,1+\frac{\alpha}{4}]$, and suppose that $u^1,u^2$ are regularised solutions of (\ref{SHE}) in the class $\mathscr{U}^\beta$,
 \end{assumption}
 
  For $(s,a) \in [0,1]_{\leq}^2$ consider the $(C(\mathbb{T}))^4$-valued $\mathscr{F}_a$-measurable random variable
\begin{equs}Z :=\Big(\phi^{u^1(s,\cdot),s}(a,\cdot), \phi^{u^2(s,\cdot),s}(a,\cdot), u^1(a,\cdot), u^2(a,\cdot)\Big). \label{defofZ}
\end{equs}
Recall the definitions of $F^{(2)}$ and $\Sigma^{(2)}$ from (\ref{F2def}) and (\ref{Sigma2def}). Moreover for $(t,x) \in [0,1]\times \mathbb{T}$ and $z  = (z_1,\dots, z_4) \in (C(\mathbb{T}))^4$, define
\begin{equs}
F^{(4)}_{q,n}(t,x,{z}) &:= \|\phi^{{z}_1}(t,x) - \phi^{{z}_2}(t,x) - \phi^{{z}_3}(t,x) + \phi^{{z}_4}(t,x)\|_{\bdot{\mathscr{W}}_q^n}, \label{F4def}\\
\Sigma^{(4)}_{q,n}(t,x,{z})&:= \|\sigma(\phi^{{z}_1}(t,x)) - \sigma(\phi^{{z}_2}(t,x)) - \sigma(\phi^{{z}_3}(t,x)) + \sigma(\phi^{{z}_4}(t,x))\|_{\bdot{\mathscr{W}}_q^n}. \label{Sigma4def}
\end{equs}
By \cref{lipschitzflowgeneral}, it follows that the expression in \eqref{F4def} is continuous in $z$. Similarly, by \cref{lipschitzflowgeneral}, the product and chain rule formulas for Malliavin derivatives, it is easy to see that the same holds for the expression in \eqref{Sigma4def}. 
Our next task is to obtain an estimate on $F^{(4)}$ evaluated at $Z$. This estimate is given in the next lemma.

\begin{lemma}[Four point estimate]\label{4pointestimate}
Let Assumption \ref{4pointassumption} hold.
Then for all $p \in [2,\infty)$, there exists a constant $N = N(n,p,\|\sigma\|_{C^{n+2}}, \alpha, \beta)$ such that for all $(S,T) \in [0,1]_{\leq}^2$, $(s,a) \in [S,T]_{\leq}^2$, $ t \in [0, 1 - a+s]$ and $x \in \mathbb{T}$, we have that                
\begin{align*}&\sup_{x \in \mathbb{T}}\|F^{(4)}_{p,n}(t,\,x,Z)\|_{L_{p}}\\
&\leq N (1+\max_{i \in \{1,2\}}[D^{u^i}]_{\mathscr{V}_{2p}^\beta})\Big([u^1,u^2]_{\mathscr{S}_p^{1/2}[S,a]} + \|u^1(S,\cdot) - u^2(S,\cdot)\|_{\mathbb{B}(\mathbb{T}, L_p)}\Big)|a-s|^{\frac{1}{2}}
\end{align*}
where $Z$ is defined by (\ref{defofZ}).
\end{lemma}

To prove the above estimate, we will need the following auxiliary lemma.

\begin{lemma}\label{pre4pointlemma}
Let Assumption \ref{4pointassumption} hold.
Then for all $p \in [2,\infty)$ there exists a constant $N = N(n,p, \|\sigma\|_{C^{n+2}}, \alpha, \beta)$ such that for all  $(s,a) \in [0,1]_{\leq}^2$, $t \in [0,1-a+s]$ and $x \in \mathbb{T}$ we have
\begin{equs}\|\Sigma^{(4)}_{p,n}(t,x,Z)\|_{L_{p}} &\leq N \max_{i \in \{1,2\}}[D^{u^i}]_{\mathscr{V}_{2p}^\beta}\sup_{x \in \mathbb{T}}\|u^1(s,\cdot) - u^2(s,\cdot)\|_{\mathbb{B}(\mathbb{T}, L_p)} (a-s)^{\beta}\\
&\qquad + N\mathbf{1}_{n \geq 1}\sum_{i=0}^{n-1} \|F^{(4)}_{2p,i}(t,x,Z)\|_{L_{p}}  + N\|F_{p,n}^{(4)}(t,x, Z)\|_{L_{p}}.
\end{equs}
\end{lemma}
\begin{proof}
We begin by proving the result for $n \geq 1$. By point (\ref{generalities4point}) in \cref{Dsigmaubuckling} (which we can apply with $\varepsilon = t^{1/4} \in [0,1]$ by \cref{posimomentbound}) and by the triangle inequality and H\"older's inequality we have that
 \begin{equs}&\|\Sigma_{p,n}^{(4)}(t,x,Z)\|_{L_p|\mathscr{F}_s}\\
 &\lesssim \sum_{i+j \leq n}  \big\|F_{4p,i}^{(2)}(t,x,Z_1,Z_2)\big\|_{L_{2p}|\mathscr{F}_s} \Big( \sum_{l \in \{1,2\}}\big\|F_{4p,j}^{(2)}(t,x, Z_l, Z_{l+2}) \big\|_{L_{2p}|\mathscr{F}_s} \Big)\\
 &\qquad + \sum_{i=0}^{n-1} \big\| F^{(4)}_{2p,i}(t,x,Z)\big\|_{L_p|\mathscr{F}_s} + \|F_{p,n}^{(4)}(t,x,Z)\|_{L_p|\mathscr{F}_s}\\
 & =: A+B+C.
 \end{equs}
 We can immediately see that $\|B\|_{L_p}, \|C\|_{L_p}$ can be estimated by the second and third terms of the desired bound. We proceed with showing that $\|A\|_{L_p}$ can be estimated by the 
 first term of the desired bound. To this end, note that by \cref{lipschitzflowgeneral} and by \cref{driftlessapprox} we have for $l=1,2$, uniformly in $j \in \{0,\dots, n\}$ that
 \begin{equs}
 \big\| F^{(2)}_{4p,j}(t,x,Z_l,Z_{l+2})\big\|_{L_{2p,\infty}^{\mathscr{F}_s}} &\lesssim \sup_{x \in \mathbb{T}}\|Z_l(x) - Z_{l+2}(x)\|_{L_{2p,\infty}^{\mathscr{F}_s}}\\
 & = \sup_{x \in \mathbb{T}}\|\phi^{u^l(s,\cdot),s}(a,x) - u^l(a,x)\|_{L_{2p,\infty}^{\mathscr{F}_s}}\\
   &\lesssim (a-s)^{\beta}\max_{l \in \{1,2\}}[D^{u^l}]_{\mathscr{V}_{2p}^\beta[s,a]}.
 \end{equs}
Therefore
\begin{equs}
A \lesssim (a-s)^{\beta}\max_{l \in \{1,2\}}[D^{u^l}]_{\mathscr{V}_{2p}^\beta}\sum_{i+j \leq n} \| F^{(2)}_{4p,i}(t,x,Z_1,Z_2) \|_{L_{2p}|\mathscr{F}_s} . \label{AboundedbyFtimestpowerD}
\end{equs}
Applying \cref{lipschitzflowgeneral} with $q = 4p$, $(p_1,p_2) = (2p,p)$, $\mathscr{G} = \mathscr{F}_s$, we get
\begin{equs}
\|F_{4p,i}^{(2)}(t,x,Z_1,Z_2)\|_{L_{2p, p}^{\mathscr{F}_s}} 
&\lesssim t^{i/4}\sup_{x \in \mathbb{T}}\|Z_1(x) - {Z}_2(x)\|_{L_{2p,p}^{\mathscr{G}}}\\
&= t^{i/4}\sup_{x \in \mathbb{T}}\|\phi^{u^1(s,\cdot), s}(a,x) - \phi^{u^2(s,\cdot), s}(a,x)\|_{L_{2p,p}^{\mathscr{G}}}\\
& = t^{i/4}\sup_{x \in \mathbb{T}}\| F_{2p, 0}(a-s, x, u^1(s,\cdot), u^2(s,\cdot))\|_{L_p},
\end{equs}
for all $i \in \{0,\dots,n\}$, where the last equality holds by \cref{markovproperty}. So applying \cref{lipschitzflowgeneral} with $q = 2p$, $p_1 = p_2 = p$, and with an arbitrary sub-$\sigma$-algebra $\mathscr{G}\subset \mathscr{F}$, we get from the above inequality that
\begin{equs}
\|F_{4p,i}^{(2)}(t,x,Z_1,Z_2)\|_{L_{2p,p}^{\mathscr{F}_s}} \lesssim \sup_{x \in \mathbb{T}}\|u^1(s,x) - u^2(s,x)\|_{L_p},
\end{equs}
and the bound is uniform in $i \in \{0,\dots, n\}$.
 Now taking the $L_p$-norm on (\ref{AboundedbyFtimestpowerD}),and applying the above inequality, we get that
\begin{equs}
\|A\|_{L_p}& \lesssim
(a-s)^{\beta}\max_{l \in \{1,2\}}[D^{u^l}]_{\mathscr{V}_{2p}^\beta}\sum_{i+j+k} \|F_{4p,i}^{(2)}(t,x,Z_1,Z_2)\|_{L_{2p,p}^{\mathscr{F}_s}}
\\&\lesssim (a-s)^{\beta}\max_{l \in \{1,2\}}[D^{u^l}]_{\mathscr{V}_{2p}^\beta}\sup_{x \in \mathbb{T}}\|u^1(s,x) - u^2(s,x)\|_{L_p}.
\end{equs}
which finishes the proof for the $n \geq 1$ case. Finally, for the $n=0$ case, note that using (\ref{4pointidentity}), we have
\begin{equs}
\|\Sigma_{p,n}^{(4)}(t,x,Z)\|_{L_p|\mathscr{F}_s} &\leq \|F_{2p,n}^{(4)}(t,x,Z_1,Z_2)\|_{L_{2p}|\mathscr{F}_s}\sum_{l \in \{1,2\}}\| F_{2p,j}^{(2)}(t,x,Z_{l}, Z_{l+2})\|_{L_{2p}|\mathscr{F}_s}\\
&\qquad\qquad\qquad\qquad\qquad\qquad\qquad+ \|F_{p,0}^{(4)}(t,x,Z)\|_{L_p|\mathscr{F}_s}\\
& = : \tilde{A} + \tilde{C}.
\end{equs}
By estimating $\|\tilde{A}\|_{L_p}$ and $\|\tilde{C}\|_{L_p}$ the same way as we did for $\|A\|_{L_p}$ and $\|C\|_{L_p}$ respectively, one can show that the desired result also holds for $n=0$, which finishes the proof.
\end{proof}
We are now in position to prove \cref{4pointestimate}.
\begin{proof}[Proof of \cref{4pointestimate}]
We begin by confirming that $\sup_{(t,x) \in [0,1]\times\mathbb{T}}\|F^{(4)}_{p,n}(t,x,Z)\|_{L_{p}} < \infty$.
This is indeed true, since by the triangle inequality, \cref{lipschitzflowgeneral} and \cref{driftlessapprox} we have
\begin{equs}
\sup_{x \in \mathbb{T}}\|F^{(4)}_{p,n}(t,x,Z)\|_{L_{p}} &\leq \sum_{i \in \{1,2\}}\| F_{p,n}^{(2)}\big(t,x, \phi^{u^i(s,\cdot),s}(a,\cdot), u^i(a,\cdot)\big)\|_{L_{p}}\\
&\lesssim \sum_{i \in \{1,2\}}t^{\frac{n}{4}}\sup_{x \in \mathbb{T}}\| \phi^{u^i(s,\cdot),s}(a,x)- u^i(a,x)\|_{L_{p}}
\\ & \lesssim t^{\frac{n}{4}}(a-s)^{\beta}\max_{i \in \{1,2\}}[D^{u^i}]_{\mathscr{V}_p^\beta}. \label{gnisfinite}
\end{equs}
The rest of the proof will be done by induction.

\underline{Step 1:} We prove that the statement holds for $n=0$. By
(\ref{phidef1}) we have for $z \in (C(\mathbb{T}))^4$ that
\begin{equs}
&(\phi^{z_1} - \phi^{z_2} - \phi^{z_3} + \phi^{z_4})(t,x)= P_{t}(z_1 -z_2 -z_3 + z_4)(x)\\
&+ \int_0^{t}\int_{\mathbb{T}} p_{t-r}(x,y)\big(\sigma(\phi^{z_1}(r,y)) -\sigma(\phi^{z_2}(r,y)) - \sigma(\phi^{z_3}(r,y))+\sigma(\phi^{z_4}(r,y))\big)\xi(dy,dr).
\end{equs}
Therefore by the BDG inequality
\begin{equs}&F_{p,0}^{(4)}(t,x,z)\\
&\lesssim P_t(z_1 - z_2 - z_3 + z_4)(x)+ \Big(\int_0^{t}\int_{\mathbb{T}}p_{t-r}^2(x,y)\big|\Sigma_{p,0}^{(4)}(r,y,z)\big|^2 dydr\Big)^{1/2}\\
&=: A(z) + B(z). \label{F4isAplusB}
\end{equs}
By the definition of the $\mathscr{S}_p^{1/2}$-bracket, we have
\begin{equs}
\|A(Z)\|_{L_{p}}&\leq \sup_{x \in \mathbb{T}}\| \phi^{u^1(s,\cdot),s}(a,x) - \phi^{u^2(s,\cdot),s}(a,x) - u^1(a,x) + u^2(a,x)\|_{L_p}\\
&\lesssim [u^1,u^2]_{\mathscr{S}_p^{1/2}[s,a]}(a-s)^{1/2}.
\end{equs}
Moreover using \cref{pre4pointlemma}, one can show that
\begin{align*}
\|B(Z)\|_{L_{p}}
&\lesssim \max_{i \in \{1,2,\}}[D^{u^i}]_{\mathscr{V}_{2p}^\beta}\sup_{x \in \mathbb{T}}\|u^1(s,x) - u^2(s,x)\|_{L_p}(a-s)^{\beta} \\
&\qquad + \Big(\int_0^{t}\int_{\mathbb{T}}p_{t-r}^2(x,y)\|F_{p,0}^{(4)}(r,y,Z)\|_{L_{p}}^2dydr\Big)^{1/2}.
\end{align*}
Using the above bounds on $A,B$, the decomposition (\ref{F4isAplusB}), and the assumption that we have $\beta \geq 1/2$, we can see that
\begin{align*}
&\|F^{(4)}_{p,0}\big(t,x,Z\big)\|_{L_{p}}^2 \\
&\lesssim \Big|(1+\max_{i \in \{1,2,\}}[D^{u^1}]_{\mathscr{V}_{2p}^\beta})\big([u^1,u^2]_{\mathscr{S}_p^{1/2}[S,a]} + \|u^1(S,\cdot) - u^2(S,\cdot)\|_{\mathbb{B}(\mathbb{T}, L_p)}\big)(a-s)^{1/2}\Big|^2\\
&\qquad + \int_0^{t}\int_{\mathbb{T}}p_{t-r}^2(x,y)\|F^{(4)}_{p,0}(r,y,Z)\|_{L_{p}}^2dydr.
\end{align*}
Since by assumption $u^1,u^2 \in \mathscr{U}_{2p}^\beta$, we have $\max_{i \in \{1,2\}}[D^{u^i}]_{\mathscr{V}_{2p}^\beta} <\infty$.
Therefore by (\ref{gnisfinite}) we have that the norm in the integrand is bounded in $(r,y)$.
Hence using \cref{Gronwalltypelemma} finishes the proof for the $n=0$ case.

\underline{Step 2:} Let $n \in \mathbb{N}$ and suppose that the statement holds for $F_{p,0}^{(4)},\dots, F^{(4)}_{p,n-1}$ for all $p \geq 2$. We aim to show that then the result also holds for $F^{(4)}_{p,n}$ for all $p \geq 2$.
Let $\gamma  = (\theta_i,\zeta_i)_{i=1}^n\in ([0,1]\times \mathbb{T})^n$ and $z \in (C(\mathbb{T}))^4$. Then by Proposition \ref{CHN3} we have that
\begin{equs}
&\mathscr{D}_\gamma^n(\phi^{z_1} - \phi^{z_2} - \phi^{z_3} + \phi^{z_4})(t,x)\\
&= \mathbf{1}_{[0,t]}(\theta^*)\sum_{k=1}^n p_{t-\theta_k}(x,\zeta_k)\\
&\qquad\qquad \times\mathscr{D}_{\hat{\gamma}_k}^{n-1}\big[
\sigma(\phi^{z_1}(\theta_k,\zeta_k)) -\sigma(\phi^{z_2}(\theta_k,\zeta_k))-\sigma(\phi^{z_3}(\theta_k,\zeta_k)) +\sigma(\phi^{z_4}(\theta_k,\zeta_k))\big]\\
& + \mathbf{1}_{[0,t]}(\theta^*)\int_{0}^{t} p_{t-r}(x,y)\\
&\qquad\qquad \times\mathscr{D}_\gamma^n\big[\sigma(\phi^{z_1}(r,y)) -\sigma(\phi^{z_2}(r,y)) - \sigma(\phi^{z_3}(r,y)) + \sigma(\phi^{z_4}(r,y))
\big]\xi(dy,dr).
\end{equs}
 Taking the $\|\cdot\|_{L_p(\Omega;H^{\otimes n})}$-norm of both sides and using the BDG inequality gives
\begin{equs}
&F^{(4)}_{p,n}(t,x,z_1,z_2,z_3,z_4)\\
&\lesssim \|\mathbf{1}_{[0,t]}(\cdot)p_{t-\cdot}(x,\cdot) \mathscr{D}^{n-1}[\sigma(\phi^{z_1}(\cdot,\cdot)) - \sigma(\phi^{z_2}(\cdot,\cdot)) - \sigma(\phi^{z_3}(\cdot,\cdot))+ \sigma(\phi^{z_4}(\cdot,\cdot))]\|_{L_p(\Omega;H^{\otimes n})}\\
&\qquad\qquad+ \Big(
\int_0^t\int_{\mathbb{T}} p_{t-r}^2(x,y)|\Sigma_{p,n}^{(4)}(r,y,z)|^2dydr
\Big)^{1/2}\\
& =: A(z_1,z_2,z_3,z_4)+B(z_1,z_2,z_3,z_4). \label{decompositionof4point}
\end{equs}
  We begin by bounding $A$.
Note that
\begin{equs}\qquad& A(z)
\lesssim \\
&\|\mathbf{1}_{[0,t]}(\cdot)p_{t-\cdot}(x,\cdot)\|\|\mathscr{D}^{n-1}[\sigma(\phi^{z_1}(\cdot,\cdot)) - \sigma(\phi^{z_2}(\cdot,\cdot)) - \sigma(\phi^{z_3}(\cdot,\cdot))+ \sigma(\phi^{z_4}(\cdot,\cdot))]\|_{H^{\otimes (n-1)}}\|_{L_p}\|_{H}\\
& = \|\mathbf{1}_{[0,t]}(\cdot)p_{t-\cdot}(x,\cdot)\Sigma^{(4)}_{p,n-1}(\cdot,\cdot,z)\|_H,
\end{equs}
and thus (recalling the definition of  $Z$ from (\ref{defofZ})) we have $A(Z) \lesssim \|\mathbf{1}_{[0,t]}(\cdot)p_{t-\cdot}(x,\cdot)\Sigma^{(4)}_{p,n-1}(\cdot,\cdot,Z)\|_H$. Hence using the Minkowski inequality we obtain that
\begin{equs}
\|A(Z)\|_{L_{p}} 
&\lesssim  \|\mathbf{1}_{[0,t](\cdot)} p_{t-\cdot}(x,\cdot) \|\Sigma^{(4)}_{p,n-1}(\cdot,\cdot,Z)\|_{L_{p}}\|_H\\
&\leq \|\mathbf{1}_{[0,t]}(\cdot)p_{t-\cdot}(x,\cdot)\|_{H}\sup_{(\theta,\zeta) \in [0,t]\times \mathbb{T}}\|\Sigma^{(4)}_{p,n-1}(\theta,\zeta,Z)\|_{L_{p}}.
\end{equs}
To bound the first factor, we note that $\|\mathbf{1}_{[0,t]}(\cdot)p_{t-\cdot}(x,\cdot)\|_{H} \lesssim t^{1/4} \leq 1$,
and to bound the second factor we use that by \cref{pre4pointlemma} and by the induction  hypothesis, for $(\theta,\zeta) \in [0,t]\times \mathbb{T}$ we have that
\begin{equs}&\|\Sigma^{(4)}_{p,n-1}(\theta,\zeta,Z)\|_{L_{p}}\\
&\lesssim \max_{i \in \{1,2\}}[D^{u^i}]_{\mathscr{V}_p^\beta}\|u^1(s,\cdot) - u^2(s,\cdot)\|_{\mathbb{B}(\mathbb{T}, L_p)} (a-s)^{\beta} + \sum_{i=0}^{n-1} \|F^{(4)}_{2p,i}(\theta,\zeta,Z)\|_{L_{p}} \\
&\lesssim (1+ \max_{i \in \{1,2\}}[D^{u^i}]_{\mathscr{V}_{2p}^\beta})\Big([u^1,u^2]_{\mathscr{S}_p^{1/2}[S,a]} + \|u^1(S,\cdot) - u^2(S,\cdot)\|_{\mathbb{B}(\mathbb{T}, L_p)}\Big) (a-s)^{\frac{1}{2}},
\end{equs}
where we used Lemma \ref{bracketlemma} and that by assumption we have $\beta \geq 1/2$.
Hence we can see that
$$\|A(Z)\|_{L_{p}} \lesssim (1+ \max_{i \in \{1,2\}}[D^{u^i}]_{\mathscr{V}_{2p}^\beta})\Big([u^1,u^2]_{\mathscr{S}_p^{1/2}[S,a]} + \|u^1(S,\cdot) - u^2(S,\cdot)\|_{\mathbb{B}(\mathbb{T}, L_p)}\Big)(a-s)^{\frac{1}{2}}.$$
 We proceed with bounding $B$. Note that
\begin{equs}
\|B(Z)\|_{L_{2,p}^{\mathscr{F}_s}} \lesssim \Big(\int_{0}^{t} \int_{\mathbb{T}} p_{t-r}^2(x,y) \| \Sigma_{q,n}^{(4)}(r,y,Z)\|_{L_{2,p}^{\mathscr{F}_s}}^2 dydr \Big)^{1/2}. \label{boundonBZ}
\end{equs}
By \cref{pre4pointlemma} and the induction  hypothesis we have for all $(r,y) \in [0,t]\times \mathbb{T}$ that 
\begin{equs}
&\|\Sigma^{(4)}_{p,n}(r,y,Z)\|_{L_{p}}\\
&\lesssim \max_{i \in \{1,2\}}[D^{u^i}]_{\mathscr{V}_{2p}^\beta}\|u^1(s,\cdot) - u^2(s,\cdot)\|_{\mathbb{B}(\mathbb{T}, L_p)} (a-s)^{\beta}\\
&\qquad + \sum_{i=0}^{n-1} \|F^{(4)}_{2p,i}(r,y,Z)\|_{L_{p}} + \|F^{(4)}_{p,n}(r,y,Z)\|_{L_{p}}\\
&\lesssim (1+ \max_{i \in \{1,2\}}[D^{u^i}]_{\mathscr{V}_{2p}^\beta})\Big([u^1,u^2]_{\mathscr{S}_p^{1/2}[S,a]} + \|u^1(S,\cdot) - u^2(S,\cdot)\|_{\mathbb{B}(\mathbb{T}, L_p)}\Big) (a-s)^{1/2}\\
&\qquad\qquad+\|F_{p,n}^{(4)}(r,y,Z)\|_{L_{p}}.
\end{equs}
where we again used the assumption that $\beta \geq \frac{1}{2}$. Putting this into (\ref{boundonBZ}), we can see that
\begin{equs}
&\|B(Z)\|_{L_{p} }\\
&\lesssim (1+ \max_{i \in \{1,2\}}[D^{u^i}]_{\mathscr{V}_{2p}^\beta})\Big([u^1,u^2]_{\mathscr{S}_p^{1/2}[S,a]} + \|u^1(S,\cdot) - u^2(S,\cdot)\|_{\mathbb{B}(\mathbb{T}, L_p)}\Big) (a-s)^{1/2}\\
&\qquad+ \Big(\int_{0}^{t}\int_{\mathbb{T}} p_{t-r}^2(x,y)\|F^{(4)}_{p,n}(r,y,Z)\|_{L_{p}}^2 dydr\Big)^{1/2}.
\end{equs}
By our bounds on $A,B$ we may conclude that
\begin{equs}&\|F^{(4)}_{p,n}(t,x,Z)\|_{L_{p}}\\
&\lesssim (1+ \max_{i \in \{1,2\}}[D^{u^i}]_{\mathscr{V}_{2p}^\beta})\Big([u^1,u^2]_{\mathscr{S}_p^{1/2}[S,a]} + \|u^1(S,\cdot) - u^2(S,\cdot)\|_{\mathbb{B}(\mathbb{T}, L_p)}\Big)(a-s)^{1/2} \\
&\qquad+ \Big(\int_{0}^{t}\int_{\mathbb{T}} p_{t-r}^2(x,y)\|F^{(4)}_{p,n}(r,y,Z)\|_{L_{p}}^2 dydr\Big)^{1/2}.
\end{equs}
By (\ref{gnisfinite})  the norm in the integrand is bounded in $(r,y)$. Hence by \cref{Gronwalltypelemma} the proof is finished.
\end{proof}

\begin{lemma}[Four point BDG inequality for driftless approximations]\label{4pointBDG}
Let $p \in [2,\infty)$, $\sigma \in C^2$, $\alpha \in (-1,0)$, $\beta \in [0, 1+\alpha/4]$ and for $i=1,2$, let $b^i \in C^\alpha$ and suppose that $u^i \in \mathscr{U}^\beta$ are regularised solutions of
$$(\partial_t - \Delta)u^i = b^i(u^i) + \sigma(u^i)\xi(dy,dr).$$ There exists a constant $N = N(p, \|\sigma\|_{C^2},  \varepsilon, \alpha, \beta)$ such that for $(s,t) \in [0,1]_{\leq}^2$ we have
\begin{equs}&\Big\|\int_s^t \int_{\mathbb{T}} p_{t-r}(x,y)\\
&\qquad\times\Big(\sigma(u^1(r,y)) - \sigma(u^2(r,y))- \sigma(\phi^{u^1(s,\cdot),s}(r,y)) + \sigma(\phi^{u^2(s,\cdot),s}(r,y))\Big) \xi(dy,dr)\Big\|_{L_p}
\\
&\leq 
N[D^{u^1}]_{\mathscr{V}_{2p}^\beta} \|u^1(s,\cdot) -u^2(s,\cdot)\|_{\mathbb{B}(\mathbb{T}, L_p)}(t-s)^{ \frac{1}{4}+\beta}\\
&\qquad+ N\Big(\int_s^t\int_{\mathbb{T}} p_{t-r}^2(x,y)\big\| u^1(r,y) - u^2(r,y) - \phi^{u^1(s,\cdot),s}(r,y) + \phi^{u^2(s,\cdot),s}(r,y)\big\|_{L_p}^2 dydr \Big)^{1/2}.
\end{equs}
\end{lemma}
\begin{proof}
By the BDG inequality and by (\ref{4pointinequality}) in \cref{4pointcompar} we have
\begin{equs}
&\Big\|\int_s^t \int_{\mathbb{T}} p_{t-r}(x,y)\Big(\sigma(u^1(r,y))- \sigma(u^2(r,y))  \\
&\qquad\qquad\qquad\qquad -\sigma(\phi^{u^1(s,\cdot),s}(r,y)) + \sigma(\phi^{u^2(s,\cdot),s}(r,y))\Big) \xi(dy,dr)\Big\|_{L_p}^2\\
&\lesssim \int_s^t \int_{\mathbb{T}}p_{t-r}^2(x,y)\Big\|\sigma(u^1(r,y))- \sigma(u^2(r,y))\\
&\qquad\qquad\qquad\qquad - \sigma(\phi^{u^1(s,\cdot),s}(r,y)) + \sigma(\phi^{u^2(s,\cdot),s}(r,y))\Big\|_{L_p}^2 dydr\\
& \lesssim I+J
\end{equs}
with
\begin{align*}
I &:= \int_s^t \int_{\mathbb{T}} p_{t-r}^2(x,y)\Big\|\Big(\phi^{u^1(s,\cdot),s}(r,y) - \phi^{u^2(s,\cdot),s}(r,y)\Big)\Big(\phi^{u^1(s,\cdot),s}(r,y) - u^1(r,y)\Big)\Big\|_{L_p}^2dydr\\
J &:= \int_s^t \int_{\mathbb{T}} p_{t-r}^2(x,y)\|u^1(r,y) - u^2(r,y) - \phi^{u^1(s,\cdot),s}(r,y) + \phi^{u^2(s,\cdot),s}(r,y)\|_{L_p}^2 dydr.
\end{align*}
By the tower rule, H\"older's inequality,  \cref{driftlessapprox}, \cref{markovproperty} and \cref{lipschitzflowgeneral} (where we recall the definition (\ref{F2def}) of $F^{(2)}$), we can see that
\begin{equs}
&\Big\|\Big(\phi^{u^1(s,\cdot),s}(r,y) - \phi^{u^2(s,\cdot),s}(r,y)\Big)\Big(\phi^{u^1(s,\cdot),s}(r,y) - u^1(r,y)\Big)\Big\|_{L_p} \\
&\leq \Big\| \big\|\phi^{u^1(s,\cdot),s}(r,y) - \phi^{u^2(s,\cdot),s}(r,y)\big\|_{L_{2p}|\mathscr{F}_s}\big\|\phi^{u^1(s,\cdot),s}(r,y) - u^1(r,y)\big\|_{L_{2p}|\mathscr{F}_s}
\Big\|_{L_p}\\
&\lesssim [D^{u^1}]_{\mathscr{V}_{2p
}^\beta}(t-s)^{\beta} \Big\| \big\|\phi^{u^1(s,\cdot),s}(r,y) - \phi^{u^2(s,\cdot),s}(r,y)\big\|_{L_{2p}|\mathscr{F}_s}\Big\|_{L_p}\\
&= [D^{u^1}]_{\mathscr{V}_{2p}^\beta}(t-s)^{\beta}\Big\|F^{(2)}_{2p,0}\big(r-s,y, u^1(s,\cdot), u^2(s,\cdot)\big)\Big\|_{L_p}\\
&\lesssim [D^{u^1}]_{\mathscr{V}_{2p}^\beta}(t-s)^{\beta}\|u^1(s,\cdot) - u^2(s,\cdot)\|_{\mathbb{B}(\mathbb{T},L_p)}.
\end{equs}
 Putting this bound into the definition of $I$, we get that
\begin{equs}
 I &\lesssim [D^{u^1}]_{\mathscr{V}_{2p}^\beta}^2(t-s)^{2\beta}\|u^1(s,\cdot) - u^2(s,\cdot)\|_{C^(\mathbb{T},L_p)}^2\int_s^t\int_{\mathbb{T}}p_{t-r}^2(x,y)dydr\\
 &\lesssim [D^{u^1}]_{\mathscr{V}_{2p}^\beta}^2(t-s)^{2\beta + \frac{1}{2}}\|u^1(s,\cdot) - u^2(s,\cdot)\|_{\mathbb{B}(\mathbb{T},L_p)}^2.
\end{equs}
This bound on $I$ and the definition of $J$ together give the desired bound.
\end{proof}

\section{Regularisation estimates}\label{regularisationsection}
Let $u,u^1,u^2$ be regularised solutions of (\ref{SHE}) with potentially different drift terms, $f$ be a measurable kernel on $(0,1]\times\T$ and $g$ be a smooth function on $\R$.
In this section, we obtain quantitative bounds for expressions of the forms $\int_s^t \int_{\mathbb{T}} f_r(y)g(u(r,y))dydr$ and $\int_s^t \int_{\mathbb{T}}f_r(y)(g(u^1(r,y))- g(u^2(r,y)))dydr$ which depend on a Besov--H\"older norm of $g$ with a negative index.

\begin{lemma}\label{newregularisationlemma}
Let Assumption \ref{assumptions} hold and let $u$ be a regularised solution of (\ref{SHE}). Suppose that $f:(0,1]\times \mathbb{T} \to \mathbb{R}$ is a measurable function such that there exist constants $K>0$ and $\zeta \in [0,\frac{1}{4}]$
 such that for all $t \in (0,1]$ it holds that
$$\int_{\mathbb{T}}|f_t(y)|dy \leq Kt^{-\zeta}.$$
 Let $p \in [1,\infty)$. For all $\lambda \in \big(4\zeta-2,-1\big)\cup(-1,0)$
 and for all $\beta$ in the nonempty set $(\frac{1}{4} - \frac{\lambda}{4} + \zeta, 1 + \frac{\alpha}{4}]$, if $u \in \mathscr{U}_2^\beta$ then there exists a constant $N = N(p,\|\sigma\|_{C^2}, \mu, \lambda, \alpha, \beta, \zeta)$, such that for all $g \in C^\infty$, $(s,t) \in [0,1]_{\leq}^2$ and $\mathscr{G} \in \{\mathscr{F}_s, \{\emptyset, \Omega\}\}$
 we have 
\begin{equs}
    &\Big\| \int_s^t \int_{\mathbb{T}}f_{t-r}(y)g(u(r,y))dydr \Big\|_{L_{p,\infty}^\mathscr{G}}\\
    &\quad\leq N\|g\|_{C^\lambda}K\Big((t-s)^{1 + \frac{\lambda}{4} -\zeta} + [D^u]_{\mathscr{V}_2^\beta[s,t]}(t-s)^{\beta +\frac{\lambda + 3}{4}- \zeta} \Big).
\end{equs}
\end{lemma}
\begin{proof}
We may assume that $p>2$. For $(S,T) \in [0,1]_{\leq}^2$ and  $(s,t) \in [S,T]_{\leq}^2$ we consider the germ
\begin{align*}A_{s,t} := \mathbb{E}^s\int_s^t \int_{\mathbb{T}} f_{T-r}(y)g(\phi^{u(s,\cdot),s}(r,y))dydr.
\end{align*}
 Then
\begin{equs}
|A_{s,t}| 
&\lesssim \int_s^t \int_{\mathbb{T}} |f_{T-r}(y)||\mathbb{E}^s g(\phi^{u(s,\cdot),s}(r,y))|dydr.
\end{equs}
Note that by \cref{malliavinbypartsnew} (with $n=0$) we have
\begin{align*}
|\mathbb{E}^sg(\phi^{u(s,\cdot),s}(r,y))| & \lesssim \|g\|_{C^\lambda}(r-s)^{\lambda/4}.
\end{align*}
By the two inequalities above, by the fact that $\|f_{T-r}\|_{L_1(\mathbb{T})} \leq K(T-r)^{-\zeta} \leq K(t-r)^{-\zeta}$ and by the Cauchy--Schwarz inequality, we have
\begin{equs}\|A_{s,t}\|_{L_{p,\infty}^{\mathscr{G}_s}} &\lesssim \|g\|_{C^\lambda}\int_s^t (r-s)^{\lambda/4} \int_{\mathbb{T}}|f_{T-r}(y)|dydr\\
&\lesssim \|g\|_{C^\lambda }\int_s^t (r-s)^{\lambda/4}(t-r)^{-\zeta}dr\\
& \lesssim \|g\|_{C^\lambda}K(t-s)^{1 + \lambda/4 - \zeta}.
\end{equs}
From the assumption that $\lambda > 4\zeta -2$ it follows that the exponent $1+\lambda/4-\zeta$ is greater than $1/2$, and thus the first condition in (\ref{condiSSL2}) is satisfied. Let $a \in [s,t]$. Then
\begin{align*}
|\mathbb{E}^s \delta A_{s,a,t}| &= |\mathbb{E}^s(A_{s,t} - A_{s,a} - A_{a,t})|\\
& = \Big|\mathbb{E}^s \int_a^t \int_{\mathbb{T}} f_{T-r}(y)\mathbb{E}^a\Big(g(\phi^{u(s,\cdot),s}(r,y)) - g(\phi^{u(a,\cdot),a}(r,y))\Big)dydr\Big|.
\end{align*}
By the Fundamental Theorem of Calculus and \cref{malliavinbypartsnew} (with $n=1$), we get that
\begin{equs}
&\Big|\mathbb{E}^a\Big(g(\phi^{u(s,\cdot),s}(r,y)) - g(\phi^{u(a,\cdot),a}(r,y))\Big)\Big|\\
& = \Big| \int_0^1 \mathbb{E}^a\Big(\nabla g\Big(\theta \phi^{u(s,\cdot),s}(r,y) + (1-\theta) \phi^{u(a,\cdot),a}(r,y)\Big) \Big(\phi^{u(s,\cdot),s}(r,y) - \phi^{u(a,\cdot),a}(r,y)\Big)\Big)d\theta \Big|\\
& = \Big| \int_0^1 \mathbb{E}^a\Big(\nabla g\Big(\theta \phi^{\phi^{u(s,\cdot),s}(a,\cdot),\,a}(r,y) + (1-\theta) \phi^{u(a,\cdot),a}(r,y)\Big)\\
&\qquad\qquad\qquad\qquad \times\Big(\phi^{\phi^{u(s,\cdot),s}(a,\cdot),\,a}(r,y) - \phi^{u(a,\cdot),a}(r,y)\Big)\Big)d\theta \Big|\\
&\lesssim \|g\|_{C^\lambda}(r-a)^{(\lambda-1)/4}\sum_{i=0}^3F_{16,i}^{(2)}(r-a,y,\phi^{u(s,\cdot),s}(a,\cdot), u(a,\cdot))
\end{equs}
where $F^{(2)}$ is defined by (\ref{F2def}).
Therefore using the above result, \cref{lipschitzflowgeneral} and \cref{driftlessapprox}, we get
\begin{equs}
&\mathbb{E}^s\Big|\mathbb{E}^a\Big(g(\phi^{u(s,\cdot),s}(r,y)) - g(\phi^{u(a,\cdot),a}(r,y))\Big)\Big|\\ &\qquad\lesssim \|g\|_{C^\lambda}(r-a)^{ (\lambda-1)/4}\sum_{i=0}^3\mathbb{E}^s F_{16,i}^{(2)}\big(r-a,y,\phi^{u(s,\cdot),s}(a,\cdot), u(a,\cdot)\big)\\
&\qquad\lesssim \|g\|_{C^\lambda}(r-a)^{ (\lambda-1)/4}\sup_{x \in \mathbb{T}}\|\phi^{u(s,\cdot),s}(a,x) - u(a,x)\|_{L_{2,\infty}^{\mathscr{F}_s}}\\
&\qquad\lesssim \|g\|_{C^\lambda}(r-a)^{(\lambda-1)/4}[D^u]_{\mathscr{V}_2^\beta[S,T]}(a-s)^{\beta}.
\end{equs}
By the above inequality, by the assumptions on $f$ and by the fact that $t-a, a-s \leq t-s$, we get
\begin{equs}
\|\mathbb{E}^s \delta A_{s,a,t}\|_{L_\infty} &\lesssim  \mathbb{E}^s \int_a^t \int_{\mathbb{T}}f_{T-r}(y) \|g\|_{C^\lambda}[D^u]_{\mathscr{V}_2^\beta[S,T]}(r-a)^{(\lambda-1)/4}(a-s)^{\beta}dydr\\
&\lesssim \|g\|_{C^\lambda}[D^u]_{\mathscr{V}_2^\beta[S,T]}(a-s)^{\beta}\int_a^t(r-a)^{(\lambda-1)/4 }\int_\mathbb{T} f_{T-r}(y)dy dr\\
&\lesssim \|g\|_{C^\lambda}[D^u]_{\mathscr{V}_2^\beta[S,T]}(a-s)^\beta K\int_a^t(t-r)^{-\zeta}(r-a)^{\lambda/4 - 1/4}dr\\
&\lesssim \|g\|_{C^\lambda}[D^u]_{\mathscr{V}_2^\beta[S,T]}K(t-s)^{\beta - \zeta + \lambda/4 + 3/4}.
\end{equs}
By the assumption that $\beta > 1/4 - \lambda/4+\zeta$, it follows that the exponent $\beta - \zeta + \lambda/4 + 3/4$ is greater than $1$,
and thus the second condition in (\ref{condiSSL2}) is also satisfied. Let
$$\mathscr{A}_{s,t} := \int_s^t\int_{\mathbb{T}}f_{T-r}(y)g(u(r,y))dydr.$$
By the regularity of $g$ and by  \cref{driftlessapprox} we can easily see that (\ref{condiSSL3}) and (\ref{condiSSL4}) are satisfied. All conditions of \cref{conditionalSSL} are satisfied. Consequently, the conclusion follows from \cref{conditionalSSL}
and the fact that $(S,T) \in [0,1]_{\leq}^2$ was arbitrary.
\end{proof}

\begin{corollary}
    \label{otherregularisationlemma}
Let Assumption \ref{assumptions} hold and let $u$ be a regularised solution of (\ref{SHE}) and let $p \in [1,\infty)$. Then for all  $\lambda \in ( -2,-1)\cup(-1,0)$, $\beta \in (\frac{1}{4} - \frac{\lambda}{4}, 1+ \frac{\alpha}{4}]$, if $u \in \mathscr{U}_2^\beta$ then there exists a constant $N = N(p, \|\sigma\|_{C^2},\mu, \lambda, \alpha, \beta )$, such that for all $g \in C^{\infty}$, $(s,t) \in [0,1]_{\leq}^2$, $ x \in \mathbb{T}$, we have 
\begin{equs}&\Big\|\int_s^t \int_{\mathbb{T}} p_{t-r}(x,y) g(u(r,y))dy dr \Big\|_{L_{p,\infty} ^{\mathscr{F}_s}} \\
&\leq N \|g\|_{C^\lambda}\Big((t-s)^{1 + \lambda/4} + [D^u]_{\mathscr{V}_2^\beta[s,t]}(t-s)^{\beta + \frac{\lambda+3}{4}}\Big).
\end{equs}
\end{corollary}
\begin{proof}
Fix $x \in \mathbb{T}$ and for each $(r,y) \in (0,1]\times \mathbb{T}$ define $f_r(y) := p_{r}(x,y)$.  Then $f$ is a measurable function which satisfies
$\int_{\mathbb{T}} f_r(y) dy =1$. Hence, applying \cref{newregularisationlemma} (with $K := 1$, $\zeta := 0$), we obtain the result.
\end{proof}

\begin{corollary}\label{spatialregularisation}
Let Assumption \ref{assumptions} hold and let $u$ be a regularised solution of (\ref{SHE}) and let $p\in [1,\infty)$.  Then for all $\lambda \in (-1,0)$
and for all
$\beta \in (\frac{1}{2} - \frac{\lambda}{4}, 1+ \frac{\alpha}{4}]$, if $u \in \mathscr{U}_2^\beta$ then there exists a constant $N = N(p,\|\sigma\|_{C^2}, \mu, \lambda, \beta)$ such that for all $g \in C^{\infty}$, $(s,t) \in [0,1]_{\leq}^2$ and $x \in \mathbb{T}$ we have
$$\Big\| \int_s^t \int_{\mathbb{T}}\big(p_{t-r}(x,y) - p_{t-r}(\bar{x},y)\big) g(u(r,y))dydr \Big\|_{L_{p,\infty}^{\mathscr{F}_s}} \leq N\|g\|_{C^\lambda}(1+[D^u]_{\mathscr{V}_2^\beta})|x-\bar{x}|^{1/2}.$$
\end{corollary}

\begin{proof}
 Fix $x,\bar{x} \in \mathbb{T}$, and for each $(r,y) \in (0,1]\times \mathbb{T}$, define $f_r(y) := p_{r}(x,y) - p_{r}(\bar{x},y)$.
Then by (\ref{pspacediffL1}), we have $\int_{\mathbb{T}}f_r(y)dy \leq C|x-\bar{x}|^{1/2}r^{-1/4}$ for some constant positive $C$. Hence applying \cref{newregularisationlemma} (with $K: = C|x-\bar{x}|^{1/2}$ and $\zeta := 1/4$), we obtain the stated estimate.
\end{proof}

\begin{lemma}\label{corelemma}
Let $p\in [2,\infty)$, $\alpha \in (-1,0)$, and let $\sigma \in C^4$ such that there exists constant $\mu>0$ such that for all $x \in \mathbb{R}$ we have $\sigma^2(x) \geq \mu^2$. For $i=1,2$, let $b^i \in C^{\alpha}$ and let $u^i$ be regularised solutions of
$$(\partial_t - \Delta)u^i = b^i(u^i) + \sigma(u^i)\xi(dy,dr)$$
in the class $\mathscr{U}^\beta$ for some $\beta \in (\frac{1}{2} - \frac{\alpha}{4}, 1+\frac{\alpha}{4}]$. There exists a constant $N = N(p, \|\sigma\|_{C^4}, \mu, \alpha, \beta)$ such that for all $g \in C^\infty$, $(s,t) \in [0,1]^2_{\leq}$ and $x \in \mathbb{T}$ we have 
\begin{equs}
&\Big\| \int_s^t \int_{\mathbb{T}}p_{t-r}(x,y)\big(g(u^1(r,y)) - g(u^2(r,y))\big) dydr\Big\|_{L_p}
\\
&\leq \,   N\|g\|_{C^\alpha}(1+\max_{i \in \{1,2\}}[D^{u^i}]_{\mathscr{V}_{2p}^\beta})(t-s)^{ (3+\alpha)/4}
\\
& \qquad \qquad\qquad \times \Big([u^1,u^2]_{\mathscr{S}^{1/2}_p[s,t]} +\|u^1(s,\cdot) - u^2(s,\cdot)\|_{\mathbb{B}(\mathbb{T}, L_p)}\Big).
\end{equs}
\end{lemma}
\begin{proof}
Let $(S, T) \in [0,1]_{\leq}^2$, $x \in \mathbb{T}$ and for $(s,t) \in [S, T]^2_\leq$  define the germ
$$A_{s,t}(x) := \mathbb{E}^s \int_s^t \int_{\mathbb{T}} p_{T-r}(x,y)\Big(g(\phi^{u^1(s,\cdot),s}(r,y)) -g(\phi^{
u^2(s,\cdot),s}(r,y))\Big)dydr.$$
We first  bound $\|A_{s,t}\|_{L_p}$.  Using \cref{malliavinbypartsnew}(with $n=1$ and thus $q=8)$ and recalling the definition of $F^{(2)}$ from (\ref{F2def}), we have
\begin{equs} &\Big|\mathbb{E}^s\Big(g(\phi^{u^1(s,\cdot),s}(r,y)) - g(\phi^{u^2(s,\cdot),s}(r,y))\Big)\Big|\\
&=\Big|\int_0^1 \mathbb{E}^s\Big(\nabla g\Big(\theta \phi^{u^1(s,\cdot),s}(r,y) + (1-\theta)\phi^{u^2(s,\cdot),s}(r,y)\Big)\Big(\phi^{u^1(s,\cdot),s}(r,y) - \phi^{u^2(s,\cdot),s}(r,y)\Big)\Big) d\theta\Big|\\
&\lesssim  \|g\|_{C^\alpha} (r-s)^{ (\alpha-1)/4}\sum_{i=0}^2F^{(2)}_{8,i}\big(r-s,y, u^1(s,\cdot), u^2(s,\cdot)\big).
\end{equs}
We take the $L_p$-norm on the inequality, and by \cref{lipschitzflowgeneral} we get
\begin{equs}
&\|\mathbb{E}^s(g(\phi^{u^1(s,\cdot),s}(r,y))  - g(\phi^{u^2(s,\cdot),s}(r,y))\|_{L_p}\\ 
&\qquad\lesssim \|g\|_{C^\alpha}(r-s)^{ (\alpha-1)/4}\|u^1(s,\cdot) - u^2(s,\cdot)\|_{\mathbb{B}(\mathbb{T}, L_p)}.
\end{equs}
Using the definition of $A$, and the above inequality, we get
\begin{equs}
\|A_{s,t}(x)\|_{L_p} &\lesssim \int_s^t \int_{\mathbb{T}}p_{T-r}(x,y)\|g\|_{C^\alpha}(r-s)^{ (\alpha-1)/4}\|u^1(s,\cdot) - u^2(s,\cdot)\|_{\mathbb{B}(\mathbb{T},L_p)}dydr\\
&\lesssim \|g\|_{C^\alpha}\Big([u^1,u^2]_{\mathscr{S}_p^{1/2}[S,T]} + \|u^1(S,\cdot) - u^2(S,\cdot)\|_{\mathbb{B}(\mathbb{T}, L_p)}\Big)(t-s)^{(3+ \alpha)/4}.
\label{eq:Ast_bound}
\end{equs}
We proceed with an estimate for $\|\mathbb{E}^s\delta A_{s,a,t}\|_{L_p}$ for $a \in [s,t]$. Note that 
\begin{equs}
|\mathbb{E}^s\delta A_{s,a,t}| 
&= |\mathbb{E}^s(A_{s,t}- A_{s,a} - A_{a,t})|
\\
& = \Big|\mathbb{E}^s \int_a^t \int_{\mathbb{T}}p_{T-r}(x,y) \mathbb{E}^a\Big(g(\phi^{u^1(s,\cdot),s}(r,y)) - g(\phi^{u^2(s,\cdot),s}(r,y))
\\
& \qquad \qquad\qquad \qquad \qquad - g(\phi^{u^1(a,\cdot),a}(r,y)) + g(\phi^{u^2(a,\cdot),a}(r,y))\Big)dydr\Big|.
\end{equs}
For $(r,y) \in [0,1]\times \mathbb{T}$ and  $z \in (C(\mathbb{T}))^4$, we define
\begin{equs}{\Gamma}_{r,y}(z) :=\big|\mathbb{E}\big(g(\phi^{z_1}(r-a,y)) - g(\phi^{z_2}(r-a,y)) - g(\phi^{z_3}(r-a,y)) + g(\phi^{z_4}(r-a,y))\big)\big|.
\end{equs}
For brevity, we fix $(r,y) \in [s,t]\times \mathbb{T}$ and we set ${\Gamma} := {\Gamma}_{r,y}$, 
    $\phi_i :=\phi_{r-a}^{z_i}(y)$ and $\delta_{i,j}:= \phi_j - \phi_i$. 
By \cref{4pointcompar} we get
\begin{equs}{\Gamma}(z)
&\leq \Big|\int_0^1 \int_0^1 \mathbb{E}\Big( \delta_{1,2}\big(\theta\delta_{1,3} + (1-\theta)\delta_{2,4}\big) \nabla^2 g(\Theta_1(\theta,\eta)) \Big)d\theta d\eta\Big|\\
&\qquad+ \Big| \int_0^1 \mathbb{E}\big((\delta_{3,4} - \delta_{1,2}) \nabla g(\Theta_2(\theta)) \big) d\theta\Big|
\end{equs}
where $\Theta_1,\Theta_2$ are convex combinations\footnote{In particular:
\begin{equs}
\Theta_1(\theta,\eta) &:= \eta(\theta \phi_1 + (1-\theta)\phi_2) + (1-\eta)(\theta \phi_3 + (1-\theta) \phi_4),\qquad
\Theta_2(\theta) &:= \theta \phi_3 + (1-\theta) \phi_4,
\end{equs}
but this is not important for the proof.} of $\phi_1,\dots, \phi_4$. 
Hence using \cref{mainmalliavintool} (with $n=2$ and $n=1$ for the first and second terms respectively) and recalling the definition of $F^{(4)}$ from (\ref{F4def}) we get that
\begin{equs}
&{\Gamma}(z)\lesssim \|g\|_{C^{\alpha}}(r-a)^{-1/2 + \alpha/4}\int_0^1\|\delta_{1,2}(\theta\delta_{1,3} + (1-\theta)\delta_{2,4})\|_{\mathscr{W}_{16}^3}d\theta\\&\qquad+ \|g\|_{C^\alpha}(r-a)^{-1/4 + \alpha/4}\|\delta_{3,4} - \delta_{1,2}\|_{\mathscr{W}_8^2}\\
& \lesssim \|g\|_{C^{\alpha}}(r-a)^{-1/2 + \alpha/4}\sum_{i=0}^3F^{(2)}_{32,i}(r-a,y,z_1,z_2)\\
&\qquad\qquad\times\Big(F^{(2)}_{32,i}(r-a,y,z_1,z_3) + F^{(2)}_{32,i}(r-a,y,z_2,z_4)\Big)
\\&\qquad+ \|g\|_{C^\alpha}(r-a)^{-1/4 + \alpha/4}\sum_{i=0}^2F^{(4)}_{8,i}(r-a,y,z). \label{gammabound}
\end{equs}
Let
$$Z := (\phi_{s,a}^{u^1(s,\cdot)}, \phi_{s,a}^{u^2(s,\cdot)}, u^1(a,\cdot), u^2(a,\cdot)).$$
By \cref{lipschitzflowgeneral} and \cref{driftlessapprox}, 
we have that for $l=1,2$, $i \in \mathbb{Z}_{\geq 0}$ that
\begin{equs}\| F^{(2)}_{32,i}(r-a,y,Z_l, Z_{l+2})\|_{L_2|\mathscr{F}_s} \lesssim \sup_{x \in \mathbb{T}}\|Z_l(x) - Z_{l+2}(x)\|_{L_{2,\infty}^{\mathscr{F}_s}} \lesssim [D^{u^l}]_{\mathscr{V}_p^\beta}(t-s)^{\beta}.
\end{equs}
Using this, and \cref{lipschitzflowgeneral}, we can see that
\begin{equs}
&\|\mathbb{E}^s\Big(F_{32,i}^{(2)}(r-a,y,Z_1,Z_2)F_{32,i}^{(2)}(r-a,y,Z_l, Z_{l+2})\Big)\|_{L_p}\\
&\lesssim \|\| F_{32,i}^{(2)}(r-a,y,Z_1,Z_2)\|_{L_2|\mathscr{F}_s}\|F_{32,i}^{(2)}(r-a,y,Z_l, Z_{l+2})\|_{L_2|\mathscr{F}_s}\|_{L_p}\\
&\lesssim [D^{u^l}]_{\mathscr{V}_p^\beta}(t-s)^{\beta} \| \| F^{(2)}_{32,i}(r-a,y,Z_1,Z_2)\|_{L_2|\mathscr{F}_s}\|_{L_p}\\
&\lesssim [D^{u^l}]_{\mathscr{V}_p^\beta}(t-s)^{\beta} \sup_{x \in \mathbb{T}}\|\phi^{u^1(s,\cdot),s}(a,x) - \phi^{u^2(s,\cdot),s}(a,x) \|_{L_{2,p}^{\mathscr{F}_s}}\\\
&\lesssim \max_{l \in \{1,2\}}[D^{u^l}]_{\mathscr{V}_{2p}^\beta}(t-s)^{\beta} \sup_{x \in \mathbb{T}}\|u^1(s,x) - u^2(s,x)\|_{L_p}. \label{F12Fijbound}
\end{equs}
Note moreover that by \cref{4pointestimate},
\begin{equs}
\sum_{i=0}^2\| &\mathbb{E}^sF_{8,i}^{(4)}(r-a,y,Z)\|_{L_p}\\
&  \lesssim (1+\max_{l \in \{1,2\}}[D^{u^l}]_{\mathscr{V}_{2p}^\beta}) \Big([u^1,u^2]_{\mathscr{S}_p^{1/2}[S,T]} + \|u^1(S,\cdot) - u^2(S,\cdot)\|_{\mathbb{B}(\mathbb{T}, L_p)}\Big)|t-s|^{1/2}. \label{F4boundforquadrature}
\end{equs}
Using (\ref{gammabound}), (\ref{F12Fijbound}) and (\ref{F4boundforquadrature}), we get that
\begin{equs}
&\|\mathbb{E}^s{\Gamma}(Z)\|_{L_p} \\
&\qquad\lesssim \|g\|_{C^\alpha}(r-a)^{-1/2 + \alpha/4}\sum_{i=0}^3 \sum_{l \in \{1,2\}}\| \mathbb{E}^s \Big(F_{32,i}^{(2)}(r-a,y,Z_1, Z_2) F_{32,i}^{(2)}(r-a,y,Z_l, Z_{l+2})\Big)\|_{L_p}\\
&\qquad\qquad + \|g\|_{C^\alpha}(r-a)^{-1/4 +\alpha/4}\sum_{i=0}^2 \| \mathbb{E}^s F_{8,i}^{(4)}(r-a,y, Z)\|_{L_p}\\
&\lesssim \|g\|_{C^\alpha}(r-a)^{-1/2 + \alpha/4} \max_{l \in \{1,2\}}[D^{u^l}]_{\mathscr{V}_{2p}^\beta}(t-s)^{\beta}\sup_{x \in \mathbb{T}}\|u^1(s,x) - u^2(s,x)\|_{L_p} \\
&\qquad+ \|g\|_{C^\alpha}(r-a)^{-1/4 + \alpha/4}\\
&\qquad\qquad \times\sum_{i=0}^2(1+\max_{l \in \{1,2\}}[D^{u^l}]_{\mathscr{V}_{2p}^\beta})([u^1, u^2]_{\mathscr{S}_p^{1/2}[S,T]} + \|u^1(S,\cdot) - u^2(S,\cdot)\|_{\mathbb{B}(\mathbb{T}, L_p)})(t-s)^{1/2}\\
&\lesssim \|g\|_{C^\alpha}(1+\max_{l \in \{1,2\}}[D^{u^l}]_{\mathscr{V}_{2p}^\beta})([u^1, u^2]_{\mathscr{S}_p^{1/2}[S,T]} + \|u^1(S,\cdot) - u^2(S,\cdot)\|_{\mathbb{B}(\mathbb{T}, L_p)})\\
&\qquad \times\Big((r-a)^{-1/2 + \alpha/4}(t-s)^{\beta} + (r-a)^{-1/4 + \alpha/4}(t-s)^{1/2} \Big).
\end{equs}
Therefore, by using \cref{markovproperty} and the above bound, we get 
\begin{equs}
&\|\mathbb{E}^s \delta A_{s,a,t}\|_{L_p}\\
& \lesssim   \int_a^t \int_{\mathbb{T}} p_{T-r}(x,y)\|\mathbb{E}^s \Gamma_{r,y}(Z)\|_{L_p} dydr\\
& \lesssim \|g\|_{C^\alpha}(1+\max_{l \in \{1,2\}}[D^{u^l}]_{\mathscr{V}_{2p}^\beta})([u^1, u^2]_{\mathscr{S}_p^{1/2}[S,T]} + \|u^1(S,\cdot) - u^2(S,\cdot)\|_{\mathbb{B}(\mathbb{T}, L_p)})\\
&\qquad \times\Big((t-s)^{\beta}\int_a^t \int_{\mathbb{T}}p_{t-r}(x,y)(r-a)^{-1/2 + \alpha/4}dydr\\
&\qquad\qquad\qquad+ (t-s)^{1/2}\int_a^t \int_{\mathbb{T}}p_{t-r}(x,y)(r-a)^{-1/4 + \alpha/4}dydr\Big)\\
&\lesssim \|g\|_{C^\alpha}(1+\max_{l \in \{1,2\}}[D^{u^l}]_{\mathscr{V}_{2p}})([u^1, u^2]_{\mathscr{S}_p^{1/2}[S,T]} + \|u^1(S,\cdot) - u^2(S,\cdot)\|_{\mathbb{B}(\mathbb{T}, L_p)})\\
&\qquad \times\Big((t-s)^{\beta}(t-a)^{1/2 + \alpha/4} + (t-s)^{1/2}(t-a)^{3/4 + \alpha/4}\Big)\\
&\lesssim \|g\|_{C^\alpha}(1+\max_{l \in \{1,2\}}[D^{u^l}]_{\mathscr{V}_{2p}})([u^1, u^2]_{\mathscr{S}_p^{1/2}[S,T]} + \|u^1(S,\cdot) - u^2(S,\cdot)\|_{\mathbb{B}(\mathbb{T}, L_p)})\\
&\qquad \times\Big((t-s)^{\beta + 1/2 + \alpha/4} + (t-s)^{5/4 + \alpha/4}\Big).
\label{eq:delta_Ast}
\end{equs}
Note that $\beta + 1/2 + \alpha/4$ and $5/4 + \alpha/4$ are greater than $1$ by the assumptions that $\beta > 1/2 - \alpha/4$ and that $\alpha >-1$.
Consequently, by \eqref{eq:Ast_bound} and \eqref{eq:delta_Ast}, we have that the condition \eqref{condiSSL2} is satisfied. In addition, by using \cref{driftlessapprox} and the regularity of $g$, it is straightforward to see that the process
\begin{align*}
	\mathscr{A}_t:=\int_0^t \int_{\mathbb{T}}p_{T-r}(x,y)\big(g(u^1(r,y)) - g(u^2(r,y))\big) dydr
\end{align*} 
satisfies 
\eqref{condiSSL3} and \eqref{condiSSL4}. Consequently, the conclusion follows from \cref{lem:conditionalSSL}
and the fact that $(S,T) \in [0,1]_{\leq}^2$ was arbitrary.
\end{proof}

\begin{corollary}\label{biuidriftcomparison}
Let $p \in [2,\infty)$ and let $\sigma \in C^4$ such that there exists a constant $\mu>0$ such that $\sigma^2(x) \geq \mu^2$. For $i=1,2$, let $b^i \in C^{\alpha}$ and let $u^i$ be regularised solutions of
$$(\partial_t - \Delta)u^i = b^i(u^i) + \sigma(u^i)\xi(dy,dr)$$
in the class $\mathscr{U}^\beta$ for some $\beta \in \big(\frac{1}{2} - \frac{\alpha}{4}, 1+\frac{\alpha}{4}\big]$. There exists a constant $N = N(p, \|\sigma\|_{C^4}, \mu, \alpha, \beta)$ such that for all $g^1, g^2 \in C^\infty$, $(s,t) \in [0,1]^2_{\leq}$, we have 
\begin{equs}
&\Big\|\int_s^t \int_{\mathbb{T}} p_{t-r}(x,y)\big(g^1(u^1(r,y)) - g^2(u^2(r,y))\big)dydr\Big\|_{L_p}\\
&\leq N\big(1+\max_{i \in \{1,2\}}[D^{u^i}]_{\mathscr{V}_{2p}^\beta}\big)|t-s|^{(3+\alpha)/4}\\
&\qquad\qquad\times\Big(\|g^1 - g^2\|_{C^{\alpha-1}} +  \|g^2\|_{C^\alpha}\Big([u^1,u^2]_{\mathscr{S}_p^{1/2}[s,t]} + \|u^1(s,\cdot) - u^2(s,\cdot)\|_{\mathbb{B}(\mathbb{T}, L_p)}\Big)\Big).
 \end{equs}
\end{corollary}

\begin{proof}
Since $\beta >\frac{1}{2} - \frac{\alpha}{4} = \frac{1}{4} - \frac{\alpha-1}{4} $, we can see that the condition of \cref{otherregularisationlemma} is satisfied with $\lambda = \alpha-1$.
The desired result follows from \cref{otherregularisationlemma} and \cref{corelemma} by the triangle inequality. 
\end{proof}

\section{The $\mathscr{S}_p$-bracket of two solutions}\label{Spsection}
Throughout the section we work with the following assumption:
\begin{assumption}\label{bucklingsectionassumptions}
Let $\sigma \in C^4$ such that there exists constant $\mu>0$ such that $\sigma^2(x) \geq \mu^2$.   Let $\alpha \in (-1,0)$, $\beta \in (\frac{1}{2} - \frac{\alpha}{4}, 1+\frac{\alpha}{4}]$, and suppose that for $i=1,2$ we are given $b^i \in C^{\alpha+}$ and that $u^i$ are regularised solutions of
$$(\partial_t - \Delta)u^i = b^i(u^i) + \sigma(u^i)\xi(dy,dr)$$
in the class $\mathscr{U}^\beta$ with initial conditions $u^i(0,\cdot) = u_0^i \in C(\mathbb{T})$. 
\end{assumption}
Recall the definition of the $\mathscr{S}_p^{1/2}$-bracket from (\ref{Spbracket}).
Informally, the aim of this section is to show that
$$[u^1,u^2]_{\mathscr{S}_p^{1/2}[0,1]} \lesssim \|u^1_0 - u^2_0\|_{\mathbb{B}(\mathbb{T})}+ \|b^1 - b^2\|_{C^{\alpha-1}}.$$

\begin{lemma}\label{Shalfofunum} Let Assumption \ref{bucklingsectionassumptions} hold and let $p \in [2,\infty)$.
 Then $[u^1,u^2]_{\mathscr{S}_p^{1/2}} <\infty$. Moreover there exists a constant $N = N(p, \mu,\|\sigma\|_{C^4}, \alpha, \beta)$ such that
\begin{equs}\,&[u^1,u^2]_{\mathscr{S}_p^{1/2}[s,t]}\\
&\leq N \big(1+\max_{i \in \{1,2\}}[D^{u^i}]_{\mathscr{V}_{2p}^\beta}\big)(1+\|b^2\|_{C^\alpha})\\
&\qquad\times\Big(\|b^1 - b^2 \|_{C^{\alpha-1}}+[u^1,u^2]_{\mathscr{S}_p[s,t]} + \|u^1(s,\cdot)- u^2(s,\cdot)\|_{\mathbb{B}(\mathbb{T}, L_p)}  \Big)(t-s)^{ (1+\alpha)/4} .\label{globalbucklingequatio}
\end{equs}
\end{lemma}
\begin{proof}
Let $(S,T) \in [0,1]_{\leq}^2$. We begin by verifying that the $\mathscr{S}_p^{1/2}$-bracket is finite. By the triangle inequality, and by \cref{driftlessapprox}, we have for $(s,t) \in [S,T]_{\leq}^2$ that
\begin{equs}
&\sup_{x \in \mathbb{T}}\|u^1(t,x) - \phi^{u^1(s,\cdot),s}(t,x) - u^2(t,x) + \phi^{u^2(s,\cdot),s}(t,x)\|_{L_{p,\infty}^{\mathscr{F}_s}}\\
&\leq \max_{i \in \{1,2\}}\sup_{x \in \mathbb{T}}\|u^i(t,x) - \phi^{u^i(s,\cdot),s}(t,x)\|_{L_{p,\infty}^{\mathscr{F}_s}}\\
&\lesssim \max_{i \in \{1,2\}}[D^{u^i}]_{\mathscr{V}_p^\beta}(t-s)^{\beta}\lesssim \max_{i \{1,2\}}[D^{u^i}]_{\mathscr{V}_p^\beta}(t-s)^{1/2},
\end{equs}
where we used that by assumption we have $\beta \geq \frac{1}{2}-\frac{\alpha}{4} > \frac{1}{2}$. Thus by the fact that $\|\cdot\|_{L_p} = \| \|\cdot\|_{L_p|\mathscr{F}_s}\|_{L_p} \leq \|\cdot\|_{L_{p,\infty}^{\mathscr{F}_s}}$, it follows that 
\begin{equs}\,[u^1,u^2]_{\mathscr{S}_p^{1/2}} \lesssim \max_{i \in \{1,2\}}[D^{u^i}]_{\mathscr{V}_p^\beta},\label{Spbracketisbdd}
\end{equs}
which is finite, since by assumption $u^i \in \mathscr{U}^\beta$.
Note that for $(s,t) \in [0,1]_{\leq}^2$, $x \in \mathbb{T}$ we have by (\ref{phidef1}) and by (\ref{integraleqn}) that
\begin{equs}
&u^1(t,x) - u^2(t,x) - \phi^{u^1(s,\cdot),s}(t,x) +\phi^{u^2(s,\cdot),s}(t,x) =\\
&= u^1(t,x) - u^2(t,x) - P_{t-s}\big(u^1(s,\cdot)
- u^2(s,\cdot)\big)(x)\\ &\qquad\qquad- \int_s^t \int_{\mathbb{T}}p_{t-r}(x,y)\big(\sigma(\phi^{u^1(s,\cdot),s}(r,y)) - \sigma(\phi^{u^2(s,\cdot),s}(r,y)\big)\xi(dy,dr)\\
&= \Big(D^{u^1}_t -D^{u^2}_t  - P_{t-s}D^{u^1}_s + P_{t-s}D^{u^2}_s\Big)(x) \\
&\qquad + \int_s^t \int_{\mathbb{T}}p_{t-r}(x,y)\big(\sigma(u^1(r,y)) - \sigma(u^2(r,y))\\
&\qquad\qquad\qquad\qquad\qquad\qquad- \sigma(\phi^{u^1(s,\cdot),s}(r,y))+ \sigma(\phi^{u^2(s,\cdot),s}(r,y))\big) \xi(dy,dr)\\
& =: I(t,x)+J(t,x).
\end{equs}
For $i=1,2$ let $(b^{i,n})_{n \in \mathbb{N}} \subset C^\infty$ with $b^{i,n} \to b^i$ in $C^\alpha$. 
Then by Definition \ref{defofsolution} and by Fatou's lemma, we have
\begin{equs}&\sup_{(t,x) \in [0,1]\times \mathbb{T}}\|I(t,x)\|_{L_p}\\
&\lesssim \liminf_{n \to \infty}\sup_{(t,x) \in [0,1]\times \mathbb{T}}\Big\| \int_s^t \int_{\mathbb{T}}p_{t-r}(x,y) \big(b^{1,n}(u^1(r,y)) - b^{2,n}(u^2(r,y))\big)dydr\Big\|_{L_p}
\end{equs}
 So by \cref{biuidriftcomparison} we have 
\begin{align*}
\|I(t,x)\|_{L_p}&\lesssim  |t-s|^{(3+\alpha)/4}(1+\max_{i \in \{1,2\}}[D^{u^i}]_{\mathscr{V}_{2p}^\beta})\\
&\qquad \times\Big(\|b^1 -b^2\|_{C^{\alpha-1}} + \|b^2\|_{C^\alpha}\Big([u^1,u^2]_{\mathscr{S}_p^{1/2}[s,t]} + \|u^1(s,\cdot) - u^2(s,\cdot)\|_{\mathbb{B}(\mathbb{T}, L_p)}\Big)\Big).
\end{align*}
Note moreover that by \cref{4pointBDG} we have
\begin{equs}
&\|J(t,x)\|_{L_p} \\
&\lesssim [D^{u^1}]_{\mathscr{V}_{2p}^\beta}\|u^1(s,\cdot) -u^2(s,\cdot)\|_{\mathbb{B}(\mathbb{T}, L_p)}(t-s)^{ \frac{1}{4}+\beta}\\
&+ \Big(\int_s^t\int_{\mathbb{T}} p_{t-r}(x,y)\big\| u^1(r,y) - u^2(r,y) - \phi^{u^1(s,\cdot),s}(r,y) + \phi^{u^2(s,\cdot),s}(r,y)\big\|_{L_p}^2 dydr \Big)^{1/2}.
\end{equs}
By our bounds on $I, J$ and by the observation that $\frac{1}{4} + \beta > \frac{1}{4} + \frac{1}{2} - \frac{\alpha}{4} > \frac{3}{4} + \frac{\alpha}{4}$, we conclude that
\begin{align*}
&\|u^1(t,x) - u^2(t,x) - \phi^{u^1(s,\cdot),s}(t,x) +\phi^{u^2(s,\cdot),s}(t,x)\|_{L_p}^2 \lesssim\\
&\lesssim (1+\max_{i \in \{1,2\}}[D^{u^i}]_{\mathscr{V}_{2p}^\beta})^2(1+\|b^2\|_{C^\alpha})^2\\
&\qquad\qquad
\times\Big(\|b^1 - b^2 \|_{C^{\alpha-1}}+[u^1,u^2]_{\mathscr{S}_p^{1/2}[s,t]} + \|u^1(s,\cdot)- u^2(s,\cdot)\|_{\mathbb{B}(\mathbb{T}, L_p)}  \Big)^2(t-s)^{(3+\alpha)/2} \\
& \qquad\qquad + \int_s^t\int_{\mathbb{T}} p_{t-r}(x,y)\big\| u^1(r,y) - u^2(r,y) - \phi^{u^1(s,\cdot),s}(r,y) + \phi^{u^2(s,\cdot),s}(r,y)\big\|_{L_p}^2 dydr.
\end{align*}
Note that the norm in the integrand is bounded in $(r,y)$, since it is bounded by $[u^1,u^2]_{\mathscr{S}_p^{1/2}}$, which is finite by (\ref{Spbracketisbdd}).
Using \cref{Gronwalltypelemma}, and \cref{bracketlemma} (where we recall that $S \leq s \leq t$), we get
\begin{equs}
&\|u^1(t,x) - u^2(t,x) - \phi^{u^1(s,\cdot),s}(t,x) +\phi^{u^2(s,\cdot),s}(t,x)\|_{L_p}\\
&\lesssim (1+\max_{i \in \{1,2\}}[D^{u^i}]_{\mathscr{V}_{2p}^\beta})(1+\|b^2\|_{C^\alpha})\\
&\qquad\qquad\times\Big([u^1,u^2]_{\mathscr{S}_p^{1/2}[S,t]} + \|u^1(S,\cdot)- u^2(S,\cdot)\|_{\mathbb{B}(\mathbb{T}, L_p)} +\|b^1 - b^2 \|_{C^{\alpha-1}} \Big)(t-s)^{ (3+\alpha)/4}.
\end{equs}
Therefore dividing both sides by $(t-s)^{1/2}$ and taking supremum over $(s,t) \in [S,T]^2_<$, we obtain the desired bound with $(S,T)$ in place of $(s,t)$. Now the desired result follows by the fact that $(S,T) \in [0,1]_{\leq}^2$ was arbitrary.
\end{proof}

\begin{lemma}[Splitting the $\mathscr{S}_p^{1/2}$-bracket]\label{Spbracketistriangular} Let Assumption \ref{bucklingsectionassumptions} hold and let $p \in [2,\infty)$.
There exists a constant $N = N(p,\|\sigma\|_{C^{4}}, \mu, \alpha, \beta)$ such for all $(S,T) \in [0,1]_{\leq}^2$ and $Q \in [S,T]$ we have
\begin{equs}\,[u^1,u^2]_{\mathscr{S}_p^{1/2}[S,T]} &\leq N(1+ \max_{i \in \{1,2\}}[D^{u^i}]_{\mathscr{V}_{2p}^\beta})\Big([u^1,u^2]_{\mathscr{S}_p^{1/2}[S,Q]} + \|u^1(S,\cdot) - u^2(S,\cdot)\|_{\mathbb{B}(\mathbb{T}, L_p)}\Big)\\
&\qquad+2[u^1,u^2]_{\mathscr{S}_p^{1/2}[Q,T]}.
\end{equs}
\end{lemma}
\begin{proof}
For $(s,t) \in [0,1]_{\leq}^2$, we  set 
$$
A(s,t) := \sup_{x \in \mathbb{T}}\|u^1(t,x) - u^2(t,x) - \phi^{u^1(s,\cdot),s}(t,x) +\phi^{u^1(s,\cdot),s}(t,x) \|_{L_p}. 
$$
For $(s,t) \in [S,Q]^2_{\leq}$ or $(s,t) \in [Q,T]^2_{\leq}$, we clearly have 
\begin{equs}                \label{eq:trivial_split}
    A(s,t) \leq \big([u^1,u^2]_{\mathscr{S}_p^{1/2}[S,Q]}+[u^1,u^2]_{\mathscr{S}_p^{1/2}[Q,T]}\big)|t-s|^{1/2}.
\end{equs}
For $s \leq Q < t$, by using the triangle inequality and keeping in mind the definition of $F^{(4)}_{p,0}$ (see \eqref{F4def}) we have
\begin{equs}               
    A(s,t) & \leq A(Q, t)
    \\
    &+ \sup_{x \in \mathbb{T}}\|\phi^{u^1(Q,\cdot), Q}(t,x) - \phi^{u^2(Q,\cdot), Q}(t,x) - \phi^{u^1(s,\cdot),s}(t,x) + \phi^{u^2(s,\cdot),s}(t,x)\|_{L_p}
    \\
    &= A(Q, t)+\sup_{x \in \mathbb{T}}\big\|F^{(4)}_{p,0}\big(t-Q,\, x,\, \phi^{u^1(s,\cdot),s}(Q,\cdot),\phi^{u^2(s,\cdot),s}(Q,\cdot), u^1(Q,\cdot), u^2(Q,\cdot)\big)\big\|_{L_p}
\end{equs}
From this and \cref{4pointestimate}, we conclude that for $s \leq Q < t$
\begin{equs}
     A(s,t) & \leq [u^1,u^2]_{\mathscr{S}_p^{1/2}[Q,T]}|t-s|^{1/2}
     \\
     &+ N(1+\max_{i \in \{1,2\}}[D^{u^i}]_{\mathscr{V}_{2p}^\beta})\Big([u^1,u^2]_{\mathscr{S}_p^{1/2}[S,Q]} + \|u^1(S,\cdot) - u^2(S,\cdot)\|_{\mathbb{B}(\mathbb{T}, L_p)}\Big) |t-s|^{1/2}.\qquad \label{Astbracketbound}
\end{equs}
By the above combined with \eqref{eq:trivial_split}, the inequality (\ref{Astbracketbound}) holds for any $(s,t) \in [S, T]^2_{\leq}$,
from which the claim follows. 
\end{proof}

\begin{lemma}\label{triangualityofSp} Let Assumption \ref{bucklingsectionassumptions} hold
and let  $K \in \mathbb{Z}_{\geq 2}$, $p \in [2,\infty)$. There exists a constant $N = N(p,\|\sigma\|_{C^{4}}, K, \mu, \alpha, \beta)$ such that with
$M : = N(1+ \max_{i \in \{1,2\}}[D^{u^i}]_{\mathscr{V}_{2p}})$
we have
$$[u^1,u^2]_{\mathscr{S}_p^{1/2}} \leq (K-1)M^{K-1}\|u^1_0 - u^2_0\|_{\mathbb{B}(\mathbb{T})} + 2\sum_{i=0}^{K-1}M^{i}[u^1,u^2]_{\mathscr{S}_p[\frac{K-i-1}{K}, \frac{K-i}{K}]}.$$
\end{lemma}
\begin{proof}
Let $a_{s,t} := [u^1,u^1]_{\mathscr{S}_p^{1/2}[s,t]}$ and $u_0^{1,2} := \|u^1_0 - u^2_0\|_{\mathbb{B}(\mathbb{T})}$. 
We will begin by using induction to show that for all $n \in \{1,\dots, K-1\}$ we have
\begin{equs}a_{0,1} \leq M^n a_{0,\frac{K-n}{K}} + \Big(\sum_{i=1}^n M^i\Big)u^{1,2}_0+2\sum_{i=0}^{n-1}M^i a_{\frac{K-i -1}{K}, \frac{K-i}{K}}. \label{azeroone}
\end{equs}
 By \cref{Spbracketistriangular} we have that
$$a_{0,1} \leq M(a_{0,\frac{K-1}{K}} + u_0^{1,2}) + 2a_{\frac{K-1}{K}, 1},$$
therefore (\ref{azeroone}) holds for the initial case $n=1$. Now suppose that (\ref{azeroone}) holds for some $n \in \mathbb{N}$. We will show that it also holds for $n+1$. To this end, we first apply the induction  hypothesis and then \cref{Spbracketistriangular}, to get that
\begin{equs}
a_{0,1} &\leq M^n a_{0,\frac{K-n}{K}} + \Big(\sum_{i=1}^n M^i\Big)u_0^{1,2} + 2\sum_{i=0}^{n-1}M^i a_{\frac{K-i -1}{K}, \frac{K-i}{K}}\\
&\leq M^n\Big( M \big(a_{0,\frac{K-n-1}{K}}+ u_0^{1,2}\big)+ 2a_{\frac{K-n-1}{K},\frac{K-n}{K}} \Big) +\Big(\sum_{i=1}^n M^i\Big)u_0^{1,2} + 2\sum_{i=0}^{n-1}M^i a_{\frac{K-i -1}{K}, \frac{K-i}{K}}\\
&= M^{n+1}a_{0,\frac{K-n-1}{K}} + M^{n+1}u_{0}^{1,2} + \Big(\sum_{i=1}^n M^i\Big)u_0^{1,2} + 2M^n a_{\frac{K-n-1}{K}, \frac{K-n}{K}} +  2\sum_{i=0}^{n-1}M^i a_{\frac{K-i -1}{K}, \frac{K-i}{K}}\\
& = M^{n+1}a_{0,\frac{K-n-1}{K}} + \Big(\sum_{i=1}^{n+1} M^i\Big)u_0^{1,2} + 2\sum_{i=0}^{n}M^i a_{\frac{K-i -1}{K}, \frac{K-i}{K}}
\end{equs}
as required. Therefore (\ref{azeroone}) is proven. Now choosing $n:=K-1$ in (\ref{azeroone}), we get
\begin{equs}
a_{0,1} &\leq \Big(\sum_{i=1}^{K-1} M^i\Big)u_0^{1,2} + M^{K-1}a_{0,\frac{1}{K}} + 2 \sum_{i=0}^{K-2}M^i  a_{\frac{K-i -1}{K}, \frac{K-i}{K}}\\
&\leq (K-1)M^{K-1}u_{0}^{1,2} + 2\sum_{i=0}^{K-1}M^i a_{\frac{K-i -1}{K}, \frac{K-i}{K}}
\end{equs}
as required.
\end{proof}

\begin{lemma}\label{Qnb1b2bound} Let Assumption \ref{bucklingsectionassumptions} hold.
For all $p \in [2,\infty)$ there exists a positive constant $K_0 = K_0(\max_{i \in \{1,2\}}[D^{u^i}]_{\mathscr{V}_{2p}^\beta}, p, \|\sigma\|_{C^4}, \mu, \alpha, \beta)$ such that if $K \in \mathbb{Z}$ satisfies $K> K_0$, then there exists a constant $M = {M}(p, \|\sigma\|_{C^1}, \alpha, \beta )$ such that  for all $n \in \{0,\dots, K\}$ we have  that
$$[u^1,u^2]_{\mathscr{S}_p^{1/2}[\frac{n}{K},\frac{n+1}{K}]} \leq M^K\left(\|u^1(0,\cdot) - u^2(0,\cdot)\|_{\mathbb{B}(\mathbb{T})} + \|b^1 - b^2\|_{C^{\alpha-1}}\right).$$
\end{lemma}
\begin{proof}
By \cref{Shalfofunum} there exists some ${N} = {N}(p,\mu, \|\sigma\|_{C^4}, \alpha, \beta)>0$ such that for all $K \in \mathbb{N}$ and $n \in \{0,\dots,K-1\}$ we have
\begin{align*}&[u^1,u^2]_{\mathscr{S}_p^{1/2}[\frac{n}{K},\frac{n+1}{K}]}\\
&\leq {N} (1+\max_{i \in \{1,2\}}[D^{u^i}]_{\mathscr{V}_{2p}^\beta})(1+\|b^2\|_{C^\alpha}) \\
&\qquad\qquad\times\Big([u^1,u^2]_{\mathscr{S}_p^{1/2}[\frac{n}{K},\frac{n+1}{K}]} + \big\|u^1\big(\frac{n}{K},\cdot\big)- u^2\big(\frac{n}{K},\cdot\big)\big\|_{\mathbb{B}(\mathbb{T}, L_p)} +\big\|b^1 - b^2 \big\|_{C^{\alpha-1}} \Big)K^{- (1+\alpha)/4}.
\end{align*}
Let $\lceil \cdot \rceil$ denote the ceiling function, and define the constants
\begin{equs}
    \tilde{N} &:= {N}(1+\max_{i \in \{1,2\}}[D^{u^i}]_{\mathscr{V}_{2p}^\beta})(1+\|b^2\|_{C^\alpha}),\\
    K_0 &:= \Big\lceil (2\tilde{N})^{\frac{4}{1+\alpha}} \Big\rceil. \label{K0def}
\end{equs}
Then for $K>K_0$ we have that
\begin{equs}\label{u1u2firstbucklingnew}
[u^1,u^2]_{\mathscr{S}_p^{1/2}[\frac{n}{K},\frac{n+1}{K}]} &\leq \big\|u^1\big(\frac{n}{K},\cdot\big)- u^2\big(\frac{n}{K},\cdot\big)\big\|_{\mathbb{B}(\mathbb{T}, L_p)} +\|b^1 - b^2 \|_{C^{\alpha-1}}.
\end{equs}
In particular, by choosing $n=0$, we have
\begin{equs}\,[u^1,u^2]_{\mathscr{S}_p^{1/2}[0,\frac{1}{K}]} \leq \|u^1_0 - u^2_0\|_{\mathbb{B}(\mathbb{T})} +\|b^1 - b^2\|_{C^{\alpha-1}}.
\label{ubracketfirststep}
\end{equs}
Let
$$a_n := [u^1,u^2]_{\mathscr{S}_p^{1/2}[\frac{n}{K},\frac{n+1}{K}]} + \big\|u^1\big(\frac{n}{K},\cdot\big) - u^2\big(\frac{n}{K},\cdot\big)\big\|_{\mathbb{B}(\mathbb{T}, L_p)} + \|b^1 - b^2\|_{C^{\alpha-1}}.$$
In the $n=0$ case we can use (\ref{ubracketfirststep}) to bound the first term to get 
\begin{equs}a_0  \leq 2\left(\|u^1_0 - u^2_0\|_{\mathbb{B}(\mathbb{T})}+\|b^1 - b^2\|_{C^{\alpha-1}}\right). \label{a0lessthanb1minusb2}
\end{equs}
For the general case $n \in \{1,\dots,K-1\}$ we first use (\ref{u1u2firstbucklingnew}) to get rid of the first term in the definition of $a_n$, and then we apply \cref{bracketlemma} as follows:
\begin{equs}
a_n &\leq 2\Big(\big\|u^1\big(\frac{n}{K},\cdot\big) - u^2\big(\frac{n}{K},\cdot\big)\big\|_{\mathbb{B}(\mathbb{T}, L_p)} + \|b^1 - b^2\|_{C^{\alpha-1}}\Big)\\
&\leq {M}\Big([u^1,u^2]_{\mathscr{S}_p^{1/2}[\frac{n-1}{K},\frac{n}{K}]} + \big\|u^1\big(\frac{n-1}{K},\cdot\big) -u^2\big(\frac{n-1}{K},\cdot\big) \big\|_{\mathbb{B}(\mathbb{T}, L_p)} + \|b^1 - b^2\|_{C^{\alpha-1}}\Big)\\
& = {M}a_{n-1}\end{equs}
for some constant ${M} ={M}(p, \|\sigma\|_{C^1}, \alpha, \beta)> 2$.
Iterating this result $n$ times and then applying (\ref{a0lessthanb1minusb2}), we get
\begin{equs}a_n \leq {M}^n a_0 &\leq M^n 2 \big(\|u^1_0 - u^2_0\|_{\mathbb{B}(\mathbb{T})}+\|b^1 - b^2\|_{C^{\alpha-1}}\big)\\
&\leq M^{K-1}M\big(\|u^1_0 - u^2_0\|_{\mathbb{B}(\mathbb{T})}+\|b^1 - b^2\|_{C^{\alpha-1}}\big),
\end{equs}
which finishes the proof.
\end{proof}
\begin{lemma}[$\mathscr{S}_p^{1/2}$-stability of regularised solutions]\label{globalbuckling} Let Assumption \ref{bucklingsectionassumptions} hold and let $p \in [2,\infty)$.
There exists a continuous map (with dependencies as indicated below) $$f = f_{p,\|\sigma\|_{C^4}, \mu, \alpha, \beta} :[0,\infty)^2 \to [0,\infty)$$ such that $f(x,y)$ is increasing in both the $x$ and $y$ variables, and
that the following inequality holds:
$$[u^1,u^2]_{\mathscr{S}_p^{1/2}} \leq f\big(\max_{i\ \in\{1,2\}}\|b^i\|_{C^\alpha}, \max_{i \in \{1,2\}}[D^{u^i}]_{\mathscr{V}_{2p}^\beta}\big)\left(\|u^1_0 - u^2_0\|_{\mathbb{B}(\mathbb{T})}+\|b^1 - b^2\|_{C^{\alpha-1}}\right).$$
\end{lemma}
\begin{proof}
Let $K \in \mathbb{Z}$ be sufficiently large so that it satisfies the assumption of \cref{Qnb1b2bound}. By (\ref{K0def}) we know that we can choose $K = N_0(1+\max_{i \in \{1,2\}}[D^{u^i}]_{\mathscr{V}_{2p}^\beta})^{\frac{4}{1+\alpha}}(1+\|b^2\|_{C^\alpha})^{\frac{4}{1+\alpha}}$
with $N_0 = N_0(p,\mu, \|\sigma\|_{C^4}, \alpha, \beta)$.
Then there exists a constant $M_1 = M_1(p,\|\sigma\|_{C^1}, \alpha, \beta)$ such that
\begin{equs}\,[u^1,u^2]_{\mathscr{S}_p^{1/2}[\frac{n}{K},\frac{n+1}{K}]} \leq M_1^K\big(\|u^1_0 - u^2_0\|_{\mathbb{B}(\mathbb{T})}+ \|b^1 - b^2\|_{C^{\alpha-1}}\big).
\label{bracketineq1}\end{equs}
Recall moreover that by \cref{triangualityofSp} there exists a constant $N_2 = N_2(p,\|\sigma\|_{C^{4}}, K, \mu, \alpha, \beta)$ such that for
$M_2 : = N_2(1+ \max_{i \in \{1,2\}}[D^{u^i}]_{\mathscr{V}_{2p}^\beta})$
we have
\begin{equs}\,[u^1,u^2]_{\mathscr{S}_p^{1/2}} \leq (K-1)M_2^{K-1}\|u^1_0 - u^2_0\|_{\mathbb{B}(\mathbb{T})}+  2\sum_{i=0}^{K-1} M_2^{i}[u^1,u^2]_{\mathscr{S}_p^{1/2}[\frac{K-i-1}{K}, \frac{K-i}{K}]}. \label{bracketineq2}
\end{equs}
By (\ref{bracketineq1}), we get that the second term on the right hand side of (\ref{bracketineq2}) is bounded by
\begin{equs}
& 2\sum_{i=0}^{K-1}M_2^i M_1^{K}\big(\|u^1_0 - u^2_0\|_{\mathbb{B}(\mathbb{T})}+ \|b^1 - b^2\|_{C^{\alpha-1}}\big)\\
&\qquad\leq 2(K-1)(M_1M_2)^K\big(\|u^1_0- u^2_0\|_{\mathbb{B}(\mathbb{T})}+ \|b^1 - b^2\|_{C^{\alpha-1}}\big).
\end{equs}
Therefore
\begin{equs}
\,[u^1,u^2]_{\mathscr{S}_p^{1/2}}\leq (K-1)\big(M_2^{K-1} + 2(M_1M_2)^K\big)\big(\|u^1_0- u^2_0\|_{\mathbb{B}(\mathbb{T})}+ \|b^1 - b^2\|_{C^{\alpha-1}}\big),
\end{equs}
and the desired result follows by the definitions of $K,M_1,M_2$.
\end{proof}

 \section{The $\mathscr{V}_p$-bracket of the drift and an a priori estimate}\label{driftsection}
 The aim of this section is to provide a priori bounds on a  regularised solution of (\ref{SHE}) under Assumption \ref{assumptions}.
\begin{lemma}\label{Dbound}
Let Assumption \ref{assumptions} hold, let  $\beta \in (\frac{1}{4} - \frac{\alpha}{4}, 1 + \frac{\alpha}{4}]$ and assume that $u$ is a regularised solution of (\ref{SHE}) in the class $\mathscr{U}_2^\beta$.  
Then $u$ is also of class $\mathscr{U}^\beta$. Moreover for all $p \in [2,\infty)$
there exists a constant $N = N(p,\|\sigma\|_{C^4}, \mu, \alpha, \beta)>0$ such that 
$$[D^u]_{\mathscr{V}_p^\beta} \leq 
N\exp\Big(N \|b\|_{C^\alpha}^{\frac{4}{\alpha+3}}\Big).$$
\end{lemma}
\begin{proof}
Let $(b^n)_{n \in \mathbb{N}} \subset C^\infty$ be a sequence of smooth functions such that $b^n \to b$ in $C^\alpha$. Then by the definition of $D^u$ (see (\ref{driftdef})), by the conditional Fatou's lemma and the usual Fatou's lemma, for $p\geq 2$ and for $(s,t) \in [0,1]_{\leq}^2$, $x \in \mathbb{T}$ we have that
\begin{equs}
\|D_t^u(x) - P_{t-s}D_s^u(x)\|_{L_{p,\infty}^{\mathscr{F}_s}} \lesssim \liminf_{n \to \infty} \sup_{x \in \mathbb{T}}\Big\|\int_s^t \int_{\mathbb{T}}p_{t-r}(x,y) b^n(u(r,y))dydr\Big\|_{L_{p,\infty}^{\mathscr{F}_s}}.
\end{equs}
Therefore by applying \cref{otherregularisationlemma} we know that
\begin{equs}
\|D_t^u(x) - P_{t-s}D_s^u(x)\|_{L_{p,\infty}^{ \mathscr{F}_s}} &\lesssim \|b\|_{C^\alpha}\Big((t-s)^{1 + \alpha/4} + [D^u]_{\mathscr{V}_2^\beta[s,t]}(t-s)^{\beta + \frac{\alpha+3}{4}}\Big)\\
&\lesssim  \|b\|_{C^\alpha}\Big((t-s)^{\beta} + [D^u]_{\mathscr{V}_2^\beta[s,t]}(t-s)^{\beta + \frac{\alpha+3}{4}}\Big),
\end{equs}
where we used the assumption that $\beta \leq 1+\alpha/4$. Hence there exists $\tilde{N} = \tilde{N}(p,\|\sigma\|_{C^4}, \mu, \alpha, \beta)$ such that for all $(s,t) \in [0,1]_{\leq}^2$ we have
\begin{equation}\label{prelocatisationforD}[D^u]_{\mathscr{V}_p^\beta[s,t]} \leq \tilde{N}\|b\|_{C^\alpha}  + \tilde{N}\|b\|_{C^\alpha}[D^u]_{\mathscr{V}_2^\beta[s,t]}(t-s)^{(\alpha+3)/4}.
\end{equation}
Since we assumed that $u \in \mathscr{U}_2^\beta$, we have $[D^u]_{\mathscr{V}_2^\beta} <\infty$, and thus by the inequality (\ref{prelocatisationforD}) we have $[D^u]_{\mathscr{V}_p^\beta}<\infty$, and thus $u \in \mathscr{U}_p^\beta$. Since $p \geq 2$ was arbitrary, it follows that $u \in \mathscr{U}^\beta.$ 

Note that on the right hand side of (\ref{prelocatisationforD}), the $[D^u]_{\mathscr{V}_2^\beta[s,t]}$ may be replaced with $[D^u]_{\mathscr{V}_p^\beta[s,t]}$. Hence choosing sufficiently large $K \in \mathbb{N}$, it follows that
\begin{equation}\label{postlocatisationforD}\max_{i \in \{0,\dots,K-1\}}[D^u]_{\mathscr{V}_p^\beta[\frac{i}{K},\frac{i+1}{K}]} \lesssim\|b\|_{C^\alpha}.
\end{equation}
To this end we may pick
$K := \Big\lceil(2\tilde{N}\|b\|_{C^\alpha})^{\frac{4}{\alpha+3}} \Big\rceil$. 
 Moreover using \cref{Vbracketistriangular} and the inequality (\ref{postlocatisationforD}), we obtain that
\begin{equs}
 \,[D^u]_{\mathscr{V}_p^\beta[0,1]} &\leq  2^K\sum_{i=0}^{K-1} [D^u]_{\mathscr{V}_p^\beta[\frac{i}{K},\frac{i+1}{K}]}\\
 &\lesssim 2^K\sum_{i=0}^{K-1}  \|b\|_{C^\alpha}
 \lesssim K2^K\|b\|_{C^\alpha},
\end{equs}
which finishes the proof .
\end{proof}

\begin{lemma}[The regularity of $D^u$]\label{regularityofD} Let Assumption \ref{assumptions} hold, and let $p \in [2,\infty)$, $\beta \in (\frac{1}{2} - \frac{\alpha}{4}, 1+\frac{\alpha}{4}]$.
There exists a constant $N = N(p,\|\sigma\|_{C^4}, \mu, \alpha, \beta)>0$ such that if $u$ is a regularised solution of class $\mathscr{U}^\beta$, then
$$\|D^u\|_{C^{\frac{1}{4},\frac{1}{2}}([0,1]\times \mathbb{T}, L_p)} \leq N(1+\|b\|_{C^\alpha})(1+[D^u]_{\mathscr{V}_p^\beta}).$$
\end{lemma}
\begin{proof}
Noting that $\|\cdot\|_{L_p} = \|\|\cdot\|_{L_p|\mathscr{F}_s}\|_{L_p} \leq \| \| \cdot\|_{L_p|\mathscr{F}_s}\|_{L_\infty}$ and that from the definition of $D^u$ (see (\ref{driftdef})) 
 we have $D_0^u =0$, we conclude for all $(t,x) \in [0,1]\times\mathbb{T}$ that
\begin{equs}
\|D_t^u(x)\|_{L_p} &\leq  \|D_t^u(x) - P_{t-0}D_0^u(x)\|_{L_{p,\infty}^{\mathscr{F}_s}}\leq  [D^u]_{\mathscr{V}_p^\beta}. \label{DCzeronorm}
\end{equs}
Let $(b^n)_{n \in \mathbb{N}} \subset C^\infty$ be a sequence of smooth functions such that $b^n \to b$ in $C^\alpha$. By (\ref{driftdef}), Fatou's lemma and \cref{spatialregularisation}, we can see that for all $x, \bar{x} \in \mathbb{T}$ and $t \in [0,1]$ we have
\begin{equs}
\|D_t^u(x) - D_t^u(\bar{x})\|_{L_p} & \leq \liminf_{n \to \infty}\Big\| \int_0^t \int_{\mathbb{T}} (p_{t-r}(x,y) - p_{t-r}(\bar{x}, y))b^n(u(r,y))dydr\Big\|_{L_p}\\
&\lesssim \|b\|_{C^\alpha}(1+[D^u]_{\mathscr{V}_p^\beta})|x-\bar{x}|^{1/2}.\label{DChalfspatial}
\end{equs}
By (\ref{DCzeronorm}) and (\ref{DChalfspatial}) we can see that
\begin{equs}\sup_{t \in [0,1]}\|D^u_t\|_{C^{1/2}(\mathbb{T}) }\lesssim (1+\|b\|_{C^\alpha})(1+[D^u]_{\mathscr{V}_p^\beta}). \label{DintheChalfnorm}
\end{equs}
Finally, note that since by assumption we have $\beta >\frac{1}{2}-\frac{\alpha}{4} > \frac{1}{4}$, and thus
\begin{equs}
\,[D^u]_{\mathscr{V}_p^{1/4}} \leq [D^u]_{\mathscr{V}_p^{\beta}}. \label{Vphalfminuseps}
\end{equs}
By (\ref{DintheChalfnorm}) and (\ref{Vphalfminuseps}), the desired bound holds for the $C^{0,\frac{1}{2}}([0,1]\times \mathbb{T}, L_p)$-norm and for the $\mathscr{V}_p^{1/4}$-bracket. Hence by \cref{embeddinglemma} the proof is finished.
\end{proof}

\begin{lemma}[An a priori estimate]\label{apriori}
Let Assumption \ref{assumptions} hold, and let $p\in [2,\infty)$, $\varepsilon  \in  (0, \frac{1}{2})$, $\beta \in (\frac{1}{2} - \frac{\alpha}{4}, 1+\frac{\alpha}{4}]$. There exists  a constant $N = N(p,\|\sigma\|_{C^4}, \mu, \alpha, \beta, \varepsilon)>0$ such that if $u$ is a regularised solution of class $\mathscr{U}^\beta$, then
\begin{equs}\,\|u -P_\cdot u_0(\cdot)\|_{C^{1/4 - \varepsilon/2, 1/2-\varepsilon}([0,1]\times \mathbb{T}, L_p)} \leq N(1+\|b\|_{C^\alpha})(1+[D^u]_{\mathscr{V}_p^\beta}).
\end{equs}
\end{lemma}
\begin{proof}
For $(t,x) \in [0,1]\times \mathbb{T}$ denote
$$V_t(x) :=\int_0^t \int_{\mathbb{T}} p_{t-r}(x,y) \sigma(u(r,y)) \xi(dy,dr).$$ By the triangle inequality
\begin{equs}
\|u - Pu_0\|_{C^{1/4 - \varepsilon/2, 1/2-\varepsilon}([0,1]\times \mathbb{T},L_p)} \leq \|D^u\|_{C^{1/4 - \varepsilon/2, 1/2-\varepsilon}([0,1]\times \mathbb{T},L_p)} + \|V\|_{C^{1/4 - \varepsilon/2, 1/2-\varepsilon}([0,1]\times \mathbb{T},L_p)}.
\end{equs}
But by \cref{regularityofD}, we know that
$\|D^u\|_{C^{1/4, 1/2}([0,1]\times \mathbb{T},L_p)} \lesssim (1+\|b\|_{C^\alpha})(1+[D^u]_{\mathscr{V}_p^\beta})$
and it can be seen from the BDG inequality and by the heat kernel estimates (\ref{pspacediffsq}) and (\ref{ptimediffsq}) that $\|V\|_{C^{1/4 - \varepsilon/2, 1/2-\varepsilon}([0,1]\times \mathbb{T}, L_p)} \lesssim 1$,
and thus the proof is finished.
\end{proof}
 
\section{Proof of the main result}\label{mainproofsection}
 \begin{theorem}[Uniqueness]
 Let Assumption \ref{assumptions} hold, let $\beta \in (\frac{1}{2} - \frac{\alpha}{4}, 1+\frac {\alpha}{4}]$ and suppose that $u^1,u^2$ are regularised solutions of (\ref{SHE}) in the class $\mathscr{U}_2^{\beta}$. Then
 $u^1(t,x) = u^2(t,x)$ almost surely for all $(t,x) \in [0,1]\times\mathbb{T}$.
 \end{theorem}
 \begin{proof}
Since $u^1,u^2 \in \mathscr{U}_2^\beta$, it also follows by \cref{Dbound} that $u^1,u^2 \in \mathscr{U}^\beta$.
Thus Assumption \ref{bucklingsectionassumptions} satisfied. Therefore by \cref{globalbuckling}  we have for $p \in [2,\infty)$ that
$$[u^1,u^2]_{\mathscr{S}_{p}^{1/2}} \leq 0.$$
So since $u^1(t,\cdot) - u^2(t,\cdot) = u^1(t,\cdot) - u^2(t,\cdot) - \phi^{u^1(0,\cdot),s}(t,\cdot) + \phi^{u^2(0,\cdot),s}(t,\cdot)$, it follows that
$$\sup_{(t,x) \in [0,1]\times \mathbb{T}}\|u^2(t,x) - u^2(t,x)\|_{L_{p}} =0,$$
and the desired result follows.
\end{proof}

Let Assumption \ref{assumptions} hold. The rest of the section is concerned with proving the existence of regularised solutions in the class $\mathscr{U}^{1+\frac{\alpha}{4}}$.  Let $(b^n)_{n \in \mathbb{N}} \subset C^\infty$ such that
$b^n \to b$ in $C^\alpha$. Suppose that for all $n \in \mathbb{N}$, $u^n$ is the classical mild solution of the SPDE
\begin{equs}
(\partial_t - \Delta)u^n &= b^n(u^n) + \sigma(u^n)\xi, \qquad
u^n(0,\cdot)& = u(0,\cdot). \label{approximateSPDE}
\end{equs}
We call $(u^n)_{n \in \mathbb{N}}$ the sequence of \emph{approximate solutions},  and for
$(t,x) \in [0,1]\times \mathbb{T}$ we define the corresponding \emph{approximate drift term} and \emph{approximate noise term}  respectively by
\begin{equs}D_t^{u^n}(x) &:= \int_0^t \int_{\mathbb{T}}p_{t-r}(x,y)b^n(u^n(r,y))dydr,\\
V^{u^n}_t(x) &:= \int_0^t \int_{\mathbb{T}}p_{t-r}(x,y)\sigma(u^n(r,y))\xi(dy,dr).
\end{equs}
By \cref{Dbound} we have for all $p \geq 1$, that  
\begin{equs}
\sup_{n \in \mathbb{N}}[D^{u^n}]_{\mathscr{V}_p^{1+\alpha/4}} <\infty. \label{uniformlyUp}
\end{equs}
\begin{lemma}[Convergence of the approximate drift and noise terms]\label{convergenceofdrift}
Let Assumption \ref{assumptions} hold, and let $p  \in [1,\infty)$ and $\varepsilon \in (0,\frac{1}{2})$. Then the sequences $(D^{u^n})_{n \in \mathbb{N}}$, $(V^{u^n})_{n \in \mathbb{N}}$ are convergent in $C^{\frac{1}{4}-\frac{\varepsilon}{2},\frac{1}{2} - \varepsilon}([0,1]\times\mathbb{T}, L_p)$.
\end{lemma}
\begin{proof}Assume without loss of generality that $p > 2$.
By \cref{biuidriftcomparison} (with $\beta = 1+\frac{\alpha}{4}$) and by \cref{globalbuckling} we have
\begin{equs}
&\sup_{(t,x) \in [0,1]\times \mathbb{T}}\|D_t^{u^n}(x) - D_t^{u^m}(x)\|_{L_p}\\
&=\sup_{(t,x) \in [0,1]\times \mathbb{T}}\Big\|\int_0^t \int_{\mathbb{T}}p_{t-r}(x,y)\big(b^n(u^n(r,y)) - b^m(u^m(r,y))\big)dydr\Big\|_{L_p}\\
&\lesssim  \|b^n - b^m\|_{C^{\alpha-1}} + [u^n,u^m]_{\mathscr{S}^{1/2}_p[0,1]}\lesssim \|b^n- b^m\|_{C^\alpha} \longrightarrow 0 \label{driftconvergesinsup}
\end{equs}
as $n \to \infty$. Moreover by \cref{regularityofD} (with $\beta = 1+\frac{\alpha}{4}$) and by (\ref{uniformlyUp}), we have that
\begin{equs}\sup_{n \in \mathbb{N}}\|D^{u^n}\|_{C^{1/4,1/2}([0,1]\times \mathbb{T}, L_p)} <\infty. \label{approxdriftisregular}
\end{equs}
By (\ref{driftconvergesinsup}), (\ref{approxdriftisregular}), and by a standard interpolation argument, we can see that $(D^{u^n})_{n \in \mathbb{N}}$ is Cauchy in $C^{\frac{1}{4}-\frac{\varepsilon}{2},\frac{1}{2} - \varepsilon}([0,1]\times\mathbb{T}, L_p)$. 

We proceed with showing that the same is true for the sequence $(V^n)_{n \in \mathbb{N}}$. To this end note that by the BDG inequality, by the definition of the $\mathscr{S}_p^{1/2}$-bracket, and by \cref{globalbuckling} we have
\begin{equs}
&\sup_{(t,x) \in [0,1]\times \mathbb{T}}\|V_t^{u^n}(x) - V_t^{u^m}(x)\|_{L_p}\\
&= \sup_{(t,x) \in [0,1]\times \mathbb{T}}\Big\|\int_0^t \int_{\mathbb{T}}p_{t-r}(x,y)\big(\sigma(u^n(r,y)) - \sigma(u^m(r,y))\big)\xi(dy,dr)\Big\|_{L_p}\\
&\lesssim t^{1/4}\|u^n - u^m\|_{\mathbb{B}([0,1]\times \mathbb{T}, L_p)} \leq [u^n,u^m]_{\mathscr{S}_p^{1/2}} \lesssim \|b^n - b^m\|_{C^{\alpha-1}} \longrightarrow 0 \label{noiseconverges}
\end{equs}
as $n,m \to \infty$. 
 Let $\gamma \in (0,\varepsilon)$. Using the BDG inequality and the heat kernel estimates (\ref{pspacediffsq}), (\ref{ptimediffsq}), we can see that for all $n \in \mathbb{N}$, $s,t \in [0,1]$, $x,\bar{x} \in \mathbb{T}$ the following estimates hold:
\begin{equs}
\|V_t^{u^n}(x) - V_t^{u^n}(\bar{x})\|_{L_p}^2 &\lesssim \int_0^t \int_{\mathbb{T}}(p_{t-r}(x,y) - p_{t-r}(\bar{x},y))^2 dydr \lesssim |x-\bar{x}|^{1-2\gamma},\\
\|V_t^{u^n}(x) - V_s^{u^n}(x)\|_{L_p}^2 &\lesssim \int_0^s (p_{t-r}(x,y) - p_{s-r}(x,y))^2 dydr + \int_s^t p_{t-r}^2(x,y)dydr\\
&\lesssim |t-s|^{1/2-\gamma}.
\end{equs}
Therefore we conclude that
\begin{equs}
\sup_{n \in \mathbb{N}}
\|V^{u^n}\|_{C^{\frac{1}{4}-\frac{\gamma}{2}, \frac{1}{2}-\gamma}([0,1]\times \mathbb{T}, L_p)} <\infty. \label{approxnoiseisregular}
\end{equs}
By (\ref{noiseconverges}), (\ref{approxnoiseisregular}), and by a standard interpolation argument, we can see that $(V^{n})_{n \in \mathbb{N}}$ is also Cauchy in $C^{\frac{1}{4}-\frac{\varepsilon}{2},\frac{1}{2} - \varepsilon}([0,1]\times\mathbb{T}, L_p)$, and thus the proof is finished.
\end{proof}
Consistently with the above lemmas, we will thus denote
\begin{equs}
{D}^{\tilde{u}} :=\lim_{n \to \infty}D^{u^n} \tand
V^{\tilde{u}} := \lim_{n \to \infty}V^{u^n},
\end{equs}
where the limits are taken pointwise in $(t,x)\in [0,1]\times \mathbb{T}$, in probability. Moreover, it follows that for all $\eps \in (0, 1/2)$ and $p \in [1, \infty) $ we have  ${D}^{\tilde{u}},V^{\tilde{u}}  \in  C^{\frac{1}{4}-\frac{\varepsilon}{2} , \frac{1}{2}-\varepsilon}([0,1]\times \mathbb{T}, L_p)$
and 
\begin{equs}                  \label{utildeDtildedef}
    \lim_{n \to \infty} \Big( || {D}^{\tilde{u}} -D^{u^n}\|_{C^{\frac{1}{4}-\frac{\eps}{2}, \frac{1}{2}-\eps}([0,1]\times \mathbb{T}, L_p)}+|| V^{\tilde{u}}-V^{u^n}\|_{C^{\frac{1}{4}-\frac{\eps}{2}, \frac{1}{2}-\eps}([0,1]\times \mathbb{T}, L_p)} \Big)=0 .
\end{equs}
Moreover for $(t,x) \in [0,1]\times \mathbb{T}$, we define
\begin{equs}
\tilde{u}(t,x) := P_t u_0(x) + D^{\tilde{u}}_t(x) + V^{\tilde{u}}_t(x). \label{utildedef}
\end{equs}
\begin{lemma}[$V^{\tilde{u}}$ is the noise term of $\tilde{u}$]\label{Vtildeuisournoiseterm}
Let Assumption \ref{assumptions} hold.
For all $(t,x) \in [0,1]\times \mathbb{T}$, we have $$V^{\tilde{u}}_t(x) = \int_0^t \int_{\mathbb{T}}p_{t-r}(x,y)\sigma(\tilde{u}(r,y))\xi(dy,dr).$$
\end{lemma}
\begin{proof}
By the definitions of $D^{\tilde{u}}$ and $V^{\tilde{u}}$ (see (\ref{utildeDtildedef})), by Fatou's lemma, and by the definition of $\tilde{u}$ (see (\ref{utildedef})) we have for $p \geq 2$ that
\begin{equs}
&\|V_t^{\tilde{u}}(x) - \int_0^t \int_{\mathbb{T}}p_{t-r}(x,y)\sigma(\tilde{u}(r,y))\xi(dy,dr)\|_{L_p}^2\\
&\leq \liminf_{n \to \infty}\Big\|\int_0^t \int_{\mathbb{T}}p_{t-r}(x,y)\sigma(u^n(r,y)) - \sigma(\tilde{u}(r,y))\xi(dy,dr)\Big\|_{L_p}\\
&\lesssim t^{1/4}\liminf_{n \to \infty}\|u^n - \tilde{u}\|_{\mathbb{B}([0,1]\times \mathbb{T}, L_p)}\\
&\lesssim \lim_{n \to \infty}\|D^{u^n} - D^{\tilde{u}}\|_{\mathbb{B}([0,1]\times \mathbb{T}, L_p)} + \lim_{n \to \infty}\|V^{u^n} - V^{\tilde{u}}\|_{\mathbb{B}([0,1]\times \mathbb{T}, L_p)}= 0,
\end{equs}
and thus the proof is finished.
\end{proof}
We  proceed with verifying that the definition of $D^{\tilde{u}}$ is not an abuse of notation, i.e. that $D^{\tilde{u}}$ is indeed the drift of $\tilde{u}$ as prescribed in (\ref{driftdef}). To this end, we will first need to prove the following lemma.

\begin{lemma} \label{convergenceofDinCquarterhalfLp} Let Assumption \ref{assumptions} hold, and for $n \in \mathbb{N}$ define random fields
$f^n: \Omega \times[0,1]\times \mathbb{T} \to \mathbb{R}$ by
\begin{equs}f^n(t,x) :=D_t^{\tilde{u}}(x) - \int_0^t \int_{\mathbb{T}} p_{t-r}(x,y) b^n(\tilde{u}(r,y))dydr. \label{fndef}
\end{equs}
Then for any $p \in [1,\infty)$ we have that $\|f^n\|_{C^{\frac{1}{4},\frac{1}{2}}([0,1]\times \mathbb{T}, L_p)} \longrightarrow 0$ as $n \to \infty$.
\end{lemma}

\begin{proof}
To bound the sup norm, we note that by  Fatou's lemma, \cref{otherregularisationlemma} (with $g = b^m - b^n$) and \cref{Dbound}, we have that
\begin{equs}
\|f^n\|_{\mathbb{B}([0,1]\times \mathbb{T}, L_p)}&=  \sup_{(t,x) \in [0,T]\times \mathbb{T}}\Big\|{D}_t^{\tilde{u}}(x) - \int_0^t \int_{\mathbb{T}} p_{t-r}(x,y) b^n(\tilde{u}(r,y))dydr\Big\|_{L_p}\\
&\leq \liminf_{m \to \infty}\sup_{(t,x) \in [0,T]\times \mathbb{T}}\Big\| \int_0^t \int_{\mathbb{T}}p_{t-r}\big(b^m(u^m(r,y)) - b^n(u^m(r,y))\big)dydr \Big\|_{L_p}\\
&\lesssim \liminf_{m \to \infty }\|b^m - b^n\|_{C^\alpha}(1+[D^{u^m}]_{\mathscr{V}_p^{1+\alpha/4}})(t-s)^{1 + \alpha/4}\lesssim \|b - b^n\|_{C^\alpha},
\end{equs}
and thus
\begin{equs}\lim_{n \to \infty}\|f^n\|_{\mathbb{B}([0,1]\times\mathbb{T}, L_p)}  = 0. \label{fnisCzerospacetime}
\end{equs}
Next, we bound the spatial seminorm. Let $x, \bar{x} \in \mathbb{T}$. In the calculation below we will use the definitions of $D^{\tilde{u}}$ $\tilde{u}$, $f^n$ (see (\ref{utildeDtildedef}), (\ref{utildedef}), and (\ref{fndef})) and the continuity of the approximate drifts, Fatou's lemma, \cref{spatialregularisation} (with $g(x) = b^m(x) - b^n(x)$) and (\ref{uniformlyUp}),
\begin{equs}
&\sup_{t \in [0,1]}\|f^n(t,x) - f^n(t,\bar{x})\|_{L_p}\\
&=\sup_{t \in [0,1]}\Big\| D_t^{\tilde{u}}(x) - D_t^{\tilde{u}}(\bar{x}) - \int_0^t \int_{\mathbb{T}}(p_{t-r}(x,y) - p_{t-r}(\bar{x},y))b^n(\tilde{u}(r,y))dydr\Big\|_{L_p}\\
& = \sup_{t \in [0,1]}\liminf_{m \to \infty}\Big\|  \int_0^t \int_{\mathbb{T}}\big(p_{t-r}(x,y) - p_{t-r}(\bar{x},y))\big(b^m(u^m(r,y)) -b^n(u^m(r,y))\big)dydr\Big\|_{L_p}\\
& \lesssim \liminf_{m \to \infty}\|b^m - b^n\|_{C^\alpha}(1+[D^{u^m}]_{\mathscr{V}_{2p}^{1+\alpha/4}})|x-\bar{x}|^{1/2} \lesssim \|b - b^n\|_{C^\alpha}|x-\bar{x}|^{1/2}.
\end{equs}
Therefore
\begin{equs}\lim_{n \to \infty}\sup_{t \in [0,1]}[f^n(t,\cdot)]_{C^{1/2}(\mathbb{T}, L_p)} = 0. \label{fnisChalfinspace}
\end{equs}
Finally, note that for $s,t \in [0,1]$ we have by Fatou's lemma, \cref{otherregularisationlemma} (with $g = b^m - b^n$), and \cref{Dbound}, that
\begin{equs}
&\sup_{x \in \mathbb{T}}\| f^n(t,\cdot) - P_{t-s}f^n(s,x)\|_{L_{p,\infty}^{\mathscr{F}_s}}\\
&= \sup_{x \in \mathbb{T}}\Big\| D_t^{\tilde{u}}(x)  - \int_0^t \int_{\mathbb{T}}p_{t-r}(x,y)b^n(\tilde{u}(r,y))dydr\\
&\qquad\qquad\qquad\qquad- P_{t-s}\Big(D_s^{\tilde{u}}(\cdot) -\int_0^s \int_{\mathbb{T}}p_{t-r}(x,y)b^n(\tilde{u})dydr \Big)\Big\|_{L_{p,\infty}^{\mathscr{F}_s}}\\
& \leq \sup_{x \in \mathbb{T}}\liminf_{m \to \infty}\Big\| \int_s^t \int_{\mathbb{T}}p_{t-r}(x,y)(b^m(u^m(r,y)) - b^n(u^m(r,y)))dydr \Big\|_{L_{p,\infty}^{\mathscr{F}_s}}\\
&\lesssim \liminf_{m \to \infty }\|b^m - b^n\|_{C^\alpha}(1+[D^{u^m}]_{\mathscr{V}_{2p}^{1+\alpha/4}})(t-s)^{1 + \alpha/4}\\
&\lesssim \|b- b^n\|_{C^\alpha}(t-s)^{1 + \alpha/4}.
\end{equs}
 It follows that
\begin{equs}
\lim_{n \to \infty}[f^n]_{\mathscr{V}_p^{1+\alpha/4}} = 0. \label{convergenceoffinbracket}
\end{equs}
By (\ref{fnisCzerospacetime}), (\ref{fnisChalfinspace}), (\ref{convergenceoffinbracket}), and by \cref{embeddinglemma} the proof is finished.
\end{proof}

\begin{corollary}[$D^{\tilde{u}}$ is the drift of $\tilde{u}$]\label{tildeDisthedrift} Let Assumption \ref{assumptions} hold.
Then the pair $(\tilde{u}, D^{\tilde{u}})$ satisfies the condition (\ref{driftdef}) from Definition \ref{defofsolution}, that is for any sequence $(b^n)_{n \in \mathbb{N}}\subset C^\infty$ such that $b^n \to b$ in $C^\alpha$, we have
$$\sup_{(t,x) \in [0,T]\times \mathbb{T}}\Big|D_t^{\tilde{u}}(x) - \int_0^t \int_{\mathbb{T}} p_{t-r}(x,y) b^n(\tilde{u}(r,y))dydr\Big|  \longrightarrow 0$$
in probability as $n \to \infty$. 
\end{corollary}

\begin{theorem}[Existence]
Let Assumption \ref{assumptions} hold. Then the process $\tilde{u}$ is a regularised solution of (\ref{SHE}) in the class $\mathscr{U}^{1 + \alpha/4}$.
\end{theorem}
\begin{proof}
Since for all $n \in \mathbb{N}$, the random field $u^n$ (which is a classically defined mild solution) is $\mathscr{P}\otimes \mathscr{B}(\mathbb{T})$-measurable, so is the limit $\tilde{u}$. By the definition of $\tilde{u}$ and by \cref{convergenceofdrift} we have that
$$\tilde{u} - P_\cdot u_0 \in C^{1/4 - \varepsilon, 1/4 - \varepsilon/2}([0,1]\times \mathbb{T}, L_p)$$
for $p \geq 1$ and for any $\varepsilon>0$. Therefore by Kolmogorov's continuity theorem, the random field $\tilde{u}(t,x) -P_tu(0,\cdot)(x)$ is continuous in $(t,x)$. So noting that $P_tu(0,x)$ is also continuous in $(t,x)$, it follows that $\tilde{u}(t,x)$ is continuous in $(t,x)$.
Note moreover that by \cref{tildeDisthedrift}, the pair $(\tilde{u}, {D}^{\tilde{u}})$ satisfies (\ref{driftdef}). 
Finally, we observe that by the definition of $\tilde{u}$ and by \cref{Vtildeuisournoiseterm} the integral equation (\ref{integraleqn}) is satisfied.
 Therefore it is clear that $\tilde{u}$ is a regularised solution of (\ref{SHE}). Moreover for all $p \geq 1$ we have
$$[D^{\tilde{u}}]_{\mathscr{V}_p^{1+\alpha/4}} \leq  \liminf_{n \to \infty}[D^{u^n}]_{\mathscr{V}_p^{1+\alpha/4}} \leq \sup_{n \in \mathbb{N}}[D^{u^n}]_{\mathscr{V}_p^{1+\alpha/4}} <\infty,$$
where the last inequality holds by (\ref{uniformlyUp}). Therefore $\tilde{u} \in \mathscr{U}^{1 + \alpha/4}$, and the proof is finished.
\end{proof}

\section{Appendix}

\begin{lemma}\label{trivialpowerslemma}
Let $\varepsilon \in (0,1/2)$, $\gamma \in (0 , \varepsilon)$, and define
$$\delta := \frac{2(\varepsilon- \gamma)}{1 - 2\gamma}.$$
Then $\delta \in (0,1)$, and for all $(t,x),(s,y) \in [0,1]\times \mathbb{T},$ we have
\begin{equs}
\Big(|t-s|^{1/4 - \gamma/2} + |x-y|^{1/2 - \gamma}\Big)^{1 - \delta} \leq |t-s|^{1/4 - \varepsilon/2} + |x-y|^{1/2 - \varepsilon}. \label{powersineqappendix}
\end{equs}
\end{lemma}
\begin{proof}
We begin by noting that since $\varepsilon \in (0,1/2)$, we have
$
\delta < \frac{2(\varepsilon- \gamma)}{2\varepsilon - 2\gamma} =1$. The positivity of $\delta$ also immediately follows from the fact that $0<\gamma <\varepsilon<1/2$.
So we have $1-\delta \in (0,1)$, and thus the map $x \mapsto |x|^{1-\delta}$ is subadditive. Hence the left hand side of (\ref{powersineqappendix}) is bounded by
$$|t-s|^{(1/4 - \gamma/2)(1-\delta)} + |x-y|^{(1/2 -\gamma)(1-\delta)}.$$
Now we just need to check that the powers in this expression match the powers on the right hand side of (\ref{powersineqappendix}). This is indeed true, since
\begin{equs}
\big(\frac{1}{4} - \frac{\gamma}{2}\big)(1-\delta) = \frac{1}{4}(1-2\gamma)\bigg(1 - \frac{2(\varepsilon-\gamma)}{1-2\gamma}\bigg) =\frac{1}{4}(1- 2\gamma - 2(\varepsilon-\gamma)) =\frac{1}{4} -\frac{\varepsilon}{2},
\end{equs}
and
\begin{equs}
\big(\frac{1}{2}- \gamma\big)(1-\delta) =\frac{1}{2}(1-2\gamma)\Big(1-\frac{2(\varepsilon-\gamma)}{1-2\gamma}\Big) = \frac{1}{2}(1-2\gamma  - 2(\varepsilon-\gamma)) = \frac{1}{2} - \varepsilon,
\end{equs}
and thus the proof is finished.
\end{proof}

\begin{proposition}
\label{borrowedheatkernelestimates}
 For any $\gamma \in[0,1]$ there exists a constant $N  = N(\gamma)>0$ such that for all $t \in [0,1]$ and $x,\bar{x},y \in \mathbb{T}$ we have
\begin{equs}|p_t(x,y) -p_t(\bar{x},y)| \leq N|x-\bar{x}|^\gamma t^{-\gamma/2}\big(p_{2t}(x,y) + p_{2t}(\bar{x},y)\big). \label{khoaskernelidentity}
\end{equs}
Moreover for any $\gamma,\beta \in [0,1]$ with $\alpha \leq \beta$ there exists a constant $N = N(\gamma,\beta)>0$ such that for all $(s,t) \in [0,1]_{\leq}^2$ and $x,\bar{x} \in \mathbb{T}$ and for all $f \in C^\alpha(\mathbb{T})$ we have
\begin{equs}|P_t f(x) - P_sf(\bar{x})| \leq N\|f\|_{C^\gamma}(|x-\bar{x}|^\beta + |t-s|^{\beta/2})s^{(\gamma-\beta)/2}.
\label{konstantinoskernelidentity}
\end{equs}
\end{proposition}
The first inequality of the lemma above is taken from the proof of  \cite[Lemma C2]{athreya2024well}, while the second inequality can be found in \cite{butkovsky2023optimal}.

\begin{proposition}
\label{heatkerneldiffsquareint}
For any $\varepsilon \in (0,1]$ there exists a constant $N = N(\varepsilon)>0$ such that for all $(s,t) \in [0,1]_{\leq}^2$, the following inequalities hold:
\begin{equs}
\int_{\mathbb{T}}|p_{t}(x,y) - p_t(\bar{x},y)|dy &\leq N|x-\bar{x}|^\varepsilon t^{-\varepsilon/2} \label{pspacediffL1},\\
\int_0^t \int_{\mathbb{T}}|p_{t-r}(x,y)- p_{t-r}(\bar{x},y)|^2 dy dr &\leq N |x-\bar{x}|^{1-\varepsilon}t^{\varepsilon/2} \label{pspacediffsq},\\
\int_0^s \int_{\mathbb{T}} |p_{t-r}(x,y) - p_{s-r}(x,y)|^2 dydr &\leq {N} |t-s|^{1/2 - \varepsilon/2}. \label{ptimediffsq}
\end{equs}
\end{proposition}
The inequality (\ref{pspacediffL1}) can be found in Lemma C2 of (\cite{athreya2024well}). The inequality (\ref{pspacediffsq}) can be proven by using (\ref{khoaskernelidentity}) and (\ref{ptimediffsq}) can be shown using (\ref{konstantinoskernelidentity}). 
\begin{lemma}
\label{gronwallintegralheatkernel}
For every $\gamma \in (1,2]$ there exists a constant $N(\gamma)>0$ such that for all $t \in [0,1]$,
$$\int_0^t \int_{\mathbb{T}} p_{t-r}^\gamma(x,y) e^{-\lambda(t-r)} dydr \leq \frac{N}{\sqrt{\lambda}}.$$
\end{lemma}
\begin{proof}
The left-hand-side is controlled by
\begin{align*}
\int_0^t (t-r)^{-1/2(\gamma-1)} e^{-\lambda(t-r)}dydr \leq
\int_0^t \frac{1}{\sqrt{t-r}}e^{-\lambda(t-r)}dr = \frac{2}{\sqrt{\lambda}}\int_0^{\sqrt{\lambda t}} e^{-\theta^2} d\theta\\
=\frac{ \sqrt{\pi}}{\sqrt{\lambda}} \text{erf}(\sqrt{\lambda t})\leq  \frac{\sqrt{\pi}}{\sqrt{\lambda}}
\end{align*}
where we used the change of variables $\theta := \sqrt{\lambda}(t-r)^{1/2}$ and the fact that $|\text{erf}(\cdot)| \leq 1$.
\end{proof}

\begin{lemma}
[A Gr\"onwall-type inequality]\label{Gronwalltypelemma}
Fix $s \geq 0$. Let $C \in \mathbb{B}([s,1],\mathbb{R})$ be a non-decreasing function and let 
 $f : [s,1] \times \mathbb{T} \to [0, \infty)$ be a bounded function. Suppose that there exists $\gamma \in (1,2]$ and $N_0 \geq 0$ such that for all  $t \in [s,1]$ and $x \in \mathbb{T}$ we have
$$
f(t,x) \leq C(t)+N_0\int_s^t \int_{\mathbb{T}} p_{t-r}^\gamma(x,y)f(r,y) dy dr.
$$
 There exists a constant $N=N(\gamma, N_0)$ such that for all $t \in [s,1]$ we have
$$\sup_{x \in \mathbb{T}}f(t,x) \leq  NC(t).$$
\end{lemma}
\begin{proof} Let $\lambda>0$ and consider the non-decreasing function of time $m: [s,1] \to \mathbb{R}$, that is given by
$$m_t := \sup_{s \leq r \leq t} \sup_{x \in \mathbb{T}}\big(f(r,x) e^{-\lambda r}\big).$$
Then 
$$f(t,x) \lesssim C(t) + \int_0^t \int_{\mathbb{T}} p_{t-r}^\gamma(x,y)m_r e^{\lambda r} dy dr,$$
where used the definition of $m$ and the fact that $[s,t] \subset [0,t]$.
Multiplying both sides by $e^{-\lambda t}$ and noting that $m_r \leq m_t$ for $r \leq t$ gives
$$f(t,x) e^{-\lambda t} \lesssim C(t) e^{-\lambda t} + m_t \int_0^t \int_{\mathbb{T}} p_{t-r}^\gamma(x,y) e^{-\lambda(t-r)} dy dr.$$
Let $T \in [s,1]$. Using \cref{gronwallintegralheatkernel} to estimate the second term, and taking supremum over 
$(t,x) \in [s,T]\times\mathbb{T}$ 
we get
$$m_T \lesssim C(T) + \frac{m_T}{\sqrt{\lambda}}.$$
Choosing $\lambda$ to be sufficiently large, we get that $m_T \lesssim C(T)$.
Since, $T \in [s,1]$ was arbitrary, the result follows by the definition of $m$.
\end{proof}

\begin{lemma}[A commonly used corollary of H\"older's inequality]\label{holdertrick}
Let $\gamma \in (1,3)$, $\delta \in (0,3)$. There exists
\begin{equs}\label{pconditionsforholder}
p>\frac{3-\delta}{3-\gamma}, \quad \text{such that} \quad \Big(\gamma - \frac{\delta}{p}\Big)\frac{p}{p-1} \geq 1,
\end{equs}
and a constant $N = N(\delta,\gamma,p)>0$, such that for all  $(s,t) \in [0,1]_{\leq}^2$ we have
\begin{equs}&\Big(\int_s^t\int_{\mathbb{T}} |p_{t-r}(x,y)|^\gamma f(r,y) dy dr\Big)^{p} \\
&\qquad\leq N (t-s)^{\frac{(3-\gamma)p}{2} + \frac{\delta -3}{2}}\int_s^t \int_{\mathbb{T}} |p_{t-r}(x,y)|^\delta f^p(r,y) dy dr. \label{holdercorollinequ}
\end{equs}
\end{lemma}
\begin{proof}
Note that for any $\gamma \in (1,3)$, $\delta \in (0,3)$ we have
$\lim_{p \to \infty}\left(\gamma - \frac{\delta}{p}\right)\frac{p}{p-1} = \gamma > 1,$
and thus it follows that for sufficiently large $p$ the conditions (\ref{pconditionsforholder}) hold.
By H\"older's inequality, the left-hand-side of (\ref{holdercorollinequ}) is bounded by
\begin{align*}
\Big(\int_s^t \int_{\mathbb{T}} |p_{t-r}(x,y)|^{(\gamma - \frac{\delta}{p})\frac{p}{p-1}} dy dr\Big)^{\frac{p-1}{p}\cdot p} \int_s^t \int_{\mathbb{T}}| p_{t-r}(x,y)|^\delta f^p(r,y) dy dr.
\end{align*}
Moreover using the results $\|p_t\|_{\mathbb{B}(\mathbb{T})} \lesssim t^{-1/2}$  and $\|p_t\|_{L_1(\mathbb{T})} =1$
to interpolate, we can see that
the first factor is bounded by
\begin{equs}\Big(\int_s^t (t-r)^{-\frac{1}{2}\Big((\gamma - \frac{\delta}{p})\frac{p}{p-1} -1\Big)} dr\Big)^{p-1} &\lesssim (t-s)^{\Big(-\frac{1}{2}\big((\gamma - \frac{\delta}{p})\frac{p}{p-1} -1\big) + 1\Big)(p-1)}= (t-s)^{\frac{(3-\gamma)p}{2} + \frac{\delta -3}{2}},\end{equs}
and thus the proof is finished.
\end{proof}

\begin{lemma}[Conditional BDG inequality for stochastic convolutions]\label{conditionalBDGforstochintegral}
Let $0\leq s \leq t$, $n \in \mathbb{Z}_{\geq 0}$ and let $X:\Omega\ \times [0,1]\times \mathbb{T} \to H^{\otimes n}$ be a $\mathscr{P}\otimes \mathscr{B}(\mathbb{R})$-measurable $H^{\otimes n}$-valued random field.  For all $p \in [2,\infty)$ there exists a constant $C_p$ such that if $f_t \in L_2([0,t]\times \mathbb{T})$ for all $t \in [0,1]$, then for all $(s,t) \in [0,1]_{\leq}^2$ we have
\begin{equs}&\mathbb{E}^s \Big\| \int_s^t \int_{\mathbb{T}} f_t(r,y) X(r,y)\xi(dy,dr)\Big\|_{H^{\otimes n}}^p\\
&\qquad\leq C_p \mathbb{E}^s \Big(\int_s^t \int_{\mathbb{T}} f_t^2(r,y) \|X(r,y) \|_{H^{\otimes n}}^2 dydr \Big)^{p/2} \label{bdgsharp},
\end{equs}
and consequentially 
\begin{equs}&\Big\|\Big\|\int_s^t \int_{\mathbb{T}} f_t(r,y)X(r,y) \xi(dy,dr)\Big\|_{H^{\otimes n}}\Big\|_{L_p|\mathscr{F}_s}^2\\
&\qquad\qquad\leq C_p \int_s^t \int_{\mathbb{T}} f_t^2(r,y) \|\|X(r,y)\|_{H^{\otimes n}}\|_{L_p|\mathscr{F}_s}^2dydr. \label{bdgused}
\end{equs}
\end{lemma}

The inequality (\ref{bdgsharp}) follows easily from the classic conditional BDG inequality. From (\ref{bdgsharp}) we can see that (\ref{bdgused}) holds by the Minkowski inequality

\begin{lemma}\label{4pointcompar}
Suppose that $f: \mathbb{R} \to \mathbb{R}$ is twice differentiable. Then for $\phi_1,\dots, \phi_4 \in \mathbb{R}$ we have
\begin{equs}
&f(\phi_1) - f(\phi_2) - f(\phi_3) + f(\phi_4) \\
&= \int_0^1 \int_0^1 (\phi_1- \phi_2)(\theta(\phi_1 - \phi_3) + (1-\theta)(\phi_2 - \phi_4))\nabla^2 f(\Theta_1(\theta,\eta))d\eta d \theta\\
&\quad+ (\phi_1 - \phi_2 - \phi_3 + \phi_4)\int_0^1 \nabla f(\Theta_2(\theta))d\theta \label{4pointidentity}
\end{equs}
where $\Theta_1(\theta,\eta)$ and $\Theta_2(\theta)$ are the convex combinations of $\phi_1,\dots, \phi_4$ given by
\begin{equs}
\Theta_1(\theta,\eta) &:= \eta(\theta \phi_1 + (1-\theta)\phi_2) + (1-\eta)(\theta \phi_3 + (1-\theta) \phi_4),\\
\Theta_2(\theta) &:= \theta \phi_3 + (1-\theta) \phi_4.
\end{equs}
Moreover
\begin{equs}&|f(\phi_1) - f(\phi_2) - f(\phi_3) + f(\phi_4)|\\
&\qquad\leq \|f\|_{C^1}|\phi_1 - \phi_2||\phi_1 - \phi_3| + \|f\|_{C^2}|\phi_1 - \phi_2 - \phi_3 + \phi_4|. \label{4pointinequality}
\end{equs}
\end{lemma}
\begin{proof}
We begin by proving (\ref{4pointidentity}). Using the notation 
\begin{equs}
\phi_{1,2}^\theta := \theta\phi_1 + (1-\theta)\phi_2, \qquad \phi_{3,4}^\theta := \theta \phi_3 + (1-\theta)\phi_4,
\end{equs}the expression $f(\phi_1) - f(\phi_2) - f(\phi_3) + f(\phi_4)$ can be rewritten as
\begin{equs}
&(\phi_1 - \phi_2)\int_0^1 \nabla f(\phi_{1,2}^\theta) d\theta - (\phi_3 - \phi_4)\int_0^1 \nabla f(\phi_{3,4}^\theta) d\theta \\
&= (\phi_1 - \phi_2)\int_0^1\Big(\nabla f(\phi_{1,2}^\theta)  -\nabla f(\phi_{3,4}^\theta)\Big)d\theta + (\phi_1 - \phi_2 - \phi_3 + \phi_4) \int_0^1 \nabla f(\phi_{3,4}^\theta) d \theta.
\end{equs}
The second term is exactly as desired, and the first term can be written as
$$(\phi_1 - \phi_2)\int_0^1(\phi_{1,2}^\theta - \phi_{3,4}^\theta)\int_0^1 \nabla^2 f(\eta \phi_{1,2}^\theta + (1-\eta)\phi_{3,4}^\theta) d\eta d\theta$$
which is indeed the first term of the desired expression. Hence (\ref{4pointidentity}) is proven. To prove (\ref{4pointinequality}), we set
$$\delta_{i,j} := \phi_j - \phi_i$$
and we note that $f(\phi_1) - f(\phi_2) - f(\phi_3) + f(\phi_4)$ can be written as
\begin{equs}
&f(\phi_1)- f(\phi_1 + \delta_{1,2}) - f(\phi_3) + f(\phi_3 + \delta_{1,2}) - f(\phi_{3} + \delta_{1,2}) + f(\phi_3 + \delta_{3,4})\\
&= -\delta_{1,2}\int_0^1 \nabla f(\phi_1 + \theta \delta_{1,2}) d\theta + \delta_{1,2}\int_0^1 \nabla f(\phi_3 + \theta \delta_{1,2}) d\theta\\
&\qquad \qquad+(\delta_{3,4} - \delta_{1,2}) \int_0^1 \nabla f(\phi_3 + \theta \delta_{3,4} + (1-\theta)\delta_{1,2}) d \theta\\
& =  \delta_{1,2}\delta_{1,3}\int_0^1 \int_0^1 \nabla^2 f(\eta\phi_3 + (1-\eta)\phi_1 +\theta \delta_{1,2}) d\theta d \eta\\
&\qquad\qquad+ (\delta_{3,4} - \delta_{1,2})\int_0^1 \nabla f(\phi_3 + \theta \delta_{3,4} + (1-\theta)\delta_{1,2})d\theta.
\end{equs}
Hence (\ref{4pointinequality}) follows as well.
\end{proof}

The following result is taken from \cite{dareiotis2022quantifying}.
\begin{proposition}
\label{laplacianlemma}
Let $\gamma \in \mathbb{R} \setminus \mathbb{Z}$. There exists a constant $N = N(\gamma)$ such that for all $f \in C^{\gamma}(\mathbb{R}^d)$ we have
$$\|(1-\Delta)^{-1} f \|_{C^{\gamma + 2}(\mathbb{R}^d)} \leq N \|f\|_{C^\gamma(\mathbb{R}^d)}.$$
\end{proposition}

\begin{lemma}[The $\mathscr{V}_p^{\gamma}$-bracket is triangular in time]\label{Vbracketistriangular}
Let $p\in [1,\infty)$, $\gamma>0$, and let $f \in \mathscr{V}_p^\gamma$. Then for all $0\leq S \leq Q \leq T \leq 1$ we have
\begin{equation}\label{triangleforVbracket}
[f]_{\mathscr{V}_p^\gamma[S,T]} \leq  2[f]_{\mathscr{V}_p^\gamma[S,Q]} + 2[f]_{\mathscr{V}_p^\gamma[Q,T]}.
\end{equation}
Consequently, for any integer $K \geq 2$ we have
\begin{equation}\label{Vbracketrepeatedtriangle}[f]_{\mathscr{V}_p^\gamma} \lesssim 2^K\sum_{i=0}^{K-1} [f]_{\mathscr{V}_p^{\gamma}[\frac{i}{K},\frac{i+1}{K}]}.
\end{equation}
\end{lemma}
\begin{proof}
For $(s,t) \in [0,1]_{\leq}^2$ define
$$A(s,t) := \sup_{x \in \mathbb{T}}\|f_t(x) - P_{t-s}f_s(x)\|_{L_{p, \infty}^{\mathscr{F}_s}}.$$
For $(s,t) \in [S,Q]_{\leq}^2 \cup [Q,T]_{\leq}^2$, we clearly have
\begin{equs}
A(s,t) \leq [f]_{\mathscr{V}_p^\gamma[s,t]}(t-s)^\gamma \leq \Big( [f]_{\mathscr{V}_p^\gamma[S,Q]} + [f]_{\mathscr{V}_p^\gamma[Q,T]}\Big)(t-s)^\gamma. \label{stisbeloworaboveQ}
\end{equs}
We also need to check what happens in the case when $Q \in (s,t)$. Then we write
\begin{equs}
A(s,t) &\leq \sup_{x \in \mathbb{T}}\|f_t(x) - P_{t-Q}f_Q(x)\|_{L_{p,\infty}^{\mathscr{F}_s}} + \sup_{x \in \mathbb{T}}\|P_{t-Q}f_Q(x) - P_{t-s}f_s(x)\|_{L_{p,\infty}^{\mathscr{F}_s}}\\
&= B(s,t) + C(s,t).
\end{equs}
Note that as $s \leq Q$, we have
$\|\cdot\|_{L_{p}|\mathscr{F}_s} \leq \|\|\cdot\|_{L_p |\mathscr{F}_Q}\|_{L_p|\mathscr{F}_s} \leq \|\cdot\|_{L_{p,\infty}^Q}$, and thus
\begin{equs}
B(s,t) \leq \sup_{x \in \mathbb{T}}\|f_t(x) - P_{t-Q}f_Q(x)\|_{L_{p,\infty}^Q} = A(Q,t).
\end{equs}
Moreover using that $s \leq Q$, we have
\begin{equs}
C(s,t) &= \sup_{x \in \mathbb{T}}\|P_{t-Q}\big(f_Q - P_{Q-s}f_s \big)(x)\|_{L_{p,\infty}^{\mathscr{F}_s}} \leq \sup_{x \in \mathbb{T}}\|f_Q(x) - P_{Q-s}f_s (x)\|_{L_{p,\infty}^{\mathscr{F}_s}} = A(s,Q)  .
\end{equs}
By the above bounds on $B$ and $C$, we conclude that
\begin{equs}
A(s,t) \leq A(s,Q) + A(Q,t) \leq \Big( [f]_{\mathscr{V}_p^\gamma[S,Q]} + [f]_{\mathscr{V}_p^\gamma[Q,T]}\Big)(t-s)^\gamma \label{Qisbetweensandt}
\end{equs}
By adding up the bounds (\ref{stisbeloworaboveQ}) and (\ref{Qisbetweensandt}), we can see that for all $(s,t) \in [S,T]_{\leq}^2$ and $Q \in [S,T]$, we have
\begin{equs}
A(s,t) \leq 2\Big( [f]_{\mathscr{V}_p^\gamma[S,Q]} + [f]_{\mathscr{V}_p^\gamma[Q,T]}\Big)(t-s)^\gamma,
\end{equs}
from which the desired result follows.
\end{proof}

\begin{lemma}[The $L_p$-valued $C^{1/4, 1/2}$-norm, and the $\mathscr{V}_p^{1/4}$-bracket]\label{embeddinglemma}
Let $\alpha \in (-1,0)$ and $p \in [1,\infty)$. There exists a constant $N = N(p, \alpha)>0$ such that for $f\in \mathscr{V}_p^{1/4}\cap C^{0,1/2}([0,1]\times \mathbb{T}, L_p)$ we have
\begin{equs}
\|f\|_{C^{1/4, 1/2}([0,1]\times \mathbb{T}, L_p)} \leq   N[f]_{\mathscr{V}_p^{1/4}} + N\|f\|_{C^{0,1/2}([0,1]\times \mathbb{T}, L_p)}.
\end{equs}
\end{lemma}
\begin{proof}
We decompose the space--time H\"older norm to the sup norm, and spatial and temporal seminorms as follows:
\begin{equs}
    &\|f^n\|_{C^{1/4,1/2}([0,1]\times \mathbb{T}, L_p)} \\
    &\leq \|f^n\|_{\mathbb{B}([0,1]\times \mathbb{T}, L_p)} + \sup_{x \in \mathbb{T}}[f^n(\cdot, x)]_{C^{1/4}([0,1])} + \sup_{t \in [0,1]}[f^n(t,\cdot)]_{C^{1/2}(\mathbb{T}, L_p)}.\label{decomptoseminorms}
\end{equs}
To bound the temporal seminorm, note that for $(s,t) \in [0,1]_{\leq}^2$ we have
\begin{equs}
\|f(t,\cdot) -f(s,\cdot)\|_{\mathbb{B}(\mathbb{T}, L_p)} &\leq \|f(t,\cdot) - P_{t-s}f(s,\cdot)\|_{\mathbb{B}(\mathbb{T}, L_p)} + \|P_{t-s}f(s,\cdot) - f(s,\cdot)\|_{\mathbb{B}(\mathbb{T}, L_p)}\\
&=: A(s,t) + B(s,t). \label{decomptoAandB}
\end{equs}
Since $\|\cdot\|_{L_p} \leq \|\cdot\|_{L_{p,\infty}^{\mathscr{F}_s}}$, it follows that
\begin{equs}
A(s,t) \leq [f]_{\mathscr{V}_p^{1/4}[s,t]}(t-s)^{1/4}.
\end{equs}
Moreover by a standard heat kernel estimate
\begin{equs}
B(s,t) &\lesssim \|f(s,\cdot)\|_{C^{1/2}(\mathbb{T}, L_p)}(t-s)^{1/4}.
\end{equs}
By putting the above bounds on $A$ and $B$ into (\ref{decomptoAandB}), we can see that
\begin{equs}
\sup_{x \in \mathbb{T}}[f(\cdot,x)]_{C^{1/4}([0,1], L_p)} \lesssim [f]_{\mathscr{V}_p^{1/4}[s,t]} + \|f\|_{C^{0,1/2}([0,1]\times\mathbb{T}, L_p)}.
\end{equs}
Using this bound on the second term of (\ref{decomptoseminorms}) finishes the proof.
\end{proof}

In the next lemma, we show a Markov-type property which will be used often. It is very standard but the proof is included for the convenience of the reader. Recall that $C(\mathbb{T})$ denotes the collection of continuous functions $f:\mathbb{T} \to \mathbb{R}$, and it is equipped with the sup-norm $\|\cdot\|_{\mathbb{B}}$. The topology induced by this norm generates the Borel $\sigma$-algebra $\mathscr{B}(C(\mathbb{T}))$ which coincides with the cylindrical $\sigma$-algebra. Moreover, recall that since $C(\mathbb{T})$ is separable,  the notions of measurable, weakly measurable, and strongly measurable $C(\mathbb{T})$-valued maps on $\Omega$ coincide. In addition, a continuous random field $u : \Omega \times \mathbb{T} \to \R$ is actually a $C(\mathbb{T})$-valued random variable. 

\begin{lemma}\label{markovproperty}      Let  $b,\sigma \in C^{1}(\mathbb{R})$, $M \in \mathbb{N}$, $(Z_i)_{i=1}^M \subset L_2(\Omega, \mathscr{F}_s, \mathbb{P}; C(\mathbb{T}))\cap \mathbb{B}(\mathbb{T}, L_2(\Omega))$, and let $\phi^{Z,s}$ be the unique solution of (\ref{phidef0}).
Further, for $p \in [1,\infty)$, $f \in C^1(\mathbb{R}^M)$, $t \in [s,1]$, and $x \in \mathbb{T}$, define $g:(C(\mathbb{T}))^M \to \mathbb{R}$ by
$$
g(z_1,\dots,z_M) := \mathbb{E}f\big(\phi^{z_1}(t-s,x),\dots, \phi^{z_M}(t-s,x) \big).
$$
Then, for $i=1,\dots,M$ and $Z_i \in L_2(\Omega, \mathcal{F}_s, \mathbb{P}; C(\mathbb{T}))\cap\mathbb{B}(\mathbb{T}, L_2(\Omega))$, we have 
\begin{equs}        \label{eq:Markov_conclusion}
\mathbb{E}^s f\big(\phi^{Z_1,s}(t,x),\dots, \phi^{Z_M,s}(t,x)\big) = g(Z_1,\dots, Z_M).
\end{equs}
\end{lemma}

\begin{proof}
Suppose first that the $Z_i$ are simple random variables of the form 
\begin{equs}   \label{eq:simple_RV}
Z_i = \sum_{k=1}^K h_{k,i} \mathbf{1}_{E_k}
\end{equs}
where $K \in \mathbb{N}$,  $(h_{k,i})_{k=1}^K \subset C(\mathbb{T})$ and $(E_k)_{k=1}^K \subset  \mathcal{F}_s$ is a partition of $\Omega$. In this case, we have for $(t,x) \in [s,1]\times \mathbb{T}$ that
\begin{equs}
\mathbb{E}^s&f(\phi^{Z_1,s}(t,x), \dots, \phi_{s,t}^{Z_M,s}(t,x))\\
&= \mathbb{E}^sf\Big(\sum_{k=1}^K \mathbf{1}_{E_k}\phi^{h_{k,1},s}(t,x), \dots, \sum_{k=1}^K \mathbf{1}_{E_k}\phi^{h_{k,M},s}(t,x)\Big)\\
&= \mathbb{E}^s 
\sum_{k=1}^K \mathbf{1}_{E_k}f(\phi^{h_{k,1},s}(t,x),\dots, \phi^{h_{k,M},s}(t,x))\\
&= \sum_{k=1}^K \mathbf{1}_{E_k}\mathbb{E}^s f(\phi^{h_{k,1},s}(t,x),\dots, \phi^{h_{k,M},s}(t,x))\\
&= \sum_{k=1}^K \mathbf{1}_{E_k}\mathbb{E}f(\phi^{h_{k,1}}(t-s,x),\dots, \phi^{h_{k,M}}(t-s,x))\\
& = \sum_{k=1}^K \mathbf{1}_{E_k}g(h_{k,1},\dots,h_{k,M})= g(Z_1,\dots,Z_M), 
\end{equs}
which shows \eqref{eq:Markov_conclusion}. For the general case, since $Z_i \in L_2(\Omega, \mathcal{F}_s, \mathbb{P}; C(\mathbb{T}))$, for $i=1, \dots, M$ there exist sequences $(Z^n_i)_{n \in \mathbb{N}}$ of the form \eqref{eq:simple_RV} such that $\| Z^n_i-Z_i\|_{\mathbb{B}(\mathbb{T})} \to 0 $ almost surely and in $L_2(\Omega)$ as $n \to \infty$. For those $Z^n_is$ and for $(t,x) \in [s,1]\times \mathbb{T}$ we have 
\begin{equs}        \label{eq:Markov_conclusion_approx}
\mathbb{E}^s f(\phi^{Z^n_1,s}(t,x),\dots \phi^{Z^n_M,s}(t,x)) = g(Z^n_1,\dots, Z^n_M).
\end{equs}
It follows from \cref{lipschitzflowgeneral} that for all $(t,x) \in [s,1]\times \mathbb{T}$, the map $ L_p(\Omega; C(\mathbb{T})) \ni Z \mapsto \phi^{Z,s}(t,x) \in L_p(\Omega) $ is Lipschitz. From this, it firstly  follows that $\phi^{Z^n_i,s}(t,x) \to \phi^{Z_i,s}(t,x)$ in $L_p(\Omega)$, which by using the Lipschitz continuity of $f$ implies that $$\mathbb{E}^s f(\phi^{Z^n_1,s}(t,x),\dots, \phi^{Z^n_M,s}(t,x)) \longrightarrow \mathbb{E}^s f(\phi^{Z_1,s}(t,x),\dots \phi^{Z_M,s}(t,x))$$ in $L_p(\Omega)$.  Secondly, it also follows that the function $g:C(\mathbb{T})^M \to \R $ is continuous. Hence, upon taking the limit in probability with $n\to \infty$ in \eqref{eq:Markov_conclusion_approx}, the result follows. 
\end{proof}

\bibliographystyle{Martin}
\bibliography{references}
\end{document}